\documentclass{article}
\usepackage[a4paper, total={7in, 10in}]{geometry}
\usepackage[utf8]{inputenc} 
\usepackage[T1]{fontenc}    
\usepackage{hyperref}       
\usepackage{url}            
\usepackage{booktabs}       
\usepackage{amsfonts}       
\usepackage{nicefrac}       
\usepackage{microtype}      
\usepackage{float}
\usepackage[english]{babel}
\usepackage{mathtools}
\usepackage{verbatim}
\usepackage{hyperref}
\usepackage{bm}
\usepackage{tabu}
\usepackage{enumitem}
\usepackage{mathrsfs}
\usepackage{enumitem}
\usepackage{amsfonts}
\usepackage{graphicx}
\usepackage{parskip}
\usepackage{color}
\usepackage{tikz}
\usepackage[]{algorithm2e}
\usepackage{amssymb}
\usepackage[numbers]{natbib}
\usepackage{bbm}
\newtheorem{theorem}{Theorem}
\newtheorem{lemma}{Lemma}
\newtheorem{remark}{Remark}
\newtheorem{assumption}{Assumption}

\newtheorem{corollary}[lemma]{Corollary}

\newenvironment{proof}[1][Proof]{\begin{trivlist}
\item[\hskip \labelsep {\bfseries #1}]}{\qed\end{trivlist}}

\newcommand*\at[2]{\left.#1\right|_{#2}}

\newcommand*\del[0]{\partial}

\newcommand*\ddt[0]{\frac{d}{d t}}

\newcommand*\tr[0]{\text{tr}}
\newcommand*\KL[2]{\mathcal{KL}\left(#1\|#2\right)}
\newcommand*\lin[1]{\bm{\left\langle} #1 \bm{\right\rangle}}
\newcommand*\lint[1]{\bm{\left\langle} #1 \bm{\right\rangle}_{\tiny{\circled{x}}}}
\newcommand*\E[1]{\mathbb{E}\left[#1\right]}
\newcommand*\Ep[2]{\mathbb{E}_{#1}\left[#2\right]}

\newcommand\numberthis{\addtocounter{equation}{1}\tag{\theequation}}
\newcommand*\lrb[1]{\left[#1\right]}
\newcommand*\lrbb[1]{\left\{#1\right\}}
\newcommand*\lrp[1]{\left(#1\right)}
\newcommand*\lrn[1]{\left\|#1\right\|}
\newcommand*\lrabs[1]{\left|#1\right|}

\newcommand*\ind[1]{{\mathbbm{1}\lrbb{#1}}}

\newcommand*{\qed}{\hfill\ensuremath{\blacksquare}}
\renewcommand*{\Re}{\mathbb{R}}
\renewcommand*{\P}{\mathbf{P}}

\newcommand*{\M}[1]{M_{#1}}

\newcommand*{\dd}[1]{\frac{\del}{\del #1}}
\newcommand*\circled[1]{\tikz[baseline=(char.base)]{
\node[shape=circle,draw,inner sep=1pt] (char) {#1};}}

\title{Quantitative Weak Convergence for Discrete Stochastic Processes}

\author{Xiang Cheng  \thanks{x.cheng@berkeley.edu; Computer Science Division, UC Berkeley; work performed while at Adobe Research.} \and Peter L. Bartlett \thanks{peter@berkeley.edu; Computer Science Division \& Department of Statistics, UC Berkeley.} \and Michael I. Jordan \thanks{jordan@cs.berkeley.edu; Computer Science Division \& Department of Statistics, UC Berkeley.}}

\begin{document}
\maketitle

\abstract{In this paper, we quantitative convergence in $W_2$ for a family of Langevin-like stochastic processes that includes stochastic gradient descent and related gradient-based algorithms. Under certain regularity assumptions, we show that the
iterates of these stochastic processes converge to an invariant
distribution at a rate of $\tilde{O}\lrp{1/\sqrt{k}}$ where $k$ is the
number of steps; this rate is provably tight up to log factors. Our result reduces to a quantitative form of the classical Central Limit Theorem
in the special case when the potential is quadratic.}

\begin{section}{Introduction}
Many randomized algorithms in machine learning can be analyzed
as some kind of stochastic process.
For example, MCMC algorithms intentionally inject carefully designed
randomness in order to sample from a desired target distribution. 
There is a second category of randomized algorithms for which the
for which the goal is optimization rather than sampling, and the randomness is viewed as a price to pay for computational
tractability. For example, stochastic gradient methods for large scale
optimization use noisy estimates of a gradient because they are cheap.
While such algorithms are not designed with the goal of sampling from a
target distribution, an algorithm of this kind has random
outputs, and its behavior is determined by the distribution of its
output. Results in this paper provide tools for analyzing the convergence of such algorithms as stochastic processes.

We establish a quantitative
Central Limit Theorem for stochastic processes that have the
following form:
\begin{align*}
x_{k+1} = x_k - \delta \nabla U(x_k) + \sqrt{\delta} \xi_k(x_k),
\numberthis\label{e:intro_discrete}
\end{align*}
where $x_k\in \Re^d$ is an iterate, $\delta$ is a stepsize,
$U:\Re^d\to\Re$ is a potential function, and $\xi(\cdot)$ is a zero-mean,
position-dependent noise variable. Under certain assumptions,
we show that \eqref{e:intro_discrete}
converges in $2$-Wasserstein distance to the following SDE:
\begin{align*}
d x(t) = -\nabla U(x(t)) dt + \sigma(x(t)) dB_t,
\numberthis \label{e:intro_exact}
\end{align*}
where $\sigma(x) = \lrp{\E{\xi(x) \xi(x)^T}}^{1/2}$. The notion of
convergence is summarized in the following informal statement of our
main theorem:
\begin{theorem}(Informal)
Let $p_k$ denote the distribution of $x_k$ in
\eqref{e:intro_discrete}, and let $p^*$ denote the invariant
distribution of \eqref{e:intro_exact}. Then there exist constants
$c_1,c_2$, such that for all $\epsilon>0$,
if $\delta \leq c_1 \epsilon^2/d^7$ and $k \ge c_2d^7/\epsilon^2$,
\begin{align*}
W_2(p_k,p^*) \leq \epsilon.
\end{align*}
\end{theorem}
In other words, under the right scaling of the step size, the long-term
distribution of $x_k$ depends only on the expected drift $\nabla U(x)$
and the covariance matrix of the noise $\sigma(x)$. As long
as we know these two quantities, we can draw conclusions
about the approximate behavior of \eqref{e:intro_discrete} through
$p^*$, and ignore the other characteristics of $\xi$.

Our result can be viewed as a general, quantitative form of the classical
Central Limit Theorem, which can be thought of as showing that
$x_k$ in \eqref{e:intro_discrete} converges in distribution to ${\cal
N}(0, I)$, for the specific case of $U(x) = \|x\|_2^2/2$ and $\sigma_x =
I$. Our result is more general: $U(x)$ can be
any strongly convex function satisfying certain regularity
assumptions and $\sigma_x$ can vary with position. We show that
$x_k$ converges to the invariant distribution of \eqref{e:intro_exact},
which is not necessarily a normal distribution.
The fact that the classical CLT is a special case implies that the
$\epsilon^{-2}$ rate in our main theorem cannot be improved in
general. We discuss this in more detail in Section \ref{sss:clt}.
\end{section}

\begin{section}{Related Work}
A crucial part of our technique in this paper is based on \cite{zhai2018high}, which established that for $iid$ random variables $x_i$ with mean zero and covariance $I$, $W_2\lrp{\frac{\sum_{i=1}^k x_i}{\sqrt{k}}, Z} = O\lrp{\frac{\beta \sqrt{d}\log (k)}{\sqrt{k}}}$, where $Z$ is the standard Gaussian random variable, and $\beta$ is a a.s. upper bound on $\|x_i\|_2$. \cite{zhai2018high} also proves a lower bound of $\Omega(\frac{\sqrt{d} \beta}{k})$, thus showing that (under their assumptions), the rate of $O\lrp{\frac{\beta \sqrt{d}\log (k)}{\sqrt{k}}}$ is tight up to log factors.

Prior to this, a number of other authors have proved an optimal $O(1/{\sqrt{k}})$ rate, but without establishing dimension dependence~\citep[see, e.g.,][]{bobkov2013entropic, bonis2015rates,rio2009upper}. Following \cite{zhai2018high}, \cite{eldan2018clt} improved the rate to $O(\frac{\beta \sqrt{d \log (k)}}{\sqrt{k}})$. Under a different set of assumptions, authors of \cite{courtade2019existence} established a $W_2$ CLT with a rate of $O(\frac{\sqrt{d C} }{\sqrt{k}})$, where $C$ is the Poincare constant. It is worth noting that the $\beta$ term in \cite{zhai2018high}, \cite{eldan2018clt} and in the results of this paper is typically on the order of $\sqrt{d}$, whereas the term $C$ in \cite{courtade2019existence} is typically dimension-free. On the other hand, the assumptions of \cite{zhai2018high} and \cite{eldan2018clt} are incompatible with \cite{courtade2019existence}. In \cite{courtade2018bounds}, under more general assumptions than \cite{zhai2018high, eldan2018clt} the author proved an optimal $\sqrt{d}$ dimensional dependence, but with suboptimal $k^{1/4}$ dependence.

Another relevant line of work is the recent work on quantitative
rates for Langevin MCMC algorithms. Langevin MCMC algorithms can be
thought of as discretizations of the Langevin diffusion SDE, which is
essentially \eqref{e:intro_exact} for $\sigma(x)=I$. Authors such as
\cite{dalalyan2017theoretical} and \cite{durmus2016high} were able to
prove quantitative convergence results for Langevin MCMC by bounding its
discretization error from the Langevin SDE. The processes we study in this paper differ from Langevin MCMC in two crucial ways: first, the noise $T_\eta(x)$ is not Gaussian, and second, the diffusion matrix in \eqref{e:exactsde} varies with $x$.

Finally, this work is also motivated by results such as those due to
\cite{ruppert1988efficient}, \cite{polyak1992acceleration},
\cite{fan2018statistical}, which show that iterates of the stochastic
gradient algorithm with diminishing step size converge asymptotically
to a normal distribution. (The limiting distribution of the
appropriately rescaled iterates is Gaussian in this case, because a
smooth $U$ is locally quadratic.) These classical
results are asymptotic and do not give explicit rates.

\end{section}

\begin{section}{Definitions and Assumptions}
\label{s:definitionsandassumptions}

We will study the discrete process given by
\begin{equation}
\label{e:discretesde}
x_{k+1} = x_{k} - \delta \nabla U(x_{k}) + \sqrt{2\delta }
T_{\eta_k}(x_{k}),
\end{equation}
where
\begin{enumerate}
\item $U(x): \Re^d \to \Re$ is the potential function,
\item $\eta_1,\eta_2,\ldots,\eta_k$ are iid random variables which take values in some set $\Omega$ and have distribution $q(\eta)$,
\item $T:\Omega\times\Re^d\to\Re$ is the noise map, and
\item $\delta>0$ is a stepsize.
\end{enumerate}
Let $\hat{p}(x)$ denote the invariant distribution of the markov chain in \eqref{e:discretesde}.
Define
\begin{align*}
\sigma_x := \lrp{\Ep{q(\eta)}{T_\eta(x)T_\eta(x)^T}}^{1/2}.
\numberthis \label{d:sigmax}
\end{align*}
We will also study the continuous SDE given by 
\begin{equation}
\label{e:exactsde}
d x(t) = - \nabla U(x(t)) dt + \sqrt{2} \sigma_{x(t)} dB_t,
\end{equation}
where $B_t$ denotes the standard $d$-dimensional Brownian motion, and
$\sigma_x: \Re^d \to \Re^{d\times d}$ is as defined in \eqref{d:sigmax}. Let
$p^*$ denote the invariant distribution of \eqref{e:exactsde}.

For convenience of notation, we define the following:
\begin{enumerate}
\item Let $p_k$ be the distribution of $x_k$ in \eqref{e:discretesde}.
\item Let $F: \Omega \times \Re^d \to \Re^d$ be the transition map:
\begin{align*}
\numberthis \label{d:Feta}
F_\eta(x) := x - \delta \nabla U(x) + \sqrt{2\delta} T_\eta(x),
\end{align*}
so that $x_{k+1} = F_{\eta_k}(x_{k})$. Note that $F_\eta(x)$ also depends on $\delta$, but we do not write this explicitly; the choice of $\delta$ should be clear from context. 
\item Define $\Phi_\delta$ as
\begin{align*}
\Phi_\delta(p) := \lrp{F_\eta}_{\#} p,
\numberthis \label{d:Phip}
\end{align*}
where $\#$ denotes the pushforward operator; i.e., $\Phi(p)$ is  the
distribution of $F_\eta(x)$ when $x\sim p$, so that $p_{k+1} =
\Phi_\delta(p_k)$
\end{enumerate}

We make the following assumptions about $U$.
\begin{assumption}
\label{ass:uissmooth}
There exist constants $m$ and $L$ satisfying, for all $x$, 
\begin{align*}
1.\ &\nabla U(0)=0, &
2.\ &\nabla^2 U(x) \preceq LI, &
3.\ &\nabla^2 U(x)\succeq mI, &
4.\ &\lrn{\nabla^3 U(x)}_2 \leq L,
\end{align*}
where $\|\cdot\|_2$ denotes the operator norm; see
\eqref{d:operatornorm} below.
\end{assumption}
We make the following assumptions about $q(\eta)$ and $T_\eta(x)$:
\begin{assumption}
\label{ass:ximeanandvariance}
There exists a constant $c_{\sigma}$, such that for all $x$, 
\begin{align*}
&1.\ \Ep{q(\eta)}{T_\eta(x)}=0, &
&2.\ \Ep{q(\eta)}{T_\eta(x)T_\eta(x)^T} \prec c_{\sigma}^2 I.
\end{align*}
\end{assumption}

\begin{subsection}{Basic Notation}
For any two distributions $p$ and $q$, let $W_2(p,q)$ be the 2-Wasserstein distance between $p$ and $q$. We overload the notation and sometimes use $W_2(x,y)$ for random variables $x$ and $y$ to denote the $W_2$ distance between their distributions.

For a $k^{th}$-order tensor $M\in \Re^{d^k}$ and a vector $v\in \Re^d$, we define the product
$A = Mv$ such that $\lrb{A}_{i_1...i_{k-1}} = \sum_{j=1}^d \lrb{M}_{i_1...i_{k-1}, j} \cdot v_j$.
Sometimes, to avoid ambiguity, we will write $A = \lint{Mv}$
instead.

We let $\|\|_2$ denote the operator norm:
\begin{align*}
\numberthis
\label{d:operatornorm}
\lrn{M}_2 = \sup_{v\in \Re^d, \|v\|_2 =1} \lrn{Mv}_2.
\end{align*}
It can be verified that for all $k$, $\|\cdot\|_2$ is a norm over $\Re^{kd}$.

Finally, we use the notation $\lin{}$ to denote two kinds of inner products:
\begin{enumerate}
\item For vectors $u, v\in \Re^d$,
$\lin{u,v}=\sum_{i=1}^d u_i v_i$ (the dot product).
\item For matrices $A, B \in \Re^{2d}$,
$\lin{A,B} := \sum_{i=1}^d \sum_{j=1}^d A_{i,j} B_{j,i}$
(the trace inner product).
\end{enumerate}
Although the notation is overloaded, the usage should be clear from context.

\end{subsection}
\end{section}

\begin{section}{Main Results and Discussion}
We will consider two settings: one in which the noise $T_\eta$ in \eqref{e:discretesde} does not depend on $x$, and one in which it does. We will treat these results separately in Theorem \ref{t:s_convergencerate} and Theorem \ref{t:convergencerate}.

\begin{subsection}{Homogeneous Noise}
\label{ss:mainresult_homogeneousnoise}
Our first theorem deals with the case when $T_\eta$ is a constant with respect to $x$. In addition to Assumption \ref{ass:uissmooth} and Assumption \ref{ass:ximeanandvariance}, we make the following assumptions:
\begin{assumption}\label{ass:simplerassumptions}
For all $x$, 
\begin{align*}
&1.\ T_\eta(x) = T_\eta, &
&2.\ \lrn{T_\eta}_2\leq \sqrt{L}, &
&3.\ \sigma_x = I.
\end{align*}
\end{assumption}
Under these assumptions, the invariant distribution $p^*(x)$ of \eqref{e:exactsde} has the form
\begin{align*}
\numberthis \label{e:simplep*}
p^*(x) \propto e^{-U(x)}.
\end{align*}
\begin{theorem}\label{t:s_convergencerate}
Let $p_0$ be an arbitrary initial distribution, and let $p_{k}$ be defined as in \eqref{e:discretesde} with step size $\delta$. Recall the definition of $\hat{p}$ as the invariant distribution of \eqref{e:discretesde} and $p^*$ as the invariant distribution of \eqref{e:exactsde}.

For $\delta \leq \frac{\epsilon^2}{d^3}\cdot poly
\lrp{\frac{1}{m},L}^{-1}$,
\begin{align*}
W_2(\hat{p},p^*) \leq \epsilon \numberthis\label{e:t1:1}.
\end{align*}
If, in addition, $k \geq \frac{d^3}{\epsilon^2}\log
\frac{W_2(p_0,p^*)}{\epsilon} \cdot poly \lrp{\frac{1}{m},L}$,
\begin{align*}
W_2\lrp{p_{k}, p^*}\leq {\epsilon} \numberthis\label{e:t1:2}.
\end{align*}
\end{theorem}
This implies that for any $k$, there exists a sufficiently small $\delta$ (whose value depends on $k$), such that
\begin{align*}
W_2(p_k,p^*) \leq \tilde{O}\lrp{\frac{1}{\sqrt{k}}}.
\numberthis \label{e:1/rootk}
\end{align*}

\begin{remark}
The dimensional dependence of $d^{3}$ in the expression for $k$ contains a caveat: additional dimensional dependence may enter through the variable $L$. In particular, the assumption that $\lrn{T_\eta}_2 \leq \sqrt{L}$ would imply that $L$ is on the order of $d$, so the actual dimension dependence can be much larger than $d^{3}$.
\end{remark}

\begin{subsubsection} {Relation to Central Limit Theorem}
\label{sss:clt}
Our result can be viewed as a generalization of the classical central limit theorem, which deals with sequences of the form
\begin{align*}
S_{k+1} = \frac{\sum_{i=0}^{k+1} \eta_i}{\sqrt{k+1}}
= \frac{\sqrt{k}}{\sqrt{k+1}} \cdot S_{k} + \frac{\eta_{k+1}}{\sqrt{k+1}}
\approx S_k - \frac{1}{2(k+1)} S_k + \frac{\sqrt{2}}{\sqrt{2(k+1})} \eta_{k+1}
\end{align*}
for some $\eta_k$ with mean $0$ and covariance $I$. Thus, the sequence $S_k$ essentially has the same dynamics as $x_k$ from \eqref{e:discretesde}, with $U(x) = - \frac{1}{2}\|x\|_2^2$,  $T_{\eta_k} = \eta_k$ and variable stepsize $\delta_k = \frac{1}{{k}}$. Assuming $\|\eta_i\|_2\leq \beta$ almost surely, the fastest convergence result is is proven in Theorem~1.1 of \cite{eldan2018clt}, with a rate of $W_2\lrp{S_{k}, Z} \leq O\lrp{\sqrt{d\log(k)}\beta /\sqrt{k}}$. It is also essentially tight, as Proposition~1.2 of \cite{zhai2018high} shows that the $W_2\lrp{S_k, Z}$ is lower bounded by $\Omega\lrp{\sqrt{d}\beta/\sqrt{k}}$ in certain cases.

Our bound in Theorem~\ref{t:s_convergencerate} (equivalently,
\eqref{e:1/rootk}) also shrinks as $1/\sqrt{k}$. We note that the
sequence $x_k$ studied in Theorem 2 differs from $S_k$, as the
stepsize for $x_k$ is constant (i.e., $\delta$ does not depend on $k$).
We stated Theorem~\ref{t:s_convergencerate} for constant step sizes
mainly to simplify the proof. Our proof technique can also be applied
to the variable step size setting; in Appendix \ref{ss:clt}, we show demonstrate how one might obtain a (suboptimal) CLT convergence rate of
$W_2\lrp{S_{k}, Z} \leq \tilde{O}\lrp{1/\sqrt{k}}$ using a similar technique as Theorem \ref{t:s_convergencerate}, but with stepsize $\delta_i = 1/(2i+1)$. This also implies that the $1/\sqrt{k}$ rate in Theorem~\ref{t:s_convergencerate}
is tight. On the other hand, our $d$ dependence is far from the optimal rate of $\sqrt{d}$. However, our bound is applicable
to a more general setting, not just for $U(x) = 1/2 \|x\|_2^2$.
\end{subsubsection}
\end{subsection}
\begin{subsection}{Inhomogeneous Noise}
\label{ss:mainresult_inhomogeneousnoise}
We now examine the convergence of \eqref{e:discretesde} under a general setting, in which the noise $T_\eta(x)$ depends on the position.

In addition to the assumptions in
Section~\ref{s:definitionsandassumptions}, we make some additional
assumptions about how $T_\eta(x)$ depends on $x$.
We begin by defining some notation. For all $x\in \Re^d$ and $\eta \in
\Omega$, we will let $G_\eta(x)\in \Re^{2d}$ denote the derivative of
$T_\eta(x)$ wrt $x$, $M_\eta(x)\in \Re^{3d}$ denote the derivative
of $G_\eta(x)$ wrt $x$, and $N_\eta(x) \in \Re^{4d}$ denote the
derivative of $M_\eta(x)$ wrt $x$, i.e.:
\begin{align*}
&1.\ \forall x,i,j \text{ and for $\eta$ a.s., } \lrb{G_\eta(x)}_{i,j} := \dd{x_j} \lrb{T_\eta(x)}_i\\
&2.\ \forall x,i,j,k \text{ and for $\eta$ a.s., } \lrb{M_\eta(x)}_{i,j,k} := \frac{\del^2}{\del x_j \del x_k} \lrb{T_\eta(x)}_i\\
&3.\ \forall x,i,j,k,l \text{ and for $\eta$ a.s., } \lrb{N_\eta(x)}_{i,j,k,l} := \frac{\del^3}{\del x_j \del x_k \del x_l} \lrb{T_\eta(x)}_i
\end{align*}
We will assume that $T_\eta(x)$, $G_\eta(x)$, $M_\eta(x)$ satisfy the following regularity: 
\begin{assumption}
\label{ass:gisregular}
There exists an $L$ that satisfies Assumption \ref{ass:uissmooth} and,
for all $x$ and for $\eta$ a.s.:
\begin{align*}
&1.\ G_\eta(x) \text{ is symmetric}, &
&2.\ \lrn{T_\eta(x)}_2 \leq  \sqrt{L} (\|x\|_2 + 1),  &
&3.\ \lrn{{G_\eta(x)}}_2\leq \sqrt{L},\\
&4.\ \lrn{{M_\eta(x)}}_2\leq \sqrt{L}, &
&5.\ \lrn{N_\eta(x)}_2 \leq \sqrt{L}.
\end{align*}
\end{assumption}
\begin{assumption}
\label{ass:discreteprocesscontracts}
For any distributions $p$ and $q$,
$
W_2(\Phi_\delta(p),\Phi_\delta(q)) \leq e^{-\lambda \delta} W_2(p,q)
$.
\end{assumption}

Finally, we assume that $\log p^*(x)$ is regular in the following sense:
\begin{assumption}
\label{ass:p^*isregular}
There exists a constant $\theta$, such that the log of the invariant
distribution of~\eqref{e:exactsde},
$f(x) := \log \lrp{p^*(x)}$, satisfies, for all $x$,
\begin{align*}
&1.\ \lrn{\nabla^3 f(x)}_2 \leq \theta, &
&2.\ \lrn{\nabla^2 f(x)}_2 \leq \theta\lrp{\|x\|_2 + 1}, &
&3.\ \lrn{\nabla f(x)}_2 \leq \theta\lrp{\|x\|_2^2 + 1}.
\end{align*}
\end{assumption}
\begin{remark}
If $\nabla^2 f(0)$ and $\nabla f(0)$ are bounded by $\theta$, then 2.~and 3.~are implied by 1., but we state the assumption this way for convenience.
\end{remark}

\begin{subsubsection}{A motivating example}
Before we state our main theorem, it will help to motivate some of our assumptions by considering an application to the stochastic gradient algorithm. 

Consider a classification problem where one tries to learn the
parameters $w$ of a model. One is given $S$ datapoints $(z_1,y_1)...(z_s,y_s)$, and a likelihood function $\ell(w,(z,y))$, and one tries to minimize $H(w)$ for
$$H(w) := \frac{1}{S}\sum_{i=1}^{S} 
H_i(w), \qquad \text{with} \qquad
H_i(w) := \ell(w, (z_i,y_i)).$$
The stochastic gradient algorithm proceeds as follows:
\begin{align*}
w_{k+1} =& w_{k} - \delta \nabla H_{\eta_k}(w_k)
\end{align*}

Let us rescale the above by defining $x := w/\sqrt{\delta}$ and $U(x) := H(w)/\delta$. One can then verify that $\nabla U(x) = \nabla H(w) / \sqrt{\delta}$ so that the above dynamics is equivalent to 
\begin{align*}
x_{k+1} 
=& x_k - \delta \nabla U_{\eta_k}\lrp{x_k}\\
=& x_k - \delta \nabla U\lrp{x_k} + \sqrt{2\delta} T_{\eta_k}(x_k),
\numberthis \label{e:sg:1}
\end{align*}
where for each $k$, $\eta_k$ is an integer sampled uniformly from $\lrbb{1...S}$, and we define $ T_{\eta_k}(x) := \sqrt{\delta/2} \lrp{\nabla U(x) - \nabla U_{\eta_k} (x)}= 1/\sqrt{2}\lrp{\nabla H(\sqrt\delta x) - \nabla H_{\eta_k}(\sqrt\delta x)}$. Notice that \eqref{e:sg:1} is identical to \eqref{e:discretesde}.

The mean and variance of $T_\eta$ are
\begin{align*}
\Ep{\eta}{T_\eta(x)} &= 0\\
\Ep{\eta}{T_\eta(x) T_\eta(x)^T}
&= {\delta/2}\cdot \Ep{i\sim Unif(\lrbb{1...S})}{\lrp{\nabla U(x) - \nabla U_i(x)}\lrp{\nabla U(x) - \nabla U_i(x)}^T}\\
&= {1/2}\cdot \Ep{i\sim Unif(\lrbb{1...S})}{\lrp{\nabla H(\sqrt\delta x) - \nabla H_i(\sqrt\delta x)}\lrp{\nabla H(\sqrt\delta x) - \nabla H_i(\sqrt\delta x)}^T}
\end{align*}
Assume that there exists a constant $c_\sigma$ such that 
$H_i(w)$ satisfies
$$\Ep{i\sim Unif(\lrbb{1...S})}{\lrp{\nabla H(w) - \nabla H_i(w)}\lrp{\nabla H(w) - \nabla H_i(w)}^T} \prec \sqrt{2} c_\sigma I ,$$
then Assumption \ref{ass:ximeanandvariance} is satisfied.

Furthermore, $T_\eta(x), G_\eta(x), M_\eta(x), N_\eta(x)$ are respectively $\sqrt{\delta/2}\nabla \lrp{U(x) - U_\eta(x)}$, \\$\sqrt{\delta/2} \nabla^2 \lrp{U(x) - U_\eta(x)}$, $\sqrt{\delta/2} \nabla^3 \lrp{U(x) - U_\eta(x)}$, $\sqrt{\delta/2} \nabla^4 \lrp{U(x) - U_\eta(x)}$, so Assumption \ref{ass:gisregular} is satisfied if the loss function $\ell$ has 2nd, 3rd and 4th order derivatives (in $w$) which are globally bounded. 

If $\nabla H_i(w)$ is $m$-strongly convex and has $L$-Lipschitz
gradients for all $i$, then $\nabla U_i(x)$ is also $m$-strongly convex and $L$-smooth for all $i$, so that Assumption~\ref{ass:discreteprocesscontracts} is satisfied for $\lambda = m$ for all $\delta\leq 1/(2L)$, by a synchronous coupling argument (see Lemma~\ref{l:discrete_contraction} in Appendix \ref{s:appendix:inhomogeneous}).

\end{subsubsection}

We will now state our main theorem for this section:

\begin{theorem}\label{t:convergencerate}
Let $p_0$ be an arbitrary initial distribution, and let $p_{k}$ be defined as in \eqref{e:discretesde} with step size $\delta$. Recall the defintion of $\hat{p}$ as the invariant distribution of \eqref{e:discretesde} and $p^*$ as the invariant distribution of \eqref{e:exactsde}.
For $\delta \leq \frac{\epsilon^2}{d^7}\cdot poly
\lrp{\frac{1}{m},L,\theta}^{-1}$,
\begin{align*}
W_2\lrp{\hat{p}, p^*}\leq {\epsilon}.
\numberthis \label{e:t2:1}
\end{align*}
If, in addition, $k \geq \frac{d^7}{\epsilon^2}\log
\frac{W_2(p_0,p^*)}{\epsilon} \cdot poly \lrp{L, \theta, \frac{1}{m},
c_{\sigma},  \frac{1}{\lambda} }$, then
\begin{align*}
W_2\lrp{p_{k}, p^*}\leq {\epsilon}.
\numberthis \label{e:t2:2}
\end{align*}
\end{theorem}

\begin{remark}

Like Theorem \ref{t:s_convergencerate}, this
also gives a $1/\sqrt k$ rate, which
is optimal. (see Section~\ref{sss:clt}).
\end{remark}

\end{subsection}
\end{section}

\begin{section}{Proof of Main Theorems}
In this section, we sketch the proofs of
Theorems~\ref{t:s_convergencerate} and~\ref{t:convergencerate}.

\begin{subsection}{Proof of Results for Homogeneous Diffusion}
\label{ss:proofsketch_homogeneous}
\begin{proof}[Proof of Theorem \ref{t:s_convergencerate}]

We first prove \eqref{e:t1:2}.

By Theorem \ref{t:s_main} below, for $\delta \leq \frac{\min\lrbb{m^2,1}}{2^{18} d^2 \lrp{L+1}^3}$,
\begin{align*}
W_2(p_{k}, p^*) \leq& e^{-m \delta k/8} W_2(p_0,p^*) + 2^{82}\delta^{1/2} d^{3/2} \lrp{L+1}^{9/2} 
\max\lrbb{ \frac{1}{m}\log \lrp{\frac{1}{m}},1}^{7}.
\numberthis \label{e:go:1}
\end{align*}
Thus if $\delta \leq \epsilon^2 \cdot \lrp{2^{166} d^3 \lrp{L+1}^9 \max\lrbb{\frac{1}{m}\log \lrp{\frac{1}{m}},1}^{14}}^{-1}$, then
\[
2^{82}\delta^{1/2} d^{3/2} \lrp{L+1}^{9/2} 
\max\lrbb{ \frac{1}{m}\log \lrp{\frac{1}{m}},1}^{7} \leq \frac{\epsilon}{2}.
\]
Additionally, if $k \geq \frac{8}{m \delta }\log \frac{2W_2\lrp{p_0,p^*}}{\epsilon}$, then $e^{-m \delta k/8} W_2(p_0,p^*) \leq \frac{\epsilon}{2}$, so together, we get
$
W_2\lrp{p_k,p^*}\leq \epsilon.
$
To summarize, our assumptions are
\begin{align*}
\delta \leq& \min \lrbb{\frac{\min\lrbb{m^2,1}}{2^{18} d^2 \lrp{L+1}^3}, \frac{\epsilon^2}{2^{166} d^3 \lrp{L+1}^9 \max\lrbb{\frac{1}{m}\log \lrp{\frac{1}{m}},1}^{14}}} = \frac{\epsilon^2}{d^3} \cdot poly \lrp{\frac{1}{m},L}^{-1}
\end{align*}
and
\begin{align*}
k \geq& \frac{8}{m\delta} \log \frac{2W_2\lrp{p_0,p^*}}{\epsilon}
= \frac{d^3}{\epsilon^2} \cdot \log \frac{W_2(p_0,p^*)}{\epsilon} poly
\lrp{\frac{1}{m},L}.
\end{align*}
This proves \eqref{e:t1:2}. To prove \eqref{e:t1:1}, use our above assumption on $\delta$, and take the limit of \eqref{e:go:1} as $k\to\infty$.

\end{proof}

\begin{theorem}\label{t:s_main}
Let $p_0$ be an arbitrary initial distribution, and let $p_{k}$ be defined as in \eqref{e:discretesde}.\\
Let $\epsilon>0$ be some arbitrary constant.
For any step size $\delta$  satisfying
$\delta \leq \frac{\min\lrbb{m^2,1}}{2^{18} d^2 \lrp{L+1}^3}$,
the Wasserstein distance between $p_{k}$ and $p^*$ is upper bounded as
\begin{align*}
W_2(p_{k}, p^*) \leq& e^{-m \delta k/8} W_2(p_0,p^*) + 2^{82}\delta^{1/2} d^{3/2} \lrp{L+1}^{9/2} 
\max\lrbb{ \frac{1}{m}\log \lrp{\frac{1}{m}},1}^{7}.
\end{align*}
\end{theorem}
\begin{proof}[Proof of Theorem \ref{t:s_main}]
Recall our definition of $\Phi_\delta$ in \eqref{d:Phip}. Let $\Phi_\delta^k$ denote $k$ repeated applications of $\Phi_\delta$, so $p_k = \Phi_\delta^k (p_0)$. Our objective is thus to bound
$W_2(\Phi_{\delta}^k(p_0), p^*)$.

We first use triangle inequality to split the objective into two terms:
\begin{align*}
W_2(\Phi_{\delta}^k(p_0), p^*)
\leq& W_2(\Phi_{\delta}^k(p_0), \Phi_{\delta}^k(p^*)) + W_2(\Phi_{\delta}^k(p^*), p^*) 
\numberthis \label{e:s_sd:9}
\end{align*}

The first term is easy to bound. We can apply Lemma~\ref{l:s_discrete_contraction} (in Appendix \ref{s:appendix:homogeneous}) to get
\begin{align*}
W_2(\Phi_{\delta}^k(p^*), p^*) \leq e^{-m \delta k/8} W_2(p_0,p^*)
\numberthis \label{e:s_sd:8}
\end{align*}

To bound the second term of \eqref{e:s_sd:9}, we use an argument adapted from (Zhai 2016):
\begin{align*}
W_2(\Phi_{\delta}^k(p^*), p^*)
=& W_2(\Phi_{\delta}(\Phi_{\delta}^{k-1} (p^*)), p^*)\\
\leq& W_2(\Phi_{\delta}(\Phi_{\delta}^{k-1} (p^*)), \Phi_{\delta}(p^*)) + W_2(\Phi_{\delta}(p^*), p^*)\\
\leq&  e^{-m \delta/8} W_2(\Phi_{\delta}^{k-1} (p^*), p^*) + W_2(\Phi_{\delta}(p^*), p^*)\\
\vdots&\\
\leq& \sum_{i=0}^{k-1} e^{-m \delta i/8} W_2(\Phi_{\delta}(p^*), p^*)\\
\leq& \frac{8}{m \delta}W_2(\Phi_{\delta}(p^*), p^*).
\end{align*}
Here the third inequality is by induction.
This reduces our problem to bounding the expression $W_2(\Phi_{\delta}(p^*), p^*)$, which can be thought of as the one-step divergence between \eqref{e:discretesde} and \eqref{e:exactsde} when $p_0=p^*$. We apply Lemma \ref{t:s_chisquaredbound} below to get
\begin{align*}
W_2(\Phi_{\delta}(p^*), p^*) \leq& 2^{78}\delta^{3/2} d^{3/2} \lrp{L+1}^{9/2} 
\max\lrbb{ \frac{1}{m}\log \lrp{\frac{1}{m}},1}^{6} \\
\Rightarrow \quad \frac{8}{m \delta}W_2(\Phi_{\delta}(p^*), p^*)\leq& 2^{82}\delta^{1/2} d^{3/2} \lrp{L+1}^{9/2} 
\max\lrbb{ \frac{1}{m}\log \lrp{\frac{1}{m}},1}^{7}.
\numberthis \label{e:s_sd:7}
\end{align*}
Thus, substituting \eqref{e:s_sd:8} and \eqref{e:s_sd:7} into \eqref{e:s_sd:9}, we get
\begin{align*}
W_2(\Phi_{\delta}^k(p_0), p^*) \leq& e^{-m \delta k/8} W_2(p_0,p^*) + 2^{82}\delta^{1/2} d^{3/2} \lrp{L+1}^{9/2} 
\max\lrbb{ \frac{1}{m}\log \lrp{\frac{1}{m}},1}^{7}.
\numberthis \label{e:s_sd:2}
\end{align*}
\end{proof}

\begin{lemma}\label{t:s_chisquaredbound}
Let $p_\delta := \Phi_\delta(p^*)$. Then for any $\delta \leq \frac{\min\lrbb{m^2,1}}{2^{18} d^2 \lrp{L+1}^3}$,
\begin{align*}
W_2 (p_\delta, p^*) \leq 2^{78}\delta^{3/2} d^{3/2} \lrp{L+1}^{9/2} 
\max\lrbb{ \frac{1}{m}\log \lrp{\frac{1}{m}},1}^{6} .
\end{align*}
\end{lemma}
(This lemma is similar in spirit to Lemma~1.6 in \cite{zhai2018high}.)

\begin{proof}[Proof of Lemma \ref{t:s_chisquaredbound}]
$ $\\
Using Talagrand's inequality and the fact that $U(x)$ is strongly
convex, we can upper bound $W_2^2(q, p^*)$ by $\chi^2(q,p^*)$ for any
distribution $q$ which has density wrt $p^*$, i.e.:
\begin{align*}
W_2^2(p^*,p_\delta) 
\leq& \frac{2}{m} \int \lrp{\frac{p_\delta (x)}{p^*(x)} - 1}^2 p^*(x)
\,dx.
\numberthis \label{e:s_rc:1}
\end{align*}

See Lemma \ref{l:s_upperboundw2bychisquared} in Appendix \ref{s:appendix:inhomogeneous} for a rigorous proof of \eqref{e:s_rc:1}.

Under our assumptions on $\delta$, we can apply
Lemma~\ref{l:s_pdeltaoverpstariscloseto1} below, giving
\begin{align*}
&\int_{B_R} \lrp{\frac{p_\delta(x)}{p^*(x)}-1}^2 p^*(x)\, dx\\
\leq& 2^{23}\delta^{3} d^2 \lrp{L+1}^{9} 
\int \exp\lrp{\frac{m}{32}\|x\|_2^2} \lrp{\|x\|_2^{12} + 1} p^*(x)\, dx\\
\leq& 2^{24}\delta^{3} d^2 \lrp{L+1}^{9} 
\lrp{\int \exp\lrp{\frac{m}{16}\|x\|_2^2} p^*(x) dx + \int
\lrp{\|x\|_2^{24} + 1} p^*(x)\, dx}\\
\leq& 2^{24}\delta^{3} d^2 \lrp{L+1}^{9} 
\lrp{8d + \max\lrbb{\lrp{2^{11} \frac{1}{m} \log \lrp{2^8/m}}^{11}, 2^{11} \frac{1}{m} }}\\
\leq& 2^{156}\delta^{3} d^3 \lrp{L+1}^{9} 
\max\lrbb{ \frac{1}{m}\log \lrp{\frac{1}{m}},1}^{11},
\end{align*}
where the first inequality is by Lemma \ref{l:s_pdeltaoverpstariscloseto1}, the second inequality is by Young's inequality, the third inequality is by Lemma \ref{l:p^*hasboundedexponent} and Lemma \ref{l:kthmomentbound}, with $c_\sigma=1$.
Plugging the above into \eqref{e:s_rc:1},
\begin{align*}
W_2^2(p^*,p_\delta) 
\leq& 2^{156}\delta^3d^3 \lrp{L+1}^{9} 
\max\lrbb{ \frac{1}{m}\log \lrp{\frac{1}{m}},1}^{12} .
\end{align*}
\end{proof}
The following lemma studies the
``discretization error'' between the SDE~\eqref{e:exactsde} and one
step of~\eqref{e:discretesde}.

\begin{lemma}\label{l:s_pdeltaoverpstariscloseto1}
Let $p_\delta := \Phi_\delta(p^*)$. For any $R\geq 0$, $x\in B_R$, and
$\delta \leq \frac{\min\lrbb{m^2,1}}{2^{18} d \lrp{L+1}^3}$,
\begin{align*}
\lrabs{\frac{p_\delta(x)}{p^*(x)}-1}\leq 512 \delta^{3/2} d
\lrp{L+1}^{9/2} \exp\lrp{\frac{m}{32}\|x\|_2^2} \lrp{\|x\|_2^6 + 1}.
\end{align*}
\end{lemma}

\begin{proof}[Proof of Lemma \ref{l:s_pdeltaoverpstariscloseto1}]
Recall that $p_\delta = \lrp{F_\eta}_{\#} p^*(x)$.
Thus by the change of variable formula, we have
\begin{align*}
p_\delta(x) 
&= \int p^*(F_\eta^{-1}(x))  \det\lrp{\nabla F_\eta\lrp{F_\eta^{-1}
(x)}}^{-1}  q(\eta)\,d\eta\\
&= \Ep{q(\eta)}{\underbrace{p^*(F_\eta^{-1}(x))}_{\circled{1}}
\underbrace{\det\lrp{\nabla F_\eta\lrp{F_\eta^{-1}
(x)}}^{-1}}_{\circled{2}} },
\numberthis \label{e:s_pdeltaintegralexpression}
\end{align*}
where $\nabla F_\eta(y)$ denotes the Jacobian matrix of $F_\eta$ at $y$. The invertibility of $F_\eta$ is shown in Lemma \ref{l:fisinvertible}.
We rewrite $\circled{1}$ as its Taylor expansion about $x$:
\begin{align*}
p^*\lrp{F_\eta^{-1}(x)}=& \underbrace{p^*(x)}_{\circled{4}} +  \underbrace{{\lin{\nabla p^*(x), F_\eta^{-1}(x) - x}}}_{\circled{5}} +\frac{1}{2}\underbrace{\lin{\nabla^2 p^*(x), \lrp{F_\eta^{-1}(x) -x}\lrp{F_\eta^{-1}(x) -x}^T}}_{\circled{6}} \\
&\quad+ \underbrace{ \int_0^1 \int_0^t \int_0^s \lin{\nabla^3 p^*\lrp{(1-r) x + r F_\eta^{-1}(x)}, \lrp{F_\eta^{-1}(x) - x}^{3}}dr ds dt }_{\circled{7}}.
\end{align*}
Substituting the above into \eqref{e:s_pdeltaintegralexpression} and applying Lemmas~\ref{l:s_circled4times2}, \ref{l:s_circled5times2}, \ref{l:s_circled6times2},
and~\ref{l:s_circled7times2}, we get
\begin{align*}
p_\delta(x) 
=& \Ep{q(\eta)}{ \lrp{\circled{4}+\circled{5}+\circled{6}+\circled{7}} \cdot \circled{2}}\\
=& p^*(x) + p^*(x) \lrp{\delta \tr\lrp{\nabla^2 U(x)}} + \delta \lin{\nabla p^*(x), \nabla U(x)} + \delta \tr\lrp{\nabla^2 p^*(x)} + \Delta,
\numberthis \label{e:la:1}
\end{align*}
for some $\Delta$ satisfying
\begin{align*}
\lrabs{\Delta} 
\leq& p^*(x) \cdot 8 \delta^{3/2} dL^{3/2}\lrp{\|x\|_2 + 1}
+ p^*(x) \cdot 16 \delta^{3/2} d L^{5/2}\lrp{\|x\|_2^3 + 1}\\
&\quad {}+ p^*(x) \cdot 64 \delta^{3/2} d \lrp{L +1}^{5/2}\lrp{\|x\|_2^5 + 1} \\
&\quad {}+ p^*(x) \cdot 256 \delta^{3/2} \exp\lrp{\frac{m}{32}\|x\|_2^2}\lrp{L+1}^{9/2} \lrp{\|x\|_2^6 + 1}\\
\leq& p^*(x) \cdot 512 \delta^{3/2} d \lrp{L+1}^{9/2}
\exp\lrp{\frac{m}{32}\|x\|_2^2} \lrp{\|x\|_2^6 + 1}.
\end{align*} 

Furthermore, by using the expression $p^*(x) \propto e^{-U(x)}$ and some algebra, we see that 
\begin{align*}
& p^*(x) \lrp{\delta \tr\lrp{\nabla^2 U(x)}} + \delta \lin{\nabla p^*(x), \nabla U(x)} + \delta \tr\lrp{\nabla^2 p^*(x)} \\
\numberthis \label{e:fk:1}
&= \delta p^*(x) \lrp{\tr\lrp{\nabla^2 U(x)} - \lrn{\nabla U(x)}_2^2 - \tr\lrp{\nabla^2 U(x)} + \tr\lrp{\nabla U(x) \nabla U(x)^T}}\\
=&=0.
\end{align*}
Substituting the above into \eqref{e:la:1} gives $p_\delta(x) = p^*(x) + \Delta$, which implies that
\begin{align*}
\lrabs{\frac{p_\delta(x)}{p^*(x)} - 1}\leq 512 \delta^{3/2} d
\lrp{L+1}^{9/2} \exp\lrp{\frac{m}{32}\|x\|_2^2} \lrp{\|x\|_2^6 + 1}.
\end{align*}
\end{proof}
\end{subsection}

\begin{subsection}{Proof of Results for Inhomogeneous Diffusion}
The proof of Theorem \ref{t:convergencerate} is quite similar to the
proof of Theorem \ref{t:s_convergencerate}, and can be found in the
Appendix (Section~\ref{s:appendix:inhomogeneous}). We will highlight
some additional difficulties in the proof compared to
Theorem~\ref{t:s_convergencerate}.

The heart of the proof lies in Lemma~\ref{t:chisquaredbound}, which
bounds the discretization error between the SDE~\eqref{e:exactsde} and
one step of the discrete process~\eqref{e:discretesde},
in the form of $W_2(\Phi_\delta(p^*), p^*)$. This is
analogous to Lemma~\ref{t:s_chisquaredbound} in Section~\ref{ss:proofsketch_homogeneous}. Compared to the proof of Lemma~\ref{t:s_chisquaredbound}, one additional difficulty is that we can no
longer rely on Talagrand's inequality \eqref{e:s_rc:1}. This is
because $p^*$ is no longer guaranteed to be strongly log-concave. We
instead use the fact that $p^*(x)$ is subgaussian to upper bound
$W_2(\cdot, p^*)$ by $\chi^2$ (see Corollary
\ref{c:upperboundw2bychisquared}).

Lemma~\ref{t:s_chisquaredbound} in turn relies crucially on bounding the
expression $\lrabs{\frac{p_\delta(x)}{p^*(x)} -1}$. This is proved in
Lemma~\ref{l:pdeltaoverpstariscloseto1}, which is the analog of Lemma~\ref{l:s_pdeltaoverpstariscloseto1} in Section
\ref{ss:proofsketch_homogeneous}. The additional difficulty is that we
have to handle the effects of a diffusion matrix $\sigma_x$ that
depends on the position $x$. Also, Lemma~\ref{l:pdeltaoverpstariscloseto1} relies on the closed-form expression
for $p^*\propto e^{-U(x)}$ in order to cancel out terms of order less
than $\delta^{3/2}$ in~\eqref{e:fk:1}. We do not have a closed-form
expression for $p^*$ when the diffusion is inhomogeneous, and we
instead rely on an argument based on the invariance of $p^*(x)$ under
the Fokker-Planck equation (see~\eqref{e:fk:2}). This allows us to,
somewhat remarkably, prove that $p_k$ converges to $p^*$ using only
the implicit description of $p^*$ as the invariant distribution of
\eqref{e:exactsde}.
\end{subsection}
\end{section}
\begin{section}{Conclusion and Future Directions}
The main result of this paper is a generalization of the
classical Central Limit Theorem to discrete-time
stochastic processes of the form~\eqref{e:discretesde}, giving
rates of convergence to a certain invariant distribution $p^*$.
Our results assume that $U(x)$ is strongly convex (Assumption~\ref{ass:uissmooth}.3). This is not strictly necessary. We use strong convexity in two ways: 
\begin{enumerate}
\item We use it for proving contraction of \eqref{e:discretesde}, as
in Lemma~\ref{l:s_discrete_contraction} and Lemma~\ref{l:discrete_contraction}. Assuming that the noise $T_\eta$
contains an independent symmetric component (e.g., Gaussian noise),
and assuming that $U(x)$ is nonconvex inside but strongly convex
outside a ball, then we can use a reflection
coupling argument to show that Assumption~\ref{ass:discreteprocesscontracts}
holds.
\item We use it for proving that $p^*$ is subgaussian, as in Lemma~\ref{l:p^*issubgaussian}. For this lemma, it suffices that $U(x)$ is
$m$-dissipative.
\end{enumerate}

Another assumption that can be relaxed is Assumption
\ref{ass:ximeanandvariance}.2, which is used to show that $p^*$ is
subgaussian. We can replace this assumption by the weaker condition
\[
\Ep{q(\eta)}{T_\eta(x)T_\eta(x)^T} \prec c_{\sigma}^2 \|x\|_2^2 I.
\]
We only need to make an additional assumption that $U(x)$ is
$M$-dissipative for some radius $D$, with $M \geq 8 c_{\sigma}^2$. We do not prove this here to keep the proofs simple; a proof will be included in the full version of this paper.

Finally, we remark that $\eqref{e:exactsde}$ suggests that $x_t$ moves quickly through regions of large $\sigma_x$. This seems to suggest that in the stochastic gradient algorithm, the iterates will, with higher probability, end up in minima where the covariance of the
gradient is small. This may in turn suggest that the noise SGD tends to select ``stable'' solutions, where stability is defined as the determinant of the covariance of the gradient. This property would not be present with a different noise such as Gaussian noise in Langevin diffusion. A rigorous investigation of this possibility is beyond the scope of this paper.

\end{section}

\begin{section}{Acknowledgements}
This work was supported in part by the Mathematical Data Science program of the
Office of Naval Research under grant number N00014-18-1-2764. Part of this work was done at the Foundations of Data Science program at the Simons Institute for the Theory of Computing. The authors thank Thomas Courtade for helpful discussions on an initial draft of the paper.
\end{section}

\nocite{*}

\newpage
\bibliographystyle{plain}
\bibliography{references}

\newpage
\appendix
\begin{section}{Auxiliary Lemmas for Section \ref{ss:mainresult_homogeneousnoise}}
\label{s:appendix:homogeneous}
In this subsection, we present the proof of Lemma~\ref{t:s_chisquaredbound}, as well as some auxiliary lemmas.
\begin{lemma}\label{l:s_circled4times2}
For $\delta \leq \frac{1}{2d^2 L}$,
\begin{align*}
& \Ep{q(\eta)}{p^*(x) \cdot \det\lrp{\nabla
F_\eta\lrp{F_\eta^{-1}(x)}}}= p^*(x) + p^*(x) \lrp{\delta
\tr\lrp{\nabla^2 U(x)}} + \Delta,
\end{align*}
for some $\lrabs{\Delta} \leq p^*(x) \cdot 8 \delta^{3/2}
dL^{3/2}\lrp{\|x\|_2 + 1}$.
\end{lemma}
\begin{proof}[Proof of Lemma \ref{l:s_circled4times2}]
Let us define 
\begin{align*}
\Delta' :=& \det\lrp{\nabla F_\eta(F_\eta^{-1} (x))}^{-1} -\lrp{1 +
\delta \tr \lrp{\nabla^2 U(x)}}.
\end{align*}
By Lemma~\ref{l:s_discretizeddeterminantexpansioninverse},
$\lrabs{\Delta'} \leq 8 \delta^{3/2} d L^{3/2}\lrp{\|x\|_2 + 1}$, so
\begin{align*}
\Ep{q(\eta)}{p^*(x) \cdot \det\lrp{\nabla F_\eta\lrp{F_\eta^{-1}(x)}}}
= & \Ep{q(\eta)}{p^*(x) \cdot \lrp{1 + \delta \tr\lrp{\nabla^2 U(x)}}}  + \Ep{q(\eta)}{p^*(x) \cdot \Delta'}\\
= & p^*(x) \lrp{1 + \delta \tr\lrp{\nabla^2 U(x)} + p^*(x)} \cdot \Delta'
\end{align*}
We complete the proof by taking $\Delta := p^*(x) \Delta'$.
\end{proof}

\begin{lemma}\label{l:s_circled5times2}
For $\delta \leq \frac{1}{64 d^2 L}$,
\begin{align*}
& \Ep{q(\eta)}{\lin{\nabla p^*(x), F_\eta^{-1}(x) - x}\cdot \det\lrp{\nabla F_\eta(F_\eta^{-1} (x))}^{-1}} = \delta \lin{\nabla p^*(x), \nabla U(x)} + \Delta
\end{align*}
for some $\lrabs{\Delta}\leq  p^*(x) \cdot 16 \delta^{3/2} d
L^{5/2}\lrp{\|x\|_2^3 + 1}$.
\end{lemma}
\begin{proof}[Proof of Lemma~\ref{l:s_circled5times2}]
Let 
\begin{align*}
&\Delta_1 := F_\eta^{-1}(x) - x - \lrp{-\sqrt{2\delta} T_\eta + \delta
\nabla U(x)},&
& \Delta_2 := \det\lrp{\nabla F_\eta(F_\eta^{-1} (x))}^{-1} - 1.
\end{align*}
By Lemma~\ref{l:s_onestepdiscretizationbounds}.2 and Corollary \ref{c:s_determinantinversenaivenaivebound},
\begin{align*}
&\lrn{\Delta_1}_2 \leq 4 \delta^{3/2} L^{3/2} \lrp{\|x\|_2 + 1}, &
&\lrabs{\Delta_2} \leq 2 \delta d L \lrp{\|x\|_2 + 1}.
\end{align*}
Moving terms around, 
\begin{align}
\nonumber
& \Ep{q(\eta)}{\lin{\nabla p^*(x), F_\eta^{-1}(x) - x}\cdot \det\lrp{\nabla F_\eta(F_\eta^{-1} (x))}^{-1}}\\
\label{e:s_fx:1}
=& \Ep{q(\eta)}{\lin{\nabla p^*(x), -\sqrt{2\delta}T_\eta}} + \Ep{q(\eta)}{\lin{\nabla p^*(x), \delta \nabla U(x)}}\\
\label{e:s_fx:3}
&\quad + \Ep{q(\eta)}{\lin{\nabla p^*(x), \sqrt{2\delta} T_\eta + \delta \nabla U(x)}\cdot \Delta_2}\\
\label{e:s_fx:4}
&\quad + \Ep{q(\eta)}{\lin{\nabla p^*(x), \Delta_1} \cdot
\det\lrp{\nabla F_\eta(F_\eta^{-1} (x))}^{-1}}.
\end{align}
The main term of interest is \eqref{e:s_fx:1}, which evaluates to
\begin{align*}
&\Ep{q(\eta)}{\lin{\nabla p^*(x), -\sqrt{2\delta}T_\eta}} + \Ep{q(\eta)}{\lin{\nabla p^*(x), \delta \nabla U(x)}}\\
=& \Ep{q(\eta)}{\lin{\nabla p^*(x), \delta \nabla U(x)}}\\
=& \delta \lin{\nabla p^*(x), \nabla U(x)},
\end{align*}
where the first equality is by Assumption~\ref{ass:ximeanandvariance}.1.

We now consider the terms in \eqref{e:s_fx:3} and \eqref{e:s_fx:4}:
\begin{align*}
\lrabs{\eqref{e:s_fx:3}} =& \lrabs{\Ep{q(\eta)}{\lin{\nabla p^*(x), \sqrt{2\delta} T_\eta + \delta \nabla U(x)}\cdot \Delta_2}}\\
\leq& \lrn{\nabla p^*(x)}_2 \Ep{q(\eta)}{\lrn{\sqrt{2\delta}T_\eta + \delta \nabla U(x)}_2\lrabs{\Delta_2}}\\
\leq& p^*(x) L\|x\|_2 \cdot \sqrt{2\delta} \lrp{\sqrt{L} + \sqrt{\delta} L\|x\|_2}\cdot 2 \delta dL\lrp{\|x\|_2 + 1}\\
\leq& 8 p^*(x) \delta^{3/2} d L^{5/2} \lrp{\|x\|_2^3 + 1},
\end{align*}
where the first inequality is by Cauchy-Schwarz, and the second
inequality is by Lemma~\ref{l:s_logp^*issmooth}.1, our upperbound on
$\lrabs{\Delta_2}$ at the start of the proof, and
Assumptions~\ref{ass:uissmooth}.2 and~\ref{ass:gisregular}.2.
\begin{align*}
\lrabs{\eqref{e:s_fx:4}} =& \lrabs{\Ep{q(\eta)}{\lin{\nabla p^*(x), \Delta_1} \cdot \det\lrp{\nabla F_\eta(F_\eta^{-1} (x))}^{-1}}}\\
\leq& \lrn{\nabla p^*(x)}_2 \Ep{q(\eta)}{\lrn{\Delta_1}_2 \cdot \lrabs{\det\lrp{\nabla F_\eta(F_\eta^{-1} (x))}^{-1}}}\\
\leq& p^*(x) L\|x\|_2 \cdot 4\delta^{3/2} L^{3/2} \lrp{\|x\|_2 + 1} \cdot \lrp{1+2 \delta dL\lrp{\|x\|_2 + 1}}\\
\leq& 8 p^*(x) \delta^{3/2} d L^{5/2}\lrp{\|x\|_2^3 + 1},
\end{align*}
where the first inequality is by Cauchy-Schwarz, and the second
inequality is by Lemma~\ref{l:s_logp^*issmooth}.1, our upperbound
on $\lrn{\Delta_1}_2$ and $\lrabs{\Delta_2}$ at the start of the
proof, and our assumption on $\delta$.

Letting $\Delta := \eqref{e:s_fx:3}+\eqref{e:s_fx:4}$, we have
\begin{align*}
\lrabs{\Delta}
\leq& 8 p^*(x) \delta^{3/2} d L^{5/2} \lrp{\|x\|_2^3 + 1}
+ 8 p^*(x) \delta^{3/2} d L^{5/2}\lrp{\|x\|_2^3 + 1}\\
\leq& p^*(x) \cdot 16 \delta^{3/2} d L^{5/2} \lrp{\|x\|_2^3 + 1}.
\end{align*}
\end{proof}

\begin{lemma}\label{l:s_circled6times2}
For $\delta \leq \frac{1}{64 d^2 (L+1)}$,
\begin{align*}
& \frac{1}{2} \Ep{q(\eta)}{\lin{\nabla^2 p^*(x), \lrp{F_\eta^{-1}(x) - x}\lrp{F_\eta^{-1}(x) - x}^T}\cdot \det\lrp{\nabla F_\eta(F_\eta^{-1} (x))}^{-1}}\\
=& \delta \tr\lrp{\nabla^2 p^*(x)} + \Delta
\end{align*}
for some $\lrabs{\Delta}\leq p^*(x) \cdot 64 \delta^{3/2} d \lrp{L
+1}^{5/2}\lrp{\|x\|_2^5 + 1} $.
\end{lemma}

\begin{proof}[Proof of Lemma~\ref{l:s_circled6times2}]
Define
\begin{align*}
&\Delta_1 := F_\eta^{-1}(x) - x - \lrp{-\sqrt{2\delta}T_\eta}, &
&\Delta_2 := \det\lrp{\nabla F_\eta(F_\eta^{-1} (x))}^{-1} -  1.
\end{align*}
By Lemma~\ref{l:s_onestepdiscretizationbounds}.3 and Corollary \ref{c:s_determinantinversenaivenaivebound},
\begin{align*}
\lrabs{\Delta_1} \leq& 2 \delta L\lrp{\|x\|_2 + 1},&
\lrabs{\Delta_2} \leq& 2 \delta d L \lrp{\|x\|_2 + 1}.
\end{align*}
Then
\begin{align}
\nonumber
& \Ep{q(\eta)}{\lin{\nabla^2 p^*(x), \lrp{F_\eta^{-1}(x) - x}\lrp{F_\eta^{-1}(x) - x}^T}\cdot \det\lrp{\nabla F_\eta(F_\eta^{-1} (x))}^{-1}}\\
\label{e:s_ta:1}
=& 2\delta \Ep{q(\eta)}{\lin{\nabla^2 p^*(x), T_\eta T_\eta^T}}\\
\label{e:s_ta:2}
&\quad + 2\delta \Ep{q(\eta)}{\lin{\nabla^2 p^*(x), T_\eta T_\eta^T}\cdot \Delta_2}\\
\label{e:s_ta:3}
&\quad + \Ep{q(\eta)}{\lin{\nabla^2 p^*(x), \Delta_1 \Delta_1^T -
\sqrt{2\delta}T_\eta \Delta_1^T - \sqrt{2\delta}\Delta_1 T_\eta
^T}\cdot \det\lrp{\nabla F_\eta(F_\eta^{-1} (x))}^{-1}}.
\end{align}
We are mainly interested in~\eqref{e:s_ta:1}, which evaluates to 
\begin{align*}
2\delta \Ep{q(\eta)}{\lin{\nabla^2 p^*(x), T_\eta T_\eta^T}}
=& 2\delta \lin{\nabla^2 p^*(x), \Ep{q(\eta)}{T_\eta T_\eta^T}}\\
=& 2\delta \tr\lrp{\nabla^2 p^*(x)},
\end{align*}
where the last equality is by Assumption \ref{ass:simplerassumptions}.1. 

We now bound the magnitudes of $\eqref{e:s_ta:2}$ and $\eqref{e:s_ta:3}$.
\begin{align*}
\lrabs{\eqref{e:s_ta:2}}
=& \lrabs{2\delta \Ep{q(\eta)}{\lin{\nabla^2 p^*(x), T_\eta T_\eta^T}\cdot \Delta_2}}\\
\leq& 2\delta \lrn{\nabla^2 p^*(x)}_2 \Ep{q(\eta)}{\lrn{T_\eta}_2^2 \lrabs{\Delta_2}}\\
\leq& 2\delta  p^*(x) \lrp{L + L^2\|x\|_2^2} \cdot L \cdot 2 \delta d L \lrp{\|x\|_2 + 1}\\
\leq& 8 \delta^2 p^*(x) d \lrp{L + 1}^4 \lrp{\|x\|_2^3 + 1},
\end{align*}
where the first inequality is by Cauchy-Schwarz, and the second inequality is by Lemma~\ref{l:s_logp^*issmooth}.2 and our upper bound on $\lrabs{\Delta_2}$ at the start of the proof.
\begin{align*}
\eqref{e:s_ta:3}
=& \Ep{q(\eta)}{\lin{\nabla^2 p^*(x), \Delta_1 \Delta_1^T - \sqrt{2\delta}T_\eta \Delta_1^T - \sqrt{2\delta}\Delta_1 T_\eta ^T}\cdot \det\lrp{\nabla F_\eta(F_\eta^{-1} (x))}^{-1}}\\
\leq& \lrn{\nabla^2 p^*(x)}_2 \Ep{q(\eta)}{\lrp{\lrn{\Delta_1}_2^2 + 2\sqrt{2\delta}\lrn{T_\eta}_2\lrn{\Delta_1}_2} \lrabs{\det\lrp{\nabla F_\eta(F_\eta^{-1} (x))}^{-1}}}\\
\leq& p^*(x) \lrp{L + L^2\|x\|_2^2} \cdot \lrp{\lrp{2 \delta L \lrp{\|x\|_2 + 1}}^2 + 4\sqrt{\delta} L^{1/2} \lrp{2 \delta L\lrp{\|x\|_2 + 1}}}\\
&\quad \cdot \lrp{1+2 \delta dL\lrp{\|x\|_2 + 1}}\\
\leq& 32 \delta^{3/2} p^*(x)\lrp{L +1}^{5/2}\lrp{\|x\|_2^5 + 1} ,
\end{align*}
where the first inequality is by Cauchy-Schwarz, and the second
inequality is by Lemma~\ref{l:s_logp^*issmooth}.2 and our upper bound
on $\lrabs{\Delta_1}$ at the start of the proof. Defining
$\Delta := \eqref{e:s_ta:2} + \eqref{e:s_ta:3}$, we have
\begin{align*}
\lrabs{\Delta} 
\leq& 8 \delta^2 p^*(x) d \lrp{L + 1}^3 \lrp{\|x\|_2^3 + 1} + 32 \delta^{3/2} p^*(x)\lrp{L +1}^{5/2}\lrp{\|x\|_2^5}\\
\leq& p^*(x) \cdot 64 \delta^{3/2} d \lrp{L +1}^{5/2}\lrp{\|x\|_2^5 +
1} .
\end{align*}
\end{proof}

\begin{lemma}\label{l:s_circled7times2}
For $\delta \leq \frac{\min\lrbb{m^2,1}}{2^{18} d^2 \lrp{L+1}^3}$,
\begin{align*}
&\!\!\!\!\!\! \lrabs{\Ep{q(\eta)}{\lrp{ \int_0^1 \int_0^t \int_0^s
\lin{\nabla^3 p^*\lrp{(1-r) x + r F_\eta^{-1}(x)}, \lrp{F_\eta^{-1}(x)
- x}^{3}}\,dr \,ds \,dt }\cdot \det\lrp{\nabla F_\eta(F_\eta^{-1} (x))}^{-1}}}\\
\leq & p^*(x) \cdot 256 \delta^{3/2}
\exp\lrp{\frac{m}{32}\|x\|_2^2}\lrp{L+1}^{9/2} \lrp{\|x\|_2^6 + 1}.
\end{align*}
\end{lemma}
\begin{proof}[Proof of Lemma~\ref{l:s_circled7times2}]
Using Lemma~\ref{l:s_onestepdiscretizationbounds}.1 and
our choice of $\delta$, $\lrn{x-F_\eta^{-1}(x)}_2\leq
\frac{1}{2}\lrp{\|x\|_2 + 1}$, and so $\lrn{F_\eta^{-1}(x)} \leq 2\|x\|_2 + 1$.
Thus for all $t\in[0,1]$,
\begin{align}
\label{e:s_he:1}
\lrn{(1-t)x + t F_\eta^{-1}(x)}_2 \leq 2\|x\|_2 + 1.
\end{align}
Thus,
\begin{align*}
&\!\!\!\!\! \lrabs{\Ep{q(\eta)}{\lrp{ \int_0^1 \int_0^t \int_0^s
\lin{\nabla^3 p^*\lrp{(1-t) x + t F_\eta^{-1}(x)}, \lrp{F_\eta^{-1}(x)
- x}^{3}}\,dr \,ds \,dt }\cdot \det\lrp{\nabla F_\eta(F_\eta^{-1} (x))}^{-1}}}\\
\leq& \Ep{q(\eta)}{ \int_0^1 \int_0^t \int_0^s \lrn{\nabla^3
p^*\lrp{(1-t) x + t F_\eta^{-1}(x)}_2} \,dr \,ds \,dt \cdot \lrn{F_\eta^{-1}(x)-x}_2^3\cdot \lrabs{\det\lrp{\nabla F_\eta(F_\eta^{-1} (x))}^{-1}}}\\
\leq& \Ep{q(\eta)}{p^*\lrp{(1-t) x + t F_\eta^{-1}(x)}\cdot  \lrp{L + 2 L^2\|(1-t) x + t F_\eta^{-1}(x)\|_2 + L^3 \|(1-t) x + t F_\eta^{-1}(x)\|_2^3}} \\
&\quad \cdot \lrp{4\delta^{1/2}L^{1/2}\lrp{\|x\|_2 + 1}}^3\cdot \lrp{1+2\delta dL\lrp{\|x\|_2 + 1}}\\
\leq& p^*(x) \exp\lrp{2L \lrp{\|x\|_2 + 1}\cdot 4 \delta^{1/2} L^{1/2}\lrp{\|x\|_2 + 1}}\cdot \lrp{L + 4L^2 \lrp{\|x\|_2 + 1} + 8L^3 \lrp{\|x\|_2^3 + 1}}\\
&\quad \cdot \lrp{4\delta^{1/2} L^{1/2}\lrp{\|x\|_2 + 1}}^3\cdot \lrp{1+2 \delta d L \lrp{\|x\|_2 + 1}}\\
\leq& p^*(x) \exp\lrp{16\delta^{1/2} L^{3/2}\lrp{\|x\|_2^2+1}}\cdot 256 \delta^{3/2} \lrp{L+1}^{9/2} \lrp{\|x\|_2^6 + 1}\\
\leq& 256 \delta^{3/2} p^*(x)
\exp\lrp{\frac{m}{32}\|x\|_2^2}\lrp{L+1}^{9/2} \lrp{\|x\|_2^6 + 1},
\end{align*}
where the first inequality is by Jensen's inequality, the triangle inequality
and the Cauchy-Schwarz inequality, the second inequality is by
Lemmas~\ref{l:s_logp^*issmooth}.3,
\ref{l:s_onestepdiscretizationbounds}.1,
and~\ref{c:s_determinantinversenaivenaivebound}, the third inequality
is by the fact that $p^*(x) \propto \exp\lrp{-U(x)}$, by
Assumption~\ref{ass:uissmooth}.2, and by~\eqref{e:s_he:1}
(we perform a first order Taylor expansion on $U(x)$),
the fourth inequality is by our assumption on $\delta$ and some
algebra, and the fifth inequality is by
our assumption on $\delta$.
\end{proof}

\begin{lemma}\label{l:s_onestepdiscretizationbounds}
For any $\delta \leq \frac{1}{16 L}$, for any $x,y$ such that $x = y - \delta \nabla U(y) + \sqrt{2\delta} T_\eta$ and for $\eta$ a.s.,
\begin{align*}
&1.\ \lrn{y-x}_2 \leq 4 \delta^{1/2} L^{1/2}\lrp{\|x\|_2 + 1},\\
&2.\ \lrn{y - x - \lrp{-\sqrt{2\delta} T_\eta + \delta \nabla U(x)}}_2
\leq 4\delta^{3/2} L^{3/2} \lrp{\|x\|_2 + 1},\\
&3.\ \lrn{y - x -\lrp{-\sqrt{2\delta} T_\eta(x)}}_2
\leq  2 \delta L\lrp{\|x\|_2 + 1}.
\end{align*}
\end{lemma}
\begin{proof}[Proof of Lemma~\ref{l:s_onestepdiscretizationbounds}]
$ $\\
\begin{enumerate}
\item \begin{align*}
\|y-x\|_2 
=& \lrn{\delta \nabla U(y) + \sqrt{2\delta} T_\eta}_2\\
\leq& \lrn{\delta \nabla U(x) + \sqrt{2\delta} T_\eta}_2 + \delta \lrn{\nabla U(y) - \nabla U(x)}_2\\
\leq& \lrn{\delta \nabla U(x) + \sqrt{2\delta} T_\eta(x)}_2 + \delta L \lrn{y-x}_2,
\end{align*}
where the first inequality is by triangle inequality, the second inequality is by Assumption \ref{ass:uissmooth}.2.\\
Moving terms around,
\begin{align*}
(1-\delta L) \|y-x\|_2 
\leq& \lrn{\delta \nabla U(x) + \sqrt{2\delta} T_\eta}_2\\
\leq& \delta L \|x\|_2 + \sqrt{2\delta L}\\
\Rightarrow \qquad\qquad\qquad\qquad \|y-x\|_2 \leq& \lrp{\delta L + \sqrt{\delta L}}\lrp{\|x\|_2 + 1}\\
\leq& 2 \delta^{1/2} L^{1/2}\lrp{\|x\|_2 + 1},
\end{align*}
where the second inequality is by Assumptions~\ref{ass:uissmooth}.2
and~\ref{ass:simplerassumptions}.2, and the third inequality is by our assumption on $\delta$.

\item 
\begin{align*}
y - \delta \nabla U(y)  + \sqrt{2\delta} T_\eta = & x\\
\Rightarrow \quad \lrn{y - x - \lrp{-\sqrt{2\delta} T_\eta + \delta \nabla U(x)}}_2
=& \delta \lrn{\nabla U(y) - \nabla U(x)}_2\\
\leq& \delta L \|y-x\|_2 \\
\leq& 4\delta^{3/2} L^{3/2} \lrp{\|x\|_2 + 1},
\end{align*}
where the first line is by definition of $x$ and $y$, the second line
is by Assumption~\ref{ass:uissmooth}.2, and the third line is by
Lemma~\ref{l:s_onestepdiscretizationbounds}.1.

\item 
\begin{align*}
\lrn{y - x - \lrp{-\sqrt{2\delta} T_\eta(x)}}_2
\leq& \lrn{y - x - \lrp{-\sqrt{2\delta} T_\eta + \delta \nabla U(x)}}_2 + \lrn{\delta \nabla U(x)}_2\\
\leq& 2 \delta^{3/2} L \lrp{\|x\|_2 + 1} + \delta L \|x\|_2\\
\leq& 2 \delta L \lrp{\|x\|_2 + 1},
\end{align*}
where the first line is by triangle inequality, the second line is by
Lemma~\ref{l:s_onestepdiscretizationbounds}.2 and
Assumption~\ref{ass:uissmooth}.2, and the third line is by our
assumption on $\delta$.
\end{enumerate}
\end{proof}

\begin{lemma}\label{l:s_auxfordeterminant}
For any $\delta \leq \frac{1}{16 L}$, for any $x,y$ such that $x = y - \delta \nabla U(y) + \sqrt{2\delta} T_\eta(y)$ and for $\eta$ a.s.,
\begin{align*}
\lrabs{\tr\lrp{\nabla^2 U(y)} - \tr\lrp{\nabla^2 U(x)}} \leq 4
\delta^{1/2} d L^{3/2}\lrp{\|x\|_2 + 1}.
\end{align*}
\end{lemma}
\begin{proof}[Proof of Lemma~\ref{l:s_auxfordeterminant}]
\begin{align*}
\lrabs{\tr\lrp{\nabla^2 U(y)} - \tr\lrp{\nabla^2 U(x)}}
=& \lrabs{\tr\lrp{\nabla^2 U(y) - \nabla^2 U(x)}}\\
\leq& d\lrn{\nabla^2 U(y) - \nabla^2 U(x)}_2\\
\leq& d L \lrn{x-y}_2\\
\leq& 4 \delta^{1/2} d L^{3/2}\lrp{\|x\|_2 + 1},
\end{align*}
where the first inequality is by Lemma~\ref{l:traceupperbound}, the
second inequality is by Assumption~\ref{ass:uissmooth}.4, and the third
inequality is by Lemma~\ref{l:s_onestepdiscretizationbounds}.1.
\end{proof}

\begin{lemma}\label{l:s_determinantexpansioninverse}
For any $\delta \leq \frac{1}{2Ld}$, for any $x$ and for $\eta$ a.s., 
\begin{align*}
\lrabs{\det\lrp{I - \lrp{{\delta} \nabla^2 U(x)}}^{-1} -\lrp{ 1 + \delta \tr\lrp{\nabla^2 U(x)}}}
\leq& 64\delta^2 d^2L^2.
\end{align*}
\end{lemma}
\begin{proof}[Proof of Lemma \ref{l:s_determinantexpansioninverse}]
First, let's consider an arbitrary symmetric matrix $A\in \Re^{2d}$, let $c$ be a constant such that $\|A\|_2 \leq c$ and let $\epsilon$ be a constant satisfying $\epsilon \leq 1/(2cd)$.
By Lemma~\ref{l:determinanttaylor}, we have
\begin{align*}
\det \lrp{I + \epsilon A} = 1 + \epsilon \tr\lrp{A} + \frac{\epsilon^2}{2} \lrp{\tr\lrp{A}^2-\tr\lrp{A^2}} + \Delta
\end{align*}
for some $|\Delta|\leq 2\epsilon^3 c^3 d^3$.
Using a Taylor expansion, we can verify that for any $a\in [-1/2,1/2]$
\begin{align*}
\lrabs{(1+a)^{-1} -\lrp{1-a + a^2}} \leq  |2a|^3.
\numberthis \label{e:s_re:1}
\end{align*}
By our assumption on $\epsilon$, we have $\epsilon \tr\lrp{A} + \frac{\epsilon^2}{2} \lrp{\tr\lrp{A}^2-\tr\lrp{A^2}} + \Delta \in [-1/2,1/2]$, therefore
\begin{align*}
\lefteqn{\lrp{1 + \epsilon \tr\lrp{A} + \epsilon^2/2
\lrp{\tr\lrp{A}^2-\tr\lrp{A^2}} + \Delta}^{-1}} &  \\
&\leq 1 - \epsilon \tr(A) - \epsilon^2/2 \lrp{\tr\lrp{A}^2-\tr\lrp{A^2}} - \Delta  \\
&\quad +\lrp{\epsilon \tr\lrp{A}+ \epsilon^2/2 \lrp{\tr\lrp{A}^2-\tr\lrp{A^2}} + \Delta}^2\\
&\quad + 2 \lrp{\epsilon \tr\lrp{A}+ \epsilon^2/2 \lrp{\tr\lrp{A}^2-\tr\lrp{A^2}} + \Delta}^3\\
&= 1 - \epsilon \tr(A) - \epsilon^2/2 \lrp{\tr\lrp{A}^2-\tr\lrp{A^2}} + \epsilon^2 \tr\lrp{A}^2\\
&\quad + \lrp{\epsilon^2/2 \lrp{\tr\lrp{A}^2-\tr\lrp{A^2}} + \Delta}\lrp{\epsilon \tr\lrp{A}+ \epsilon^2/2 \lrp{\tr\lrp{A}^2-\tr\lrp{A^2}} + \Delta}\\
&\quad + 2 \lrp{\epsilon \tr\lrp{A}+ \epsilon^2/2 \lrp{\tr\lrp{A}^2-\tr\lrp{A^2}} + \Delta}^3\\
&\leq 1 - \epsilon \tr(A) - \epsilon^2/2 \lrp{\tr\lrp{A}^2-\tr\lrp{A^2}} + \epsilon^2 \tr\lrp{A}^2 \\
& \quad + 4\lrp{\epsilon^2 c^2 d^2 + \epsilon^3 c^3 d^3}\lrp{\epsilon c d + \epsilon^2 c^2 d^2 + \epsilon^3 c^3 d^3} + 16 {\epsilon c d + \epsilon^2 c^2 d^2 + \epsilon^3 c^3 d^3}\\
&\leq 1 - \epsilon \tr(A) - \epsilon^2/2 \lrp{\tr\lrp{A}^2-\tr\lrp{A^2}} + \epsilon^2 \tr\lrp{A}^2 + 32 \lrp{\epsilon c d}^3 \\
&= 1 - \epsilon \tr(A) + \epsilon^2/2 \lrp{\tr\lrp{A}^2+\tr\lrp{A^2}}
+ 32 \lrp{\epsilon c d}^3,
\end{align*}
where the first inequality is by \eqref{e:s_re:1}, the first
equality is by moving terms around, the second inequality is by our
assumption that $\|A\|_2 \leq c$, by our assumption that
$\lrabs{\Delta}\leq 2\epsilon^3c^3d^3$, and by Lemma~\ref{l:traceupperbound}, and the last two lines are by collecting
terms. Conversely, one can show that
\begin{align*}
& \lrp{1 + \epsilon \tr\lrp{A} + \epsilon^2/2 \lrp{\tr\lrp{A}^2-\tr\lrp{A^2}} + \Delta}^{-1} \\
\geq& 1 - \epsilon \tr(A) + \epsilon^2/2
\lrp{\tr\lrp{A}^2+\tr\lrp{A^2}} - 32 \lrp{\epsilon c d}^3.
\end{align*}
The proof is similar and is omitted.

Therefore
\begin{align*}
\lrabs{\det\lrp{I + \epsilon A}^{-1} -\lrp{ 1 - \epsilon \tr(A) +
\epsilon^2/2 \lrp{\tr\lrp{A}^2+\tr\lrp{A^2}}}}\leq  32 \lrp{\epsilon c
d}^3.
\numberthis \label{e:s_re:2}
\end{align*}
Now, we consider the case that $A:= -\nabla^2 U(x)$, $\epsilon := {\delta}$ and $c:= {L}$. Recall our assumption that $\delta \leq \frac{1}{2 d L}$. Combined with Assumption \ref{ass:uissmooth}.2, we get
\begin{align*}
&1.\ \lrn{A}_2 \leq c, &
&2.\ \epsilon = {\delta} \leq {1}/\lrp{2Ld}= 1/(2cd).
\end{align*}
Using \eqref{e:s_re:2},
\begin{align*}
\det \lrp{I - \delta\lrp{\nabla^2 U(x) }}^{-1}
=:& \det\lrp{I+\epsilon A}^{-1}\\
=& 1 - \epsilon \tr(A) + \epsilon^2/2 \lrp{\tr\lrp{A}^2+\tr\lrp{A^2}} + 32 \lrp{\epsilon c d}^3\\
\leq& 1 + \delta \tr\lrp{\nabla^2 U(x)}\\
&\quad + \frac{\delta^2}{2} \tr\lrp{\nabla^2 U(x)}^2 + \frac{\delta^2}{2}\tr\lrp{\lrp{\nabla^2 U(x)}^2}\\
&\quad + 32\delta^3 d^3 L^3\\
\leq& 1 + \delta \tr\lrp{\nabla^2 U(x)} + 64\delta^2 d^2L^2,
\end{align*}
where the first inequality is by \eqref{e:s_re:2}, the first
inequality is by definition of $A$ and $\epsilon$, and the second
inequality is by Assumption \ref{ass:uissmooth}.2 and moving terms
around.

Conversely, one can show that
\begin{align*}
\det \lrp{I - \delta\lrp{\nabla^2 U(x) }}^{-1}
\geq& 1 + \delta \tr\lrp{\nabla^2 U(x)} - 64\delta^2 d^2L^2
\end{align*}
The proof is similar and is omitted.
\end{proof}

\begin{lemma}\label{l:s_discretizeddeterminantexpansioninverse}
For any $\delta \leq \frac{1}{64 d^2 L}$, for any $x$ and for $\eta$ a.s., 
\begin{align*}
\det\lrp{\nabla F_\eta(F_\eta^{-1} (x))}^{-1} = 1 + \delta \tr\lrp{\nabla^2 U(x)} + \Delta
\end{align*}
for some $\lrabs{\Delta}\leq 8 \delta^{3/2} dL^{3/2}\lrp{\|x\|_2 + 1}$.
\end{lemma}

\begin{proof}[Proof of Lemma \ref{l:s_discretizeddeterminantexpansioninverse}]
Consider the Jacobian matrix inside the determinant. By definition of $F_\eta$, we know that 
$$\nabla F_\eta\lrp{F_\eta^{-1}(x)} = I - \delta \nabla^2 U\lrp{F_\eta^{-1}(x)}.$$
Thus,
\begin{align*}
&\det\lrp{\nabla F_\eta(F_\eta^{-1}(x)}^{-1}\\
=& \det\lrp{I - \delta \nabla^2 U\lrp{F_\eta^{-1}(x)}}^{-1} \\
\leq& 1 + \delta \tr\lrp{\nabla^2 U(F_\eta^{-1}(x))} + 64 \delta^{2} d^{2} L^2\\
\leq& 1 + \delta \tr\lrp{\nabla^2 U(x)} + \delta \lrabs{\tr\lrp{\nabla^2 U(F_\eta^{-1}(x)} - \tr\lrp{\nabla^2 U (x)}} + 64 \delta^{2} d^{2} L^2\\
\leq& 1 + \delta\tr\lrp{\nabla^2 U(x)} + 4\delta^{3/2} d L^{3/2}\lrp{\|x\|_2 + 1} + 64\delta^2 d^2 L^2\\
\leq& 1 + \delta\tr\lrp{\nabla^2 U(x)}  + 8 \delta^{3/2} dL^{3/2}\lrp{\|x\|_2 + 1},
\end{align*}
where the first inequality is by Lemma~\ref{l:s_determinantexpansioninverse}, the second inequality is by the triangle inequality,
the third inequality is by Lemma~\ref{l:s_auxfordeterminant}, and the fourth inequality is by our assumption that $\delta \leq \frac{1}{64Ld^2}$.
Conversely, one can show that
\begin{align*}
\det\lrp{\nabla F_\eta(F_\eta^{-1}(x)}^{-1}
\geq& 1 + \delta\tr\lrp{\nabla^2 U(x)}  - 8 \delta^{3/2} dL^{3/2}\lrp{\|x\|_2 + 1}.
\end{align*}
The proof is similar and is omitted.
\end{proof}

\begin{corollary}\label{c:s_determinantinversenaivenaivebound}
For any $\delta \leq \frac{1}{8 Ld^2}$, for any $x$, and for $\eta$ a.s.,
\begin{align*}
\lrabs{\det\lrp{\nabla F_\eta(F_\eta^{-1} (x))}^{-1} - 1}
\leq&  2 \delta d L\lrp{\|x\|_2 + 1}.
\end{align*}
\end{corollary}
\begin{proof}[Proof of Corollary \ref{c:s_determinantinversenaivenaivebound}]
From Lemma~\ref{l:s_discretizeddeterminantexpansioninverse}, we get
\begin{align*}
\lrabs{\det\lrp{\nabla F_\eta(F_\eta^{-1} (x))}^{-1} - 1}
\leq& \lrabs{\det\lrp{\nabla F_\eta(F_\eta^{-1} (x))}^{-1} - \lrp{1 + \delta \tr\lrp{\nabla^2 U(x)}}} + \lrabs{\delta \tr\lrp{\nabla^2 U(x)}}\\
\leq& 8 \delta^{3/2} d L^{3/2}\lrp{\|x\|_2 + 1} + \delta d L \|x\|_2\\
\leq& 2 \delta d L\lrp{\|x\|_2 + 1},
\end{align*}
where the first inequality is by the triangle inequality, the second inequality is by Lemma~\ref{l:s_discretizeddeterminantexpansioninverse} and
Assumption~\ref{ass:uissmooth}.2, and the third inequality is by our assumption on $\delta$.
\end{proof}

\begin{lemma}\label{l:s_upperboundw2bychisquared}
Let $p^*(x) \propto e^{-U(x)}$, for any $q$ which is absolutely continuous wrt $p^*(x)$, 
\begin{align*}
W_2^2(p^*,q) 
\leq& \frac{2}{m} \int \lrp{\frac{q(x)}{p^*(x)} - 1}^2 p^*(x) \,dx.
\end{align*}
\end{lemma}

\begin{proof}[Proof of Lemma \ref{l:s_upperboundw2bychisquared}]
By Theorems~1 and~2 (Talagrand's Inequality) from \cite{otto2000generalization},
we see that if $p^*(x) \propto e^{-U(x)}$ for an $m$-strongly-convex $U(x)$ (Assumption~\ref{ass:uissmooth}.3), then for all $q$ absolutely continuous wrt $p$,
\begin{align*}
W_2^2(q,p^*) \leq \frac{2}{m}\KL{q}{p^*}.
\end{align*}
By the inequality $t\log t \leq t^2 -t$, we get
\begin{align*}
\KL{q}{p^*}
=& \int \frac{q(x)}{p^*(x)} \log \frac{q(x)}{p^*(x)} p^*(x) \,dx\\
\leq& \int \lrp{\lrp{\frac{q(x)}{p^*(x)}}^2 - \frac{q(x)}{p^*(x)}} p^*(x) \,dx\\
=& \int \lrp{\frac{q(x)}{p^*(x)} - 1}^2 p^*(x) \,dx.
\end{align*}
Combining the two inequalities, we get that
\begin{align*}
W_2^2(q,p^*)\leq \frac{2}{m} \int \lrp{\frac{q(x)}{p^*(x)} - 1}^2 p^*(x) \,dx.
\end{align*}
\end{proof}

\begin{lemma}
\label{l:s_logp^*issmooth}
For $p^*(x) \propto e^{-U(x)}$, and for any $x$, 
\begin{align*}
1.\ &\lrn{\nabla p^*(x)}_2\leq p^*(x) \cdot \lrp{L \|x\|_2},\\
2.\ &\lrn{\nabla^2 p^*(x)}_2 \leq p^*(x) \cdot \lrp{L + L^2 \|x\|_2^2},\\
3.\ &\lrn{\nabla^3 p^*(x)}_2\leq \lrp{L + 2 L^2\|x\|_2 + L^3 \|x\|_2^3}.
\end{align*}
\end{lemma}
\begin{proof}[Proof of Lemma~\ref{l:s_logp^*issmooth}]
$ $\\
\begin{enumerate}
\item
\begin{align*}
\lrn{\nabla p^*(x)}_2
=& \lrn{e^{-U(x)} \lrp{-\nabla U(x)}}_2\\
\leq& p^*(x) \cdot \lrp{L \|x\|_2},
\end{align*}
where the inequality is by Assumption~\ref{ass:uissmooth}.2.
\item
\begin{align*}
\lrn{\nabla^2 p^*(x)}_2
=& \lrn{e^{-U(x)} \lrp{-\nabla^2 U(x) + \nabla U(x) \nabla U(x)^T}}_2\\
\leq& p^*(x) \cdot \lrp{L + L^2 \|x\|_2^2},
\end{align*}
where the inequality is by Assumption~\ref{ass:uissmooth}.2.
\item
\begin{align*}
\lrn{\nabla^3 p^*(x)}
=& p^*(x) \lrn{-\nabla^3 U(x)  + \nabla^2 U(x) \otimes \nabla U(x) + \nabla U \otimes \nabla^2 U(x) - \nabla U(x) \otimes \nabla U(x) \otimes \nabla U(x)}_2\\
\leq& p^*(x) \lrp{L + 2 L^2\|x\|_2 + L^3 \|x\|_2^3},
\end{align*}
where the inequality is by Assumptions~\ref{ass:uissmooth}.2.~and~\ref{ass:uissmooth}.4.
\end{enumerate}
\end{proof}

\begin{lemma}\label{l:s_discrete_contraction}
For any $\delta \leq \frac{1}{2L}$ and for any distributions $p$ and $q$, under the assumptions of Section \ref{ss:mainresult_homogeneousnoise},
\begin{align*}
W_2(\Phi_\delta(p),\Phi_\delta(q))\leq& e^{-m\delta/8} W_2(p,q).
\end{align*}
\end{lemma}
\begin{proof}[Proof of Lemma \ref{l:s_discrete_contraction}]
Let $\gamma^*$ be an optimal coupling between $p$ and $q$, i.e.
\begin{align*}
W_2^2(p,q) = \Ep{\gamma^*(x,y)}{\|x-y\|_2^2}.
\end{align*}
We define a coupling $\gamma'$ as follows:
\begin{align*}
\gamma'(x,y):= \lrp{F_\eta, F_\eta}_{\#} \gamma^*,
\end{align*}
where $\#$ denotes the push-forward operator. (See \eqref{d:Feta} for the definition of $F_\eta$.)
It is thus true by definition that $\gamma'$ is a valid coupling between $\Phi_\delta(p)$ and $\Phi_\delta(q)$. 
Thus,
\begin{align*}
W_2(\Phi_\delta(p),\Phi_\delta(q))
\leq& \Ep{\gamma'(x,y)}{\|x-y\|_2^2}\\
=& \Ep{\gamma^*(x,y)}{\|F_\eta(x)-F_\eta(y)\|_2^2}\\
=& \Ep{\gamma^*(x,y)}{\|x-\delta \nabla U(x) + \sqrt{2\delta} T_\eta - \lrp{y-\delta \nabla U(y) + \sqrt{2\delta} T_\eta}\|_2^2}\\
\leq& \Ep{\gamma^*(x,y)}{\lrp{1-m\delta/2}\lrn{x-y}_2^2}\\
\leq& e^{-m\delta/4}\Ep{\gamma^*(x,y)}{\lrn{x-y}_2^2}\\
=& e^{-m\delta/4} W_2^2(p,q),
\end{align*}
where the second inequality follows from Assumptions~\ref{ass:uissmooth}.2 and~\ref{ass:uissmooth}.3 and our assumption that $\delta\leq \frac{1}{2L}$, and the third inequality is by the fact that $m\delta/2\leq m/(2L) \leq 1/2$.
\end{proof}
\end{section}

\begin{section}{Auxiliary Lemmas for Section \ref{ss:mainresult_inhomogeneousnoise}}
\label{s:appendix:inhomogeneous}

\begin{proof}[Proof of Theorem \ref{t:convergencerate}]
By Theorem \ref{t:main}, for 
\begin{align*}
\frac{1}{\delta} \geq& \max
\begin{cases}
2^8 d^2 L\\
2^{37} L\theta^2\\
2^{37}L\theta^2 \lrp{\frac{c_{\sigma}^2}{m}\log \frac{c_{\sigma}^2}{m}}^3\\
2^{72} L \theta^2 \frac{c_{\sigma}^6}{m^3}\log \lrp{\frac{2^{62} L c_{\sigma}^2}{m}}\\
{d^7}{\epsilon^{-2}} \cdot {2^{142} L^2 \lrp{\theta^3 + \theta^2 + \theta}^2}{\lambda^{-2}} \cdot \lrp{\frac{c_{\sigma}^2}{m} \log \frac{c_{\sigma}^2}{m}}^{12}\\
{d^7}{\epsilon^{-2}} \cdot {2^{142} L^2 \lrp{\theta^3 + \theta^2 + \theta}^2}{\lambda^{-2}}\\
{d^7}{\epsilon^{-2}} \cdot {2^{142} L^2 \lrp{\theta^3 + \theta^2 + \theta}^2}{\lambda^{-2}} \lrp{\frac{c_{\sigma}^2}{m} \log \lrp{2^{324} d^5 L \lrp{\theta^3 + \theta^2 + \theta}\lambda^{-6}\epsilon^{-6}}}^{12}
\end{cases}\\
=& \frac{d^7}{\epsilon^2} \lrp{poly \lrp{L, \theta, \frac{1}{m}, c_{\sigma},  \frac{1}{\lambda} }},
\end{align*}
we can bound
\begin{align*}
W_2(p_{k}, p^*) \leq& e^{-\lambda \delta k} W_2(p_0,p^*) + \frac{\epsilon}{2}.
\numberthis \label{e:in:1}
\end{align*}
To prove \eqref{e:t2:1}, take the limit of \eqref{e:in:1} as $k\to\infty$.

Next, if 
\begin{align*}
k \geq& \frac{1}{\lambda \delta}\log \frac{2W_2\lrp{p_0,p^*}}{\epsilon} \\
=& \frac{d^7}{\epsilon^2} \cdot \log \frac{W_2(p_0,p^*)}{\epsilon} poly \lrp{L, \theta, \frac{1}{m}, c_{\sigma},  \frac{1}{\lambda} },
\end{align*}
then $e^{-m \delta k/8} W_2(p_0,p^*) \leq \frac{\epsilon}{2}$, so we get
\begin{align*}
W_2\lrp{p_k,p^*}\leq \epsilon.
\end{align*}
This proves \eqref{e:t2:2}.
\end{proof}

\begin{theorem}\label{t:main}
Let $p_0$ be an arbitrary initial distribution, and let $p_{k\delta}$ be defined as in \eqref{e:discretesde}.\\
Let $\epsilon>0$ be some arbitrary constant.
For any stepsize $\delta$  satisfying
\begin{align*}
\frac{1}{\delta} \geq \max
\begin{cases}
2^8 d^2 L\\
2^{37} L\theta^2\\
2^{37}L\theta^2 \lrp{\frac{c_{\sigma}^2}{m}\log \frac{c_{\sigma}^2}{m}}^3\\
2^{72} L \theta^2 \frac{c_{\sigma}^6}{m^3}\log \lrp{\frac{2^{62} L c_{\sigma}^2}{m}}\\
{d^7}{\epsilon^{-2}} \cdot {2^{142} L^2 \lrp{\theta^3 + \theta^2 + \theta}^2}{\lambda^{-2}} \cdot \lrp{\frac{c_{\sigma}^2}{m} \log \frac{c_{\sigma}^2}{m}}^{12}\\
{d^7}{\epsilon^{-2}} \cdot {2^{142} L^2 \lrp{\theta^3 + \theta^2 + \theta}^2}{\lambda^{-2}}\\
{d^7}{\epsilon^{-2}} \cdot {2^{142} L^2 \lrp{\theta^3 + \theta^2 + \theta}^2}{\lambda^{-2}} \lrp{\frac{c_{\sigma}^2}{m} \log \lrp{2^{324} d^5 L \lrp{\theta^3 + \theta^2 + \theta}\lambda^{-6}\epsilon^{-6}}}^{12}
\end{cases}
\numberthis \label{e:sd:5}
\end{align*}
the Wasserstein distance between $p_{k}$ and $p^*$ is upper bounded as
\begin{align*}
W_2(p_{k}, p^*) \leq& e^{-\lambda \delta k} W_2(p_0,p^*) + \frac{\epsilon}{2}.
\end{align*}
\end{theorem}

\begin{proof}[Proof of Theorem \ref{t:main}]
We first use the triangle inequality to split the objective into two terms:
\begin{align*}
W_2(\Phi_{\delta}^k(p_0), p^*) \leq  W_2(\Phi_{\delta}^k(p_0), \Phi_{\delta}^k(p^*)) + W_2(\Phi_{\delta}^k(p^*), p^*)
\numberthis \label{e:ry:1}
\end{align*}

The first term is easy to bound. We use Assumption \ref{ass:discreteprocesscontracts} to get
\begin{align*}
W_2(\Phi_{\delta}^k(p_0), \Phi_{\delta}^k(p^*)) \leq e^{-\lambda\delta k} W_2(p_0,p^*)
\end{align*}

We now bound the second term of \eqref{e:ry:1}:

\begin{align*}
W_2(\Phi_{\delta}^k(p^*), p^*)
=& W_2(\Phi_{\delta}(\Phi_{\delta}^{k-1} (p^*)), p^*)\\
\leq& W_2(\Phi_{\delta}(\Phi_{\delta}^{k-1} (p^*)), \Phi_{\delta}(p^*)) + W_2(\Phi_{\delta}(p^*), p^*)\\
\leq&  e^{-\lambda \delta} W_2(\Phi_{\delta}^{k-1} (p^*), p^*) + W_2(\Phi_{\delta}(p^*), p^*)\\
\vdots&\\
\leq& \sum_{i=0}^{k-1} e^{-\lambda \delta i} W_2(\Phi_{\delta}(p^*), p^*)\\
\leq& \frac{1}{\lambda \delta}W_2(\Phi_{\delta}(p^*), p^*),
\numberthis\label{e:sd:1}
\end{align*}
where the first inequality is by triangle inequality, the second inequality is by Assumption \ref{ass:discreteprocesscontracts}.

Next, we apply Lemma~\ref{t:chisquaredbound} to get
\begin{align*}
& W_2(\Phi_{\delta}(p^*), p^*)\\
\leq& 2^{70}\delta^{3/2} d^{7/2} L \lrp{\theta^3 + \theta^2 + \theta} \max\lrbb{\frac{c_{\sigma}^2}{m} \log \frac{c_{\sigma}^2}{m}, \frac{c_{\sigma}^2}{m}\log
\lrp{\frac{1}{2^{124}d^6 L^2 \lrp{\theta^3 + \theta^2 + \theta} \delta^3 }}, 1}^{6}.
\end{align*}

Note that the first four clauses under \eqref{e:sd:5} satisfy the requirement of Lemma \ref{t:chisquaredbound}.  

There is a little trickiness due to the $\log\frac{1}{\delta}$ term in the above upper bound. The calculations to get rid of the $\log \frac{1}{\delta}$ term are packed away in Lemma \ref{l:reallyannoyinglemma2}. We verify that $\delta$ satisfies the conditions \eqref{e:iw:1} of Lemma \ref{l:reallyannoyinglemma2} as the last 3 clauses of \eqref{e:sd:5} implies,
\begin{align*}
\frac{1}{\delta} \geq
\frac{d^7}{\epsilon^2} \cdot \frac{2^{142} L^2 \lrp{\theta^3 + \theta^2 + \theta}^2}{\lambda^2} \cdot \max
\begin{cases}
\lrp{\frac{c_{\sigma}^2}{m} \log \frac{c_{\sigma}^2}{m}}^{12}\\
1\\
\lrp{\frac{c_{\sigma}^2}{m} \log \lrp{2^{324} d^5 L \lrp{\theta^3 + \theta^2 + \theta}\lambda^{-6}\epsilon^{-6}}}^{12},
\end{cases}
\end{align*}
Thus we can apply Lemma \ref{l:reallyannoyinglemma2} to get
\begin{align*}
&\frac{1}{\lambda \delta} W_2 (\Phi_{\delta}(p^*), p^*)\\
\leq& 2^{70}\delta^{1/2} d^{7/2} L \lrp{\theta^3 + \theta^2 + \theta} \max\lrbb{\frac{c_{\sigma}^2}{m} \log \frac{c_{\sigma}^2}{m}, \frac{c_{\sigma}^2}{m}\log \lrp{\frac{1}{2^{124}d^6 L^2 \lrp{\theta^3 + \theta^2 + \theta} \delta^3 }^2}, 1}^{6} \lambda^{-1}\\
\leq& \frac{\epsilon}{2}.
\numberthis \label{e:sd:3}
\end{align*}
The conclusion follows by substituting \eqref{e:sd:1} and \eqref{e:sd:3} into \eqref{e:ry:1}.

\end{proof}
\begin{lemma}\label{t:chisquaredbound}
Let $p_\delta := \Phi_\delta(p^*)$. For any $\delta$ satisfying
\begin{align*}
\frac{1}{\delta} \geq \max
\begin{cases}
2^8 d^2 L\\
2^{37} L\theta^2\\
2^{37}L\theta^2 \lrp{\frac{c_{\sigma}^2}{m}\log \frac{c_{\sigma}^2}{m}}^3\\
2^{72} L \theta^{2} \frac{c_{\sigma}^2}{m}\log \lrp{\frac{2^{62} L c_{\sigma}^2}{m}},
\end{cases}
\end{align*}
we have
\begin{align*}
W_2^2 (p_\delta, p^*) \leq 2^{140}\delta^3 d^7 L^2 \lrp{\theta^3 + \theta^2 + \theta}^2 \max\lrbb{\frac{c_{\sigma}^2}{m} \log \frac{c_{\sigma}^2}{m},
\frac{c_{\sigma}^2}{m}\log \lrp{\frac{1}{2^{124}d^6 L^2 \lrp{\theta^3 + \theta^2 + \theta}^2 \delta^3 }}, 1}^{11}.
\end{align*}
\end{lemma}
\begin{proof}[Proof of Lemma \ref{t:chisquaredbound}]
$ $\\
Let us define the radius
\begin{align*}
R:=& 2^7\sqrt{\max\lrbb{\frac{c_{\sigma}^2}{m} \log\frac{c_{\sigma}^2}{m}, {\frac{c_{\sigma}^2}{m}\log \lrp{\frac{1}{2^{124}d^6 L^2 \lrp{\theta^3 + \theta^2 + \theta}^2 \delta^3 }}}, 1}}
\end{align*}
We can verify that by the defintion of $R$ and our assumptions on $\delta$,
  \[
  R\geq \sqrt{\max\lrbb{2^{13}\frac{c_{\sigma}^2}{m} \lrp{\log \lrp{\frac{2^{11}c_{\sigma}^2}{m}}}, 1}}
  \]
and $\delta \leq \frac{1}{16 L}$, so we can apply Corollary \ref{c:upperboundw2bychisquared} to give
\begin{align*}
W_2^2\lrp{p^*,p_\delta}
\leq& 4R^2 \int_{B_R} \lrp{\frac{p_\delta(x)}{p^*(x)}-1}^2 p^*(x) dx + 84 d \exp\lrp{-\frac{mR^2}{64c_{\sigma}^2}}\\
\leq& 4R^2 \int_{B_R} \lrp{\frac{p_\delta(x)}{p^*(x)}-1}^2 p^*(x) dx + 2^{124}\delta^3 d^6 L^2 \lrp{\theta^3 + \theta^2 + \theta} ^2,
\numberthis \label{e:rc:1}
\end{align*}
where the second inequality follows from the definition of $R$, which implies that $R\geq \frac{c_{\sigma}^2}{m}\log \lrp{\frac{1}{2^{124}d^6 L^2 \lrp{\theta^3 + \theta^2 +
\theta} \delta^3}}$.

Next, we apply Lemma~\ref{l:reallyannoyinglemma}, which shows that under our assumptions on $\delta$ and our definition of $R$,
\begin{align*}
\delta \leq \min\lrbb{\frac{1}{2^8 d^2 L}, \frac{1}{2^{15}L \theta^2 \lrp{R^6 + 1}}}.
\end{align*}
We can thus apply Lemma~\ref{l:pdeltaoverpstariscloseto1} to get
\begin{align*}
&\int_{B_R} \lrp{\frac{p_\delta(x)}{p^*(x)}-1}^2 p^*(x) dx\\
\leq& 2^{30}\delta^{3} d^6 L^2 \lrp{\theta^3 + \theta^2 + \theta}^2 \int_{B_R} \lrp{\|x\|_2^{22}+1} p^*(x) dx\\
\leq& 2^{30}\delta^{3} d^6 L^2 \lrp{\theta^3 + \theta^2 + \theta}^2 \lrp{\max\lrbb{\lrp{640 \frac{c_{\sigma}^2}{m}\log \lrp{\frac{160 c_{\sigma}^2}{m}}}^{10},1280 d \frac{c_{\sigma}^2}{m}} + 1}\\
\leq& 2^{124}\delta^3 d^7 L^2 \lrp{\theta^3 + \theta^2 + \theta}^2 \max\lrbb{\frac{c_{\sigma}^2}{m} \log \frac{c_{\sigma}^2}{m}, 1}^{10},
\numberthis \label{e:rc:2}
\end{align*}
where the second inequality follows from Lemma~\ref{l:kthmomentbound}.

Plugging the above into \eqref{e:rc:1}, we get
\begin{align*}
&\eqref{e:rc:1}\\
=& 4R^2 \lrp{2^{124}\delta^3 d^7 L^2 \lrp{\theta^3 + \theta^2 + \theta}^2 \max\lrbb{\frac{c_{\sigma}^2}{m} \log \frac{c_{\sigma}^2}{m}, 1}^{10}} + \lrp{2^{124}\delta^3 d^7 L^2 \lrp{\theta^3 + \theta^2 + \theta}^2}\\
\leq& 2^{126}\delta^3 d^7 L^2 \lrp{\theta^3 + \theta^2 + \theta}^2 \max\lrbb{\frac{c_{\sigma}^2}{m} \log \frac{c_{\sigma}^2}{m}, 1}^{10} \cdot R^2\\
\leq& 2^{140}\delta^3 d^7 L^2 \lrp{\theta^3 + \theta^2 + \theta}^2 \max\lrbb{\frac{c_{\sigma}^2}{m} \log \frac{c_{\sigma}^2}{m}, \frac{c_{\sigma}^2}{m}\log
\lrp{\frac{1}{2^{124}d^6 L^2 \lrp{\theta^3 + \theta^2 + \theta}^2 \delta^3 }}, 1}^{11},
\end{align*}
where the first line is by \eqref{e:rc:2} and \eqref{e:rc:1}, the second line is because $R\geq 1$, the third line is again by definition of $R$ and some algebra.
\end{proof}

\begin{lemma}\label{l:pdeltaoverpstariscloseto1}
Let $p_\delta := \Phi_\delta(p^*)$. For any $R\geq 0$, for all $x\in B_R$, and for all $\delta \leq \min\lrbb{\frac{1}{2^8 d^2 L}, \frac{1}{2^{15}\theta^2 \lrp{R^6 + 1} L}}$
\begin{align*}
\lrabs{\frac{p_\delta(x)}{p^*(x)}-1}\leq 2^{15}\delta^{3/2} d^3 L^{3/2} \lrp{\theta^3 + \theta^2 + \theta} \lrp{\|x\|_2^{11}+1}.
\end{align*}
\end{lemma}

\begin{proof}[Proof of Lemma~\ref{l:pdeltaoverpstariscloseto1}]
By the definition~\eqref{d:Phip}, $p_\delta = \Phi_\delta(p^*) = \lrp{F_\eta}_{\#} p^*$.
The change of variable formula gives
\begin{align*}
p_\delta(x) 
=& \int p^*(F_\eta^{-1}(x))  \det\lrp{\nabla F_\eta\lrp{F_\eta^{-1} (x)}}^{-1}  q(\eta)d\eta\\
=&\Ep{q(\eta)}{\underbrace{p^*(F_\eta^{-1}(x))}_{\circled{1}}  \underbrace{\det\lrp{\nabla F_\eta\lrp{F_\eta^{-1} (x)}}^{-1}}_{\circled{2}} },
\numberthis \label{e:pdeltaintegralexpression}
\end{align*}
where in the above, $\nabla F_\eta(y)$ denotes the Jacobian matrix of $F_\eta$ at $y$. The invertibility of $F_\eta$ is proven in Lemma \ref{l:fisinvertible}.
We now rewrite $\circled{1}$ as its Taylor expansion:
\begin{align*}
& p^*\lrp{F_\eta^{-1}(x_\delta)}\\
=& \underbrace{p^*(x)}_{\circled{4}} +  \underbrace{{\lin{\nabla p^*(x), F_\eta^{-1}(x) - x}}}_{\circled{5}}\\
&\quad\quad\ \ +\frac{1}{2}\underbrace{\lin{\nabla^2 p^*(x), \lrp{F_\eta^{-1}(x) -x}\lrp{F_\eta^{-1}(x) -x}^T}}_{\circled{6}} \\
&\quad\quad\quad+ \underbrace{ \int_0^1 \int_0^t \int_0^s \lin{\nabla^3 p^*\lrp{(1-r) x + r F_\eta^{-1}(x)}, \lrp{F_\eta^{-1}(x) - x}^{3}}\,dr \,ds \,dt }_{\circled{7}}.
\end{align*}
Putting everything together, we get
\begin{align*}
p_\delta(x)
=& \Ep{\eta}{ \lrp{\circled{4}+\circled{5}+\circled{6}+\circled{7}} \cdot \circled{2}}\\
=& p^*(x)+ \delta p^*(x)\lrp{\sum_{i=1}^d \sum_{j=1}^d \frac{\del^2}{\del x_i \del x_j} \lrb{\sigma_{x}\sigma_{x}^T}_{i,j} + \delta \tr\lrp{\nabla^2 U(x)}}\\
&\quad + \delta \lrp{\sum_{i=1}^d \dd{x_i} p^*(x) \cdot \dd{x_i} U(x)} \\
&\quad + 2\delta {\sum_{i=1}^{d}\lrp{\dd{x_i} p^*(x)  \lrp{\sum_{j=1}^d\dd{x_j} \lrb{\sigma_{x}\sigma_{x}^T}_{i,j}}}}\\
&\quad +  \delta \lin{\nabla^2p^*(x), \sigma_{x}\sigma_{x}^T}\\
&\quad + \Delta\\
=& p^*(x) + \Delta
\numberthis \label{e:fk:2}
\end{align*}
The third equality is by Lemma~\ref{l:fokkerplanckcharacterizationofstationarydistribution}.
The second equality is by Lemmas~\ref{l:circled4times2}, \ref{l:circled5times2}, \ref{l:circled6times2} and~\ref{l:circled7times2}. Note that by our assumption that $x\in
B_R$ and $\delta \leq \min\lrbb{\frac{1}{2^8d^2L}, \frac{1}{2^{15}\theta^2 \lrp{R^6 + 1} L}}$, $\delta$ satisfies the condition for Lemmas~\ref{l:circled4times2},
\ref{l:circled5times2}, \ref{l:circled6times2} and~\ref{l:circled7times2}.
Also by these four lemmas, we have 
\begin{align*}
\lrabs{\Delta} 
\leq& p^*(x) \cdot 128 \delta^{3/2} d^3 L^{3/2}\lrp{\|x\|_2^2 + 1}\\
&\quad + p^*(x) \cdot 256 \delta^{3/2} d^2  L^{3/2}\theta\lrp{\|x\|_2^6 + 1}\\
&\quad + p^*(x)\cdot 256 \delta^{3/2} d L^{3/2} \lrp{\theta^2 + \theta} \lrp{\|x\|_2^{10} + 1}\\
&\quad + p^*(x) \cdot 2^{14} \delta^{3/2} L^{3/2}\lrp{\theta^3+\theta^2+\theta}\lrp{\|x\|_2^{11} + 1}\\
\leq& p^*(x) \cdot 2^{15}\delta^{3/2} d^3 L^{3/2} \lrp{\theta^3 + \theta^2 + \theta} \lrp{\|x\|_2^{11}+1}.
\end{align*}
As a consequence,
\begin{align*}
\lrabs{\frac{p_\delta(x)}{p^*(x)} - 1}\leq 2^{15}\delta^{3/2} d^3 L^{3/2} \lrp{\theta^3 + \theta^2 + \theta} \lrp{\|x\|_2^{11}+1}.
\end{align*}
\end{proof}

\begin{lemma}\label{l:circled4times2}
For $\delta \leq \frac{1}{2^8 d^2 L}$,
\begin{align*}
& \Ep{q(\eta)}{p^*(x) \cdot \det\lrp{\nabla F_\eta\lrp{F_\eta^{-1}(x)}}}\\
=& p^*(x) + p^*(x) \delta \sum_{i=1}^d \sum_{j=1}^d \frac{\del^2}{\del x_i \del x_j} \lrb{\sigma_{x}\sigma_{x}^T}_{i,j} + \delta \tr\lrp{\nabla^2 U(x)}+ \Delta,
\end{align*}
for some $\lrabs{\Delta} \leq p^*(x) \cdot 128 \delta^{3/2} d^3 L^{3/2}\lrp{\|x\|_2^2 + 1}$.
\end{lemma}
\begin{proof}[Proof of Lemma \ref{l:circled4times2}]
Let us define 
\begin{align*}
\Delta' :=& \det\lrp{\nabla F_\eta(F_\eta^{-1} (x))}^{-1} \\*
&\quad {} -\left(1 - \sqrt{2\delta} \tr\lrp{G_\eta(x)} + 2\delta \tr\lrp{\lint{M_\eta(x),T_\eta(x)}}\right. \\*
&\qquad\qquad \left.{} + \delta \tr\lrp{\nabla^2 U(x)} + \delta \tr\lrp{G_\eta(x)}^2 + \delta \tr\lrp{\lrp{G_\eta(x)}^2}\right).
\end{align*}
By Lemma~\ref{l:discretizeddeterminantexpansioninverse}, $\lrabs{\Delta'} \leq 128 \delta^{3/2} d^3 L^{3/2}\lrp{\|x\|_2^2 + 1}$. Hence,
\begin{align*}
& \Ep{q(\eta)}{p^*(x) \cdot \det\lrp{\nabla F_\eta\lrp{F_\eta^{-1}(x)}}}\\
= & \Ep{q(\eta)}{p^*(x) \cdot \lrp{1 - \sqrt{2\delta} \tr\lrp{G_\eta(x)} + 2\delta \tr\lrp{\lint{M_\eta(x),T_\eta(x)}}
  + \delta \tr\lrp{\nabla^2 U(x)} + \delta \tr\lrp{G_\eta(x)}^2 + \delta \tr\lrp{\lrp{G_\eta(x)}^2}}}\\
& \quad + \Ep{q(\eta)}{p^*(x) \cdot \Delta'}\\
= & p^*(x) + p^*(x) \Ep{q(\eta)}{ \lrp{- \sqrt{2\delta} \tr\lrp{G_\eta(x)} + 2\delta \tr\lrp{\lint{M_\eta(x),T_\eta(x)}} + \delta \tr\lrp{G_\eta(x)}^2 + \delta \tr\lrp{\lrp{G_\eta(x)}^2}}}\\
& \quad + p^*(x) \delta \tr\lrp{\nabla^2 U(x)} + p^*(x) \cdot \Delta'\\
=& p^*(x) + p^*(x)\cdot \delta \sum_{i=1}^d \sum_{j=1}^d \frac{\del^2}{\del x_i \del x_j} \lrb{\sigma_x\sigma_x^T}_{i,j} + \delta \tr\lrp{\nabla^2 U(x)}
+ p^*(x) \cdot \Delta'.
\end{align*}
We complete the proof by taking $\Delta := p^*(x) \Delta'$.
\end{proof}

\begin{lemma}\label{l:circled5times2}
For $\delta \leq \frac{1}{2^8 d^2 L}$,
\begin{align*}
& \Ep{q(\eta)}{\lin{\nabla p^*(x), F_\eta^{-1}(x) - x}\cdot \det\lrp{\nabla F_\eta(F_\eta^{-1} (x))}^{-1}}\\
=& \delta \sum_{i=1}^d \lrp{\dd{x_i} p^*(x) \cdot \lrp{\dd{x_i}U(x) + 2\sum_{j=1}^d\dd{x_j} \lrb{\sigma_{x}\sigma_{x}^T}_{i,j}}} + \Delta
\end{align*}
for some $\lrabs{\Delta}\leq  p^*(x) \cdot 256 \delta^{3/2} d^2  L^{3/2}\theta\lrp{\|x\|_2^6 + 1}$.
\end{lemma}
\begin{proof}[Proof of Lemma \ref{l:circled5times2}]
Let 
\begin{align*}
& \Delta_1 := F_\eta^{-1}(x) - x - \lrp{-\sqrt{2\delta} T_\eta(x) + \delta \nabla U(x) + 2\delta G_\eta(x) T_\eta(x)},\\
& \Delta_2 := \det\lrp{\nabla F_\eta(F_\eta^{-1} (x))}^{-1} - 1,\\
& \Delta_3 := \det\lrp{\nabla F_\eta(F_\eta^{-1} (x))}^{-1} - \lrp{1-\sqrt{2\delta} \tr\lrp{G_\eta(x)}}.
\end{align*}
By Lemma~\ref{l:onestepdiscretizationbounds}.2, Corollary~\ref{c:determinantinversenaivenaivebound} and Corollary~\ref{c:determinantinversenaivebound},
\begin{align*}
&\lrn{\Delta_1}_2 \leq 16 \delta^{3/2} L^{3/2} \lrp{\|x\|_2^2 + 1}, \\
&\lrabs{\Delta_2} \leq  2\delta^{1/2} d L^{1/2} \lrp{\|x\|_2^2 + 1}, \\
&\lrabs{\Delta_3} \leq 8 \delta d^2 L \lrp{\|x\|_2^2 + 1}.
\end{align*}
Moving terms around, 
\begin{align}
\nonumber
& \Ep{q(\eta)}{\lin{\nabla p^*(x), F_\eta^{-1}(x) - x}\cdot \det\lrp{\nabla F_\eta(F_\eta^{-1} (x))}^{-1}}\\
\label{e:fx:1}
=& \Ep{q(\eta)}{\lin{\nabla p^*(x), -\sqrt{2\delta}T_\eta(x)}\cdot (1-\sqrt{2\delta}\tr\lrp{G_\eta(x)}}\\
\label{e:fx:1.5}
&\quad + \Ep{q(\eta)}{\lin{\nabla p^*(x), \delta \nabla U(x) + 2\delta G_\eta(x)T_\eta(x)}}\\
\label{e:fx:2}
&\quad + \Ep{q(\eta)}{\lin{\nabla p^*(x), -\sqrt{2\delta}T_\eta(x)}\cdot \Delta_3}\\
\label{e:fx:3}
&\quad + \Ep{q(\eta)}{\lin{\nabla p^*(x), \delta \nabla U(x) + 2\delta G_\eta(x)T_\eta(x)}\cdot \Delta_2}\\
\label{e:fx:4}
&\quad + \Ep{q(\eta)}{\Delta_1\cdot \det\lrp{\nabla F_\eta(F_\eta^{-1} (x))}^{-1}}.
\end{align}
The main term of interest are \eqref{e:fx:1} and \eqref{e:fx:1.5}, which evaluate to
\begin{align*}
& \Ep{q(\eta)}{\lin{\nabla p^*(x), -\sqrt{2\delta}T_\eta(x)}\cdot (1-\sqrt{2\delta}\tr\lrp{G_\eta(x)}} + \Ep{q(\eta)}{\lin{\nabla p^*(x), \delta \nabla U(x) +2\delta G_\eta(x)T_\eta(x)}}\\
=& \Ep{q(\eta)}{\lin{\nabla p^*(x), -\sqrt{2\delta}T_\eta(x)}\cdot (-\sqrt{2\delta}\tr\lrp{G_\eta(x)}} + \Ep{q(\eta)}{\lin{\nabla p^*(x), \delta \nabla U(x) +2\delta G_\eta(x)T_\eta(x)}}\\
=& \lin{\nabla p^*(x),  \delta \nabla U(x) + 2\delta\Ep{q(\eta)}{\tr\lrp{G_\eta(x)}T_\eta(x)} + 2\delta \Ep{q(\eta)}{G_\eta(x)T_\eta(x)}}\\
=& \delta \sum_{i=1}^d \lrp{\dd{x_i} p^*(x) \cdot \lrp{\dd{x_i}U(x) + 2\sum_{j=1}^d\dd{x_j} \lrb{\sigma_{x}\sigma_{x}^T}_{i,j}}},
\end{align*}
where the first equality is by Assumption~\ref{ass:ximeanandvariance}.1, and the last equality is by Lemma~\ref{l:sigmaderivative}.
We now consider the terms~\eqref{e:fx:2}, \eqref{e:fx:3} and \eqref{e:fx:4}:
\begin{align*}
\lrabs{\eqref{e:fx:2}}
\leq& \lrabs{\|\nabla p^*(x)} \sqrt{2\delta}\Ep{q(\eta)}{\lrabs{T_\eta(x)}\lrabs{\Delta_3}}\\
\leq& p^*(x) \theta \lrp{\|x\|_2^2+1} \cdot \sqrt{2\delta L} \lrp{\|x\|_2+1}\cdot 8 \delta d^2 L \lrp{\|x\|_2^2 + 1}\\
\leq& 16 \delta^{3/2} p^*(x) d^2  L^{3/2} \theta\lrp{\|x\|_2^5 + 1},
\end{align*}
where the first inequality is by Cauchy-Schwarz, and the second inequality is by Lemma~\ref{l:logp^*issmooth}.1 and our upperbound on $\lrabs{\Delta_3}$ at the start of the proof.
\begin{align*}
\lrabs{\eqref{e:fx:3}}
\leq& \lrn{\nabla p^*(x)}_2 \delta \Ep{q(\eta)}{\lrn{\nabla U(x) + 2G_\eta(x)T_\eta(x)}_2\cdot \lrabs{\Delta_2}}\\
\leq& p^*(x) \theta \lrp{\|x\|_2^2+1} \cdot 3 \delta L \lrp{\|x\|_2 + 1} \cdot 2\delta^{1/2} d L^{1/2} \lrp{\|x\|_2^2 + 1} \\
\leq& 32 \delta^{3/2}p^*(x) d  L^{3/2}\theta\lrp{\|x\|_2^5 + 1},
\end{align*}
where the first inequality is by Cauchy-Schwarz, and the second inequality is by Lemma~\ref{l:logp^*issmooth}.1 and our upperbound on $\lrabs{\Delta_2}$ at the start of the proof.
\begin{align*}
\lrabs{\eqref{e:fx:4}}
\leq& \lrn{\nabla p^*(x)}_2 \cdot \Ep{q(\eta)}{\lrabs{\Delta_1} \cdot \lrabs{\det\lrp{\nabla F_\eta(F_\eta^{-1} (x))}^{-1}}}\\
\leq& p^*(x) \theta \lrp{\|x\|_2^2+1} \cdot \lrp{16 \delta^{3/2} L^{3/2} \lrp{\|x\|_2^2 + 1}}\cdot \lrp{1 + 2\delta^{1/2} d L^{1/2} \lrp{\|x\|_2^2 + 1}}\\
\leq& 128 \delta^{3/2} p^*(x)  L^{3/2}\theta\lrp{\|x\|_2^6 + 1},
\end{align*}
where the first inequality is by Cauchy-Schwarz, and the second inequality is by Lemma~\ref{l:logp^*issmooth}.1 and our upperbound on $\lrn{\Delta_1}_2$ and $\lrabs{\Delta_2}$ at the start of the proof.\\
Defining $\Delta := \eqref{e:fx:2}+\eqref{e:fx:3}+\eqref{e:fx:4}$, we have
\begin{align*}
\lrabs{\Delta}
\leq& 16 \delta^{3/2} p^*(x) d^2  L^{3/2} \theta\lrp{\|x\|_2^5 + 1}\\
&\quad + 32 \delta^{3/2}p^*(x) d  L^{3/2}\theta\lrp{\|x\|_2^5 + 1}\\
&\quad + 128 \delta^{3/2} p^*(x) \theta L^{3/2}\lrp{\|x\|_2^6 + 1}\\
\leq& p^*(x) \cdot 256 \delta^{3/2} d^2  L^{3/2}\theta\lrp{\|x\|_2^6 + 1}.
\end{align*}
\end{proof}
\begin{lemma}\label{l:circled6times2}
For $\delta \leq \frac{1}{2^8 d^2 L}$,
\begin{align*}
& \frac{1}{2} \Ep{q(\eta)}{\lin{\nabla^2 p^*(x), \lrp{F_\eta^{-1}(x) - x}\lrp{F_\eta^{-1}(x) - x}^T}\cdot \det\lrp{\nabla F_\eta(F_\eta^{-1} (x))}^{-1}}\\
=& \delta \lin{\nabla^2 p^*(x), \sigma_x \sigma_x^T} + \Delta
\end{align*}
for some $\lrabs{\Delta}\leq p^*(x)\cdot 256 \delta^{3/2} d L^{3/2} \lrp{\theta^2 + \theta} \lrp{\|x\|_2^{10} + 1}$.
\end{lemma}
\begin{proof}[Proof of Lemma \ref{l:circled6times2}]
Define
\begin{align*}
&\Delta_1 := F_\eta^{-1}(x) - x - \lrp{-\sqrt{2\delta}T_\eta(x)}, &
&\Delta_2 := \det\lrp{\nabla F_\eta(F_\eta^{-1} (x))}^{-1} -  1.
\end{align*}
By Lemma~\ref{l:onestepdiscretizationbounds}.3 and Corollary \ref{c:determinantinversenaivenaivebound},
\begin{align*}
\lrabs{\Delta_1} \leq&  16 \delta L \lrp{\|x\|_2^2 + 1}&
\lrabs{\Delta_2} \leq&  2\delta^{1/2} d L^{1/2} \lrp{\|x\|_2^2 + 1}
\end{align*}
Then
\begin{align}
\nonumber
& \Ep{q(\eta)}{\lin{\nabla^2 p^*(x), \lrp{F_\eta^{-1}(x) - x}\lrp{F_\eta^{-1}(x) - x}^T}\cdot \det\lrp{\nabla F_\eta(F_\eta^{-1} (x))}^{-1}}\\
\label{e:ta:1}
=& 2\delta \Ep{q(\eta)}{\lin{\nabla^2 p^*(x), T_\eta(x) T_\eta(x)^T}}\\
\label{e:ta:2}
&\quad + 2\delta \Ep{q(\eta)}{\lin{\nabla^2 p^*(x), T_\eta(x) T_\eta(x)^T}\cdot \Delta_2}\\
\label{e:ta:3}
&\quad + \Ep{q(\eta)}{\lin{\nabla^2 p^*(x), \Delta_1 \Delta_1^T - \sqrt{2\delta}T_\eta(x) \Delta_1^T - \sqrt{2\delta}\Delta_1 T_\eta(x) ^T}\cdot \det\lrp{\nabla
F_\eta(F_\eta^{-1} (x))}^{-1}}.
\end{align}
We are mainly interested in \eqref{e:ta:1}, which evaluates to 
\begin{align*}
& 2\delta \Ep{q(\eta)}{\lin{\nabla^2 p^*(x), T_\eta(x) T_\eta(x)^T}}\\
=& 2\delta \lin{\nabla^2 p^*(x), \Ep{q(\eta)}{T_\eta(x) T_\eta(x)^T}}\\
=& 2\delta \lin{\nabla^2 p^*(x), \sigma_x \sigma_x^T},
\end{align*}
where the last equality is by definition of $T_\eta(x)$ and $\sigma_x$.
We now bound the magnitude of~$\eqref{e:ta:2}$ and~$\eqref{e:ta:3}$.
\begin{align*}
\lrabs{\eqref{e:ta:2}}
=& \lrabs{2\delta \Ep{q(\eta)}{\lin{\nabla^2 p^*(x), T_\eta(x) T_\eta(x)^T}\cdot \Delta_2}}\\
\leq& 2\delta \lrn{\nabla^2 p^*(x)}_2 \Ep{q(\eta)}{\lrn{T_\eta(x)}_2^2 \lrabs{\Delta_2}}\\
\leq& 4\delta p^*(x) \lrp{\theta^2 +\theta}\lrp{\|x\|_2^4 +1} \cdot L \lrp{\|x\|_2^2 + 1}\cdot 2\delta^{1/2} d L^{1/2} \lrp{\|x\|_2^2 + 1}\\
\leq& 32 \delta^{3/2} p^*(x) d L^{3/2} (\theta + \theta^2) \lrp{\|x\|_2^8 + 1},
\end{align*}
where the first inequality is by Cauchy-Schwarz, and the second inequality is by Lemma~\ref{l:logp^*issmooth}.2 and our upper bound on $\lrabs{\Delta_2}$ at the start of the proof.
\begin{align*}
\lrabs{\eqref{e:ta:3}}
=& \Ep{q(\eta)}{\lin{\nabla^2 p^*(x), \Delta_1 \Delta_1^T + \sqrt{2\delta}T_\eta(x) \Delta_1^T + \sqrt{2\delta}\Delta_1 T_\eta(x) ^T}\cdot \det\lrp{\nabla F_\eta(F_\eta^{-1} (x))}^{-1}}\\
\leq& \lrn{\nabla^2 p^*(x)}_2 \Ep{q(\eta)}{\lrp{\lrn{\Delta_1}_2^2 + 2\sqrt{2\delta}\lrn{T_\eta(x)}_2\lrn{\Delta_1}_2} \lrabs{\det\lrp{\nabla F_\eta(F_\eta^{-1} (x))}^{-1}}}\\
\leq& p^*(x) \lrp{\theta^2 +\theta}\lrp{\|x\|_2^4 +1} \cdot \lrp{\lrp{16 \delta L \lrp{\|x\|_2^2 + 1}}^2 + 2\sqrt{2\delta} \lrp{L^{1/2} \lrp{\|x\|_2 + 1}}\lrp{16 \delta L \lrp{\|x\|_2^2 + 1}}}\\
&\quad \cdot \lrp{1 + 2\delta^{1/2} d L^{1/2} \lrp{\|x\|_2^2 + 1}}\\
\leq& 256 \delta^{3/2} p^*(x) d L^{3/2} \lrp{\theta^2 + \theta} \lrp{\|x\|_2^{10} + 1},
\end{align*}
where the first inequality is by Cauchy-Schwarz, and the second inequality is by Lemma~\ref{l:logp^*issmooth}.2 and our upper bound on $\lrabs{\Delta_1}$ at the start of the proof.\\
Defining $\Delta := \eqref{e:ta:2} + \eqref{e:ta:3}$, we have
\begin{align*}
\lrabs{\Delta} 
\leq& 32 \delta^{3/2} p^*(x) d L^{3/2} (\theta + \theta^2) \lrp{\|x\|_2^8 + 1} + 256 \delta^{3/2} p^*(x) d L^{3/2} \lrp{\theta^2 + \theta} \lrp{\|x\|_2^{10} + 1}\\
\leq& 512 \delta^{3/2} p^*(x) d L^{3/2} (\theta^2 + \theta)\lrp{\|x\|_2^{10} + 1}.
\end{align*}
\end{proof}
\begin{lemma}\label{l:circled7times2}
For $\delta \leq \min\lrbb{\frac{1}{2^8 d ^2 L}, \frac{1}{2^{15}\lrp{\|x\|_2^6 + 1} \theta^2 L}}$,
\begin{align*}
& \lrabs{\Ep{q(\eta)}{\lrp{ \int_0^1 \int_0^t \int_0^s \lin{\nabla^3 p^*\lrp{(1-r) x + r F_\eta^{-1}(x)}, \lrp{F_\eta^{-1}(x) - x}^{3}}dr ds dt }\cdot \det\lrp{\nabla F_\eta(F_\eta^{-1} (x))}^{-1}}}\\
\leq & p^*(x) \cdot 2^{14} \delta^{3/2} L^{3/2}\lrp{\theta^3+\theta^2+\theta}\lrp{\|x\|_2^{11} + 1}.
\end{align*}
\end{lemma}
\begin{proof}[Proof of Lemma \ref{l:circled7times2}]
Using Lemma~\ref{l:onestepdiscretizationbounds}.1, by our choice of $\delta$, $\lrn{x-F_\eta^{-1}(x)}_2\leq \frac{1}{2}\lrp{\|x\|_2 + 1}$, thus $\lrn{F_\eta^{-1}(x)} \leq
2\|x\|_2 + 1$. Hence, for all $t\in[0,1]$,
\begin{align}
\label{e:he:1}
\lrn{(1-t)x + t F_\eta^{-1}(x)}_2 \leq 2\|x\|_2 + 1,
\end{align}
and
\begin{align*}
&\!\!\!\!\!\! \lrabs{\Ep{q(\eta)}{\lrp{ \int_0^1 \int_0^t \int_0^s \lin{\nabla^3 p^*\lrp{(1-t) x + t F_\eta^{-1}(x)}, \lrp{F_\eta^{-1}(x) - x}^{3}}dr ds dt }\cdot \det\lrp{\nabla F_\eta(F_\eta^{-1} (x))}^{-1}}}\\
\leq& \Ep{q(\eta)}{ \int_0^1 \int_0^t \int_0^s \lrn{\nabla^3 p^*\lrp{(1-t) x + t F_\eta^{-1}(x)}_2} dr ds dt dt\cdot \lrn{F_\eta^{-1}(x)-x}_2^3\cdot \lrabs{\det\lrp{\nabla F_\eta(F_\eta^{-1} (x))}^{-1}}}\\
\leq& \Ep{q(\eta)}{2p^*\lrp{(1-t) x + tF_\eta^{-1}(x)} \lrp{\theta^3 + \theta^2 + \theta}\lrp{\lrn{(1-t) x + tF_\eta^{-1}(x)}_2^6+1}}\\
&\quad \cdot \lrp{2 \delta^{1/2} L^{1/2}\lrp{\|x\|_2 + 1}}^3\cdot \lrp{1 + 2\delta^{1/2} d L^{1/2} \lrp{\|x\|_2^2 + 1}}\\
\leq& p^*(x)\exp\lrp{2\theta\lrp{\|x\|_2^2 + \lrn{F_\eta^{-1}(x)}_2^2}\lrn{F_\eta^{-1}(x)-x}_2}\cdot 2^{13} \delta^{3/2}  L^{3/2} \lrp{\theta^3+\theta^2+\theta} \lrp{\|x\|_2^{11} + 1}\\
\leq& p^*(x)\exp\lrp{8\theta\lrp{\|x\|_2^2 + 1}\lrp{2 \delta^{1/2} L^{1/2}\lrp{\|x\|_2 + 1}}}\cdot 2^{13} \delta^{3/2}  L^{3/2} \lrp{\theta^3+\theta^2+\theta} \lrp{\|x\|_2^{11} + 1}\\
\leq& p^*(x)\exp\lrp{32 \delta^{1/2} \theta \lrp{\|x\|_2^3 + 1} L^{1/2}}\cdot 2^{13} \delta^{3/2}  L^{3/2} \lrp{\theta^3+\theta^2+\theta} \lrp{\|x\|_2^{11} + 1}\\
\leq& 2^{14} \delta^{3/2} p^*(x) L^{3/2}\lrp{\theta^3+\theta^2+\theta}\lrp{\|x\|_2^{11} + 1},
\end{align*}
where the first inequality is by Jensen's inequality, the triangle inequality and Cauchy-Schwarz, the second inequality is by Lemma~\ref{l:logp^*issmooth}.3 and
\eqref{e:he:1}, the third inequality is by Lemma~\ref{l:p^*growsslowly}, tnd he fourth inequality is by Lemma~\ref{l:onestepdiscretizationbounds}.1,
Assumption~\ref{ass:uissmooth}.2, Assumption~\ref{ass:gisregular}.2, and our assumption that $\delta\leq\frac{1}{d^2 L }$, so that $\lrn{F_\eta(x)}_2\leq 2\|x\|_2 + 2$. The fifth
inequality is by moving terms around, and the sixth inequality is by our assumption that $\delta\leq \frac{1}{2^{15}\theta^2 \lrp{\|x\|_2^6 + 1}L}$.
\end{proof}

\begin{lemma}\label{l:allthediscretizationbounds}
For any $\delta$, for any $x,y$, and for $\eta$ a.s.,
\begin{align*}
&1.\ \lrn{T_\eta(x)- T_\eta(y)}_2 \leq L^{1/2} \lrn{x-y}_2,\\
&2.\ \lrn{G_\eta(x) - G_\eta(y)}_2\leq L^{1/2} \lrn{x-y}_2, \\
&3.\ \lrn{T_\eta(y) - T_\eta(x) - G_\eta(x) (y-x)}_2 \leq L^{1/2} \lrn{y-x}_2^2,
\\
&4.\ 
\lrn{G_\eta(x) - G_\eta(y) - \lint{M_\eta(x), y-x}}_2\leq L^{1/2}\lrn{x-y}_2^2.
\end{align*}
\end{lemma}
\begin{proof}[Proof of Lemma \ref{l:allthediscretizationbounds}]
$ $\\
\begin{enumerate}
\item We use Assumption \ref{ass:gisregular}.3 and a Taylor exapansion:
\begin{align*}
\lrn{T_\eta(x)- T_\eta(y)}_2
=& \lrn{\int_0^1 G_\eta(t(x) + (1-t) y) (x-y)\,dt}_2\\
\leq& L^{1/2} \lrn{x-y}_2.
\end{align*}

\item We use Assumption \ref{ass:gisregular}.4 and a Taylor expansion:
\begin{align*}
\lrn{G_\eta(x) - G_\eta(y)}_2
=& \lrn{\int_0^1 \M{t(x) + (1-t) y}(\eta) (x-y)\,dt}_2\\
\leq& L^{1/2} \lrn{x-y}_2.
\end{align*}

\item Using Taylor's theorem and the definitions of $T_\eta$, $G_\eta$ and $M_\eta$ from Assumption~\ref{ass:gisregular}:
\begin{align*}
T_\eta(y) 
=& T_\eta(x) + \int_0^1 \lint{G_\eta ((1-t) x + t y), (y-x)} \,dt\\
=& T_\eta(x) + \int_0^1 \lint{\lrp{G_\eta(x) + \int_0^t \lint{M_\eta ((1-s) x + s y), (y-x)} ds}, (y-x)} \,dt\\
=& T_\eta(x) + \lint{G_\eta(x), y-x} + \int_0^1 \int_0^t \lint{\lint{ M_\eta ((1-s) x + s y), y-x}, y-x} \,ds \,dt,
\end{align*}
therefore,
\begin{align*}
&\lrn{T_\eta(y) - T_\eta(x) - G_\eta(x)(y-x)}_2\\
\leq& \int_0^1 \int_0^t \lrn{\lint{\lint{ M_\eta ((1-s) x + s y), y-x}, y-x} }_2 \,ds \,dt\\
\leq& \int_0^1 \int_0^t \lrn{M_\eta ((1-s) x + s y)}_2 \lrn{y-x}_2^2 \,ds \,dt\\
\leq& L^{1/2} \lrn{y-x}_2^2,
\end{align*}
where the first inequality is by the triangle inequality, the second inequality is by definition of the $\|\cdot\|_2$ norm in \eqref{d:operatornorm},
and the last inequality is by Assumption~\ref{ass:gisregular}.4.

\item Using Taylor's theorem and the definitions of $T_\eta$, $G_\eta$, $M_\eta$ and $N_\eta$,
\begin{align*}
G_\eta(y) 
=& G_\eta(x) + \int_0^1 {M_\eta ((1-t) x + t y) (y-x)} \,dt\\
=& G_\eta(x) + \int_0^1 {\lrp{M_\eta(x) + \int_0^t \lint{N_\eta ((1-s) x + s y) (y-x)} \,ds} (y-x)} \,dt\\
=& G_\eta(x) + {M_\eta(x) \lrp{y-x}} + \int_0^1 \int_0^t \lrp{\lint{ N_\eta ((1-s) x + s y), y-x}}\lrp{y-x} \,ds \,dt.
\end{align*}
Therefore,
\begin{align*}
&\lrn{G_\eta(y) - G_\eta(x) - M_\eta(x)(y-x)}_2\\
\leq& \int_0^1 \int_0^t \lrn{\lint{\lint{ N_\eta ((1-s) x + s y), y-x}, y-x} }_2 \,ds \,dt\\
\leq& \int_0^1 \int_0^t \lrn{N_\eta ((1-s) x + s y)}_2 \lrn{y-x}_2^2 \,ds \,dt\\
\leq& L^{1/2} \lrn{y-x}_2^2,
\end{align*}
\end{enumerate}
where the first inequality is by our expansion above and Jensen's inequality, the second inequality is by definition of $\lrn{\cdot}_2$
in~\eqref{d:operatornorm}, and the third inequality is by Assumption~\ref{ass:gisregular}.5.
\end{proof}
\begin{lemma}\label{l:onestepdiscretizationbounds}
For any $\delta \leq \frac{1}{32 L}$, for any $x,y$ such that $x = F_\eta(y)$ and for $\eta$ a.s.
\begin{align*}
&1.\ \lrn{y-x}_2 \leq  2 \delta^{1/2} L^{1/2}\lrp{\|x\|_2 + 1},\\
&2.\ \lrn{y - x - \lrp{-\sqrt{2\delta} T_\eta(x) + \delta \nabla U(x)  + {2\delta} G_\eta(x) T_\eta(x)}}_2
\leq 16 \delta^{3/2} L^{3/2} \lrp{\|x\|_2^2 + 1},\\
&3.\ \lrn{y - x -\lrp{-\sqrt{2\delta} T_\eta(x)}}_2
\leq 16 \delta L \lrp{\|x\|_2^2 + 1}.
\end{align*}
\end{lemma}

\begin{proof}[Proof of Lemma \ref{l:onestepdiscretizationbounds}]
$ $\\
\begin{enumerate}
\item \begin{align*}
\|y-x\|_2 
=& \lrn{\delta \nabla U(y) + \sqrt{2\delta} T_\eta(y)}_2\\
\leq& \lrn{\delta \nabla U(x) + \sqrt{2\delta} T_\eta(x)}_2 + \delta \lrn{\nabla U(y) - \nabla U(x)}_2 + \sqrt{2\delta}\lrn{T_\eta(y) - T_\eta(x)}_2\\
\leq& \lrn{\delta \nabla U(x) + \sqrt{2\delta} T_\eta(x)}_2 + \delta L \lrn{y-x}_2 + \sqrt{2\delta L}\lrp{\lrn{y-x}_2+1},
\end{align*}
where the first inequality is by the triangle inequality, the second inequality is by Assumptions~\ref{ass:uissmooth}.2 and~\ref{ass:gisregular}.3.
Moving terms around,
\begin{align*}
(1-\delta L - \sqrt{2\delta L}) \|y-x\|_2 
\leq& \lrn{\delta \nabla U(x) + \sqrt{2\delta} T_\eta(x)}_2\\
\leq& \delta L \|x\|_2 + \sqrt{2\delta L}(\|x\|_2 +1)\\
\Rightarrow \qquad\qquad\qquad\qquad \|y-x\|_2 \leq& 2 \lrp{\delta L + \sqrt{2 \delta L}}\lrp{\|x\|_2 + 1}\\
\leq& \delta^{1/2} {L}^{1/2}\lrp{\|x\|_2 + 1},
\end{align*}
where the second inequality is by Assumptions~\ref{ass:uissmooth}.1, \ref{ass:uissmooth}.2 and~\ref{ass:gisregular}.2, and the third inequality is by our assumption that $\delta\leq 1/(32L)$.

\item 
We first bound the expression $T_\eta(y) - T_\eta(x) + \sqrt{2\delta} G_\eta(x) T_\eta(x)$.

Plugging in $x = F_\eta(y) := y - \delta \nabla U(y) + \sqrt{2\delta} T_\eta(y)$, we get
\begin{align*}
& \lrn{T_\eta(y) - T_\eta(x) - G_\eta(x) (y-x)}_2\\
=& \lrn{T_\eta(y) - T_\eta(x) - G_\eta(x) \lrp{\delta \nabla U(y) - \sqrt{2\delta} T_\eta(y)}}_2\\
\geq& \lrn{T_\eta(y) - T_\eta(x) - G_\eta(x) \lrp{ - \sqrt{2\delta} T_\eta(x)}}_2\\
&\quad - \lrn{G_\eta(x) \lrp{\delta \nabla U(x) - \delta \nabla U(y)}}_2\\
&\quad - \lrn{G_\eta(x) \lrp{\sqrt{2\delta} T_\eta(x) -\sqrt{2\delta} T_\eta(y)}}_2\\
&\quad - \lrn{G_\eta(x) \delta \nabla U(x)}_2\\
\geq& \lrn{T_\eta(y) - T_\eta(x) - G_\eta(x) \lrp{ - \sqrt{2\delta} T_\eta(x)}}_2\\
&\quad - \delta L^{3/2} \lrn{x-y}_2 - \sqrt{2\delta}L \lrn{x-y}_2 - \delta L^{3/2}\|x\|_2,
\end{align*}
where the first inequality is by triangle inequality, and the second inequality is by Assumptions~\ref{ass:uissmooth} and~\ref{ass:gisregular} and Lemma~\ref{l:allthediscretizationbounds}.
Moving terms around, we get
\begin{align*}
& \lrn{T_\eta(y) - T_\eta(x) + \sqrt{2\delta}G_\eta(x) T_\eta(x)}_2\\
\leq& \lrn{T_\eta(y) - T_\eta(x) - G_\eta(y-x)}_2 + \delta L^{3/2} \lrn{x-y}_2 + \sqrt{2\delta}L\lrn{x-y}_2 + \delta L^{3/2}\|x\|_2\\
\leq& L^{1/2} \lrn{x-y}_2^2 + \delta L^{3/2} \lrn{x-y}_2 + \sqrt{2\delta}L\lrn{x-y}_2 + \delta L^{3/2}\|x\|_2\\
\leq& 8\delta L^{3/2} \lrp{\|x\|_2^2 + 1},
\numberthis \label{e:gl:1}
\end{align*}
where the second inequality is by Lemma~\ref{l:allthediscretizationbounds}.3, Lemma~\ref{l:onestepdiscretizationbounds}.1, and Young's Inequality, and the third inequality
is by our assumption that $\delta \leq 1/(32L)$.
Finally, by definition of $F_\eta(x)$, 
\begin{align*}
& x = y - \delta \nabla U(y)  + \sqrt{2\delta} T_\eta(y)\\
\Rightarrow \quad &y = x + \delta \nabla U(y) - \sqrt{2\delta} T_\eta(y)\\
\Rightarrow \quad & \lrn{y - x - \lrp{-\sqrt{2\delta} T_\eta(x) + \delta \nabla U(x)  + {2\delta} G_\eta(x)T_\eta(x)}}_2\\
=& \lrn{\delta \nabla U(y) - \sqrt{2\delta} T_\eta(y) - \lrp{-\sqrt{2\delta} T_\eta(x) + \delta \nabla U(x) + {2\delta} G_\eta(x)T_\eta(x)}}_2\\
=& \lrn{\delta \lrp{\nabla U(y) - \nabla U(x)} + \sqrt{2\delta} \lrp{T_\eta(x) - T_\eta(y) - \sqrt{2\delta} G_\eta(x)T_\eta(x)}}_2\\
\leq& \delta \lrn{\nabla U(x) - \nabla U(y)}_2 + \sqrt{2\delta} \lrn{T_\eta(y) - T_\eta(x) + \sqrt{2\delta} G_\eta(x)T_\eta(x)}_2\\
\leq& \delta L \lrn{x-y}_2 + 8\sqrt{2} \delta^{3/2}L^{3/2} \lrp{\|x\|_2^2+1}\\
\leq& 2 \delta^{3/2} L^{3/2} \lrp{\|x\|_2 + 1} + 8\sqrt{2} \delta^{3/2} L^{3/2} \lrp{\|x\|_2^2+1}\\
\leq& 16 \delta^{3/2} L^{3/2} \lrp{\|x\|_2^2 + 1},
\end{align*}
where the first inequality is by triangle inequality, the second inequality is by Assumptions~\ref{ass:uissmooth}.2 and~\eqref{e:gl:1}, and
the third inequality is by Lemma~\ref{l:onestepdiscretizationbounds}.1.

\item 
\begin{align*}
& \lrn{y - x - \lrp{-\sqrt{2\delta} T_\eta(x)}}_2\\
\leq& \lrn{T_\eta(y) - T_\eta(x) + \sqrt{2\delta}G_\eta(x) T_\eta(x)}_2 + \lrn{\delta \nabla U(x)  + {2\delta} G_\eta(x) T_\eta(x)}_2\\
\leq& 16 \delta^{3/2} L^{3/2} \lrp{\|x\|_2^2 + 1} + \delta L \|x\|_2 + 2\delta L \lrp{\|x\|_2 + 1}\\
\leq& 16 \delta^{3/2} L^{3/2} \lrp{\|x\|_2^2 + 1} + \delta L \|x\|_2 + 2\delta L (\|x\|_2 + 1)\\
\leq& 16 \delta L \lrp{\|x\|_2^2 + 1},
\end{align*}
where the first inequality is by the triangle inequality, and the second inequality is by Lemma~\ref{l:onestepdiscretizationbounds}.2 and Assumptions~\ref{ass:uissmooth} and~\ref{ass:gisregular}. The last inequality is by our assumption that $\delta\leq 1/(32 L)$.
\end{enumerate}
\end{proof}

\begin{lemma}\label{l:auxfordeterminant}
For any $\delta \leq \frac{1}{32 L}$, for any $x,y$ such that $x = F_\eta(y)$ and for $\eta$ a.s.
\begin{align*}
&1.\ \lrabs{\tr\lrp{G_\eta(x)} - \tr\lrp{G_\eta(y)} - \sqrt{2\delta} \tr\lrp{\lint{M_\eta(x), T_\eta(x)}}}\leq 8 \delta d L^{3/2}\lrp{\|x\|_2^2 + 1},\\
&2.\ \lrabs{\tr\lrp{\nabla^2 U(y)} - \tr\lrp{\nabla^2 U(x)}} \leq 2\delta^{1/2} d L^{3/2}\lrp{\|x\|_2 + 1},\\
&3.\ \lrabs{\tr\lrp{G_\eta(y)}^2 - \tr\lrp{G_\eta(x)}^2}\leq 4 \delta^{1/2} d^2 L^{3/2}\lrp{\|x\|_2 + 1},\\
&4.\ \lrabs{\tr\lrp{G_\eta(y)^2 - G_\eta(x)^2}} \leq 4 \delta^{1/2} d L^{3/2} \lrp{\|x\|_2 + 1}.
\end{align*}
\end{lemma}
\begin{proof}[Proof of Lemma \ref{l:auxfordeterminant}]
$ $\\
\begin{enumerate}

\item
By our definition of $x$ and $y$, 
\begin{align*}
&\lrn{G_\eta(y) - G_\eta(x) - \lint{M_\eta(x), y-x}}_2\\
=&\lrn{G_\eta(y) - G_\eta(x) - \lint{M_\eta(x), \delta \nabla U(y) - \sqrt{2\delta} T_\eta(y)}}_2\\
\geq& \lrn{G_\eta(y) - G_\eta(x) - \lint{M_\eta(x), - \sqrt{2\delta} T_\eta(x)}}_2 \\
&\quad - \lrn{\lint{M_\eta(x), \delta \nabla U(x)-\delta \nabla U(y)} }_2 \\
&\quad - \lrn{\lint{M_\eta(x), \sqrt{2\delta} T_\eta(x) - \sqrt{2\delta} T_\eta(y)}}_2\\
&\quad - \lrn{\lint{M_\eta(x), \delta \nabla U(x)}}_2\\
\geq& \lrn{G_\eta(y) - G_\eta(x) - \lint{M_\eta(x), - \sqrt{2\delta} T_\eta(x)}}_2 \\
&\quad - \delta L^{3/2} \lrn{x-y}_2 -  \sqrt{2\delta}L \lrn{x-y}_2 - \delta L^{3/2} \lrn{x}_2,
\end{align*}
where the first equality is by definition of $x$ and $y$, the first inequality is by the triangle ienquality, and the second inequality is by
Assumptions~\ref{ass:gisregular}.4, \ref{ass:gisregular}.3, and~\ref{ass:uissmooth}.2 and Lemma~\ref{l:allthediscretizationbounds}.1.
Moving terms around, we get
\begin{align*}
&\lrn{G_\eta(y) - G_\eta(x) + \lint{M_\eta(x), \sqrt{2\delta} T_\eta(x)}}_2\\
\leq& \lrn{G_\eta(y) - G_\eta(x) - \lint{M_\eta(x), y-x}}_2 + \delta L^{3/2} \lrn{x-y}_2 + \sqrt{2\delta}L \lrn{x-y}_2 + \delta L^{3/2} \lrn{x}_2\\
\leq& L^{1/2} \lrn{x-y}_2^2 + \delta L^{3/2} \lrn{x-y}_2 + \sqrt{2\delta}L \lrn{x-y}_2 + \delta L^{3/2} \lrn{x}_2\\
\leq& 8 \delta L^{3/2} \lrp{\|x\|_2^2 + 1},
\end{align*}
where the second inequality is by Lemma~\ref{l:allthediscretizationbounds}.4, and the third inequality is by Lemma~\ref{l:onestepdiscretizationbounds}.1 and our assumption that $\delta \leq 1/(32 L)$.
Finally, using the inequality $\tr{A} \leq d \lrn{A}_2$ from Lemma~\ref{l:traceupperbound}, we get
\begin{align*}
& \lrabs{\tr\lrp{G_\eta(y) - G_\eta(x) + \lint{M_\eta(x), \sqrt{2\delta} T_\eta(x)}}}\\
\leq & d\lrn{G_\eta(y) - G_\eta(x) + \lint{M_\eta(x), \sqrt{2\delta} T_\eta(x)}}_2\\
\leq& 8 \delta d L^{3/2}\lrp{\|x\|_2^2 + 1}.
\end{align*}

\item 
\begin{align*}
& \lrabs{\tr\lrp{\nabla^2 U(y)} - \tr\lrp{\nabla^2 U(x)}}\\
=& \lrabs{\tr\lrp{\nabla^2 U(y) - \nabla^2 U(x)}}\\
\leq& d\lrn{\nabla^2 U(y) - \nabla^2 U(x)}_2\\
\leq& d L \lrn{x-y}_2\\
\leq& 2 \delta^{1/2}d L^{3/2} \lrp{\|x\|_2 + 1},
\end{align*}
where the first inequality is by Lemma~\ref{l:traceupperbound}, the second inequality is by Assumption \ref{ass:uissmooth}.4, the third inequality is by Lemma~\ref{l:onestepdiscretizationbounds}.1.

\item
\begin{align*}
& \lrabs{\tr\lrp{G_\eta(y)}^2 - \tr\lrp{G_\eta(x)}^2}\\
=& \lrabs{\tr\lrp{G_\eta(y) - G_\eta(x)}\tr\lrp{G_\eta(y) + G_\eta(x)}}\\
\leq& d^2 \lrn{G_\eta(y) + G_\eta(x)}_2\lrn{G_\eta(y) - G_\eta(x)}_2\\
\leq& 2 d^2 L \lrn{x-y}_2\\
\leq& 4 \delta^{1/2} d^2 L^{3/2}\lrp{\|x\|_2 + 1},
\end{align*}
where the first inequality is by Lemma~\ref{l:traceupperbound}, the second inequality is by Assumptions \ref{ass:gisregular}.3 and \ref{ass:gisregular}.4, the last
inequality is by Lemma~\ref{l:onestepdiscretizationbounds}.1.

\item
\begin{align*}
& \lrabs{\tr\lrp{G_\eta(y)^2 - G_\eta(x)^2}}\\
=& \lrabs{\tr\lrp{G_\eta(y)^2 + G_\eta(y) G_\eta(x) - G_\eta(y)G_\eta(x)- G_\eta(x)^2}}\\
=& \lrabs{\tr\lrp{G_\eta(y)^2 + G_\eta(y) G_\eta(x) - G_\eta(x)G_\eta(y)- G_\eta(x)^2}}\\
=& \lrabs{\tr\lrp{\lrp{G_\eta(y) - G_\eta(x)}\lrp{G_\eta(y) + G_\eta(x)}}}\\
\leq& d\lrn{\lrp{G_\eta(y) + G_\eta(x)} \lrp{G_\eta(y) - G_\eta(x)}}_2\\
\leq& d\lrn{G_\eta(y) + G_\eta(x)}_2\lrn{G_\eta(y) - G_\eta(x)}_2\\
\leq& 2 d L \lrn{x-y}_2\\
\leq& 4 \delta^{1/2} d L^{3/2} \lrp{\|x\|_2 + 1},
\end{align*}
where the second inequality is because $\tr\lrp{AB} = \tr\lrp{BA}$, the first inequality is by Lemma~\ref{l:traceupperbound}, the second inequality is by Cauchy Schwarz,
the third inequality is by Assumption \ref{ass:gisregular}.4, the fourth inequality is by Lemma~\ref{l:onestepdiscretizationbounds}.1.
\end{enumerate}
\end{proof}

\begin{lemma}\label{l:determinantexpansioninverse}
For any $\delta \leq \frac{1}{2^8 d^2 L}$, for any $x$, and for $\eta$ a.s.,
\begin{align*}
& \left|\det\lrp{I - \lrp{{\delta} \nabla^2 U(x) - \sqrt{2\delta} G_\eta(x)} }^{-1} \right. \\
&\quad \left.-\lrp{ 1 - \sqrt{2\delta} \tr\lrp{G_\eta(x)} + \delta \tr\lrp{\nabla^2 U(x)}  + \delta \tr\lrp{G_\eta(x)}^2 + \delta \tr\lrp{\lrp{G_\eta(x)}^2}}\right|\\
\leq& 90 \delta^{3/2} d^3 L^{3/2}.
\end{align*}
\end{lemma}
\begin{proof}[Proof of Lemma~\ref{l:determinantexpansioninverse}]
First, let's consider an arbitrary symmetric matrix $A\in \Re^{2d}$, let $c$ be a constant such that $\|A\|_2 \leq c$ and let $\epsilon$ be a constant satisfying $\epsilon \leq 1/(2cd)$.

By Lemma~\ref{l:determinanttaylor}, we have
\begin{align*}
\det \lrp{I + \epsilon A} = 1 + \epsilon \tr\lrp{A} + \frac{\epsilon^2}{2} \lrp{\tr\lrp{A}^2-\tr\lrp{A^2}} + \Delta
\end{align*}
for some $|\Delta|\leq \epsilon^3 c^3 d^3$.

On the other hand, using Taylor expansion of $1/(1+x)$ about $x=0$, we can verify that for any $a\in [-1/2,1/2]$
\begin{align*}
\lrabs{(1+a)^{-1} -\lrp{1-a + a^2}} \leq  |2a|^3.
\numberthis \label{e:re:1}
\end{align*}
By our assumption on $\epsilon$, we have $\epsilon \tr\lrp{A} + \frac{\epsilon^2}{2} \lrp{\tr\lrp{A}^2-\tr\lrp{A^2}} + \Delta \in [-1/2,1/2]$, therefore
\begin{align*}
& \lrp{\det \lrp{I + \epsilon A}}^{-1}\\
=& \lrp{1 + \epsilon \tr\lrp{A} + \epsilon^2/2 \lrp{\tr\lrp{A}^2-\tr\lrp{A^2}} + \Delta}^{-1} \\
\leq& 1 - \epsilon \tr(A) - \epsilon^2/2 \lrp{\tr\lrp{A}^2-\tr\lrp{A^2}} - \Delta  \\
&\quad +\lrp{\epsilon \tr\lrp{A}+ \epsilon^2/2 \lrp{\tr\lrp{A}^2-\tr\lrp{A^2}} + \Delta}^2\\
&\quad + 2 \lrp{\epsilon \tr\lrp{A}+ \epsilon^2/2 \lrp{\tr\lrp{A}^2-\tr\lrp{A^2}} + \Delta}^3\\
=& 1 - \epsilon \tr(A) - \epsilon^2/2 \lrp{\tr\lrp{A}^2-\tr\lrp{A^2}} + \epsilon^2 \tr\lrp{A}^2\\
&\quad + \lrp{\epsilon^2/2 \lrp{\tr\lrp{A}^2-\tr\lrp{A^2}} + \Delta}\lrp{\epsilon \tr\lrp{A}+ \epsilon^2/2 \lrp{\tr\lrp{A}^2-\tr\lrp{A^2}} + \Delta}\\
&\quad + 2 \lrp{\epsilon \tr\lrp{A}+ \epsilon^2/2 \lrp{\tr\lrp{A}^2-\tr\lrp{A^2}} + \Delta}^3\\
\leq& 1 - \epsilon \tr(A) - \epsilon^2/2 \lrp{\tr\lrp{A}^2-\tr\lrp{A^2}} + \epsilon^2 \tr\lrp{A}^2 + 10 \lrp{\epsilon c d}^3 \\
=& 1 - \epsilon \tr(A) + \epsilon^2/2 \lrp{\tr\lrp{A}^2+\tr\lrp{A^2}} + 10 \lrp{\epsilon c d}^3,
\end{align*}
where the first inequality is by \eqref{e:re:1}, the first inequality is by moving terms around, the second inequality is by our assumption that $\|A\|_2 \leq c$ and the
fact that $\lrabs{\Delta}\leq \epsilon^3c^3d^3$ and by Lemma~\ref{l:traceupperbound}, and the last two lines are by collecting terms.
Conversely, one can show that
\begin{align*}
& \lrp{1 + \epsilon \tr\lrp{A} + \epsilon^2/2 \lrp{\tr\lrp{A}^2-\tr\lrp{A^2}} + \Delta}^{-1} \\
\geq& 1 - \epsilon \tr(A) + \epsilon^2/2 \lrp{\tr\lrp{A}^2+\tr\lrp{A^2}} - 10 \lrp{\epsilon c d}^3.
\end{align*}
The proof is similar and is omitted.

Therefore
\begin{align*}
\lrabs{\det\lrp{I + \epsilon A}^{-1} -\lrp{ 1 - \epsilon \tr(A) + \epsilon^2/2 \lrp{\tr\lrp{A}^2+\tr\lrp{A^2}}}}\leq  10 \lrp{\epsilon c d}^3. \numberthis \label{e:re:2}
\end{align*}

Now, we consider the case that $A:= -\sqrt{\delta} \nabla^2 U(x) + \sqrt{2}G_\eta(x)$, $\epsilon := \sqrt{\delta}$ and $c:= 2 L^{1/2}$. Recall our assumption that $\delta \leq \frac{1}{2^8 d^2 L }$. Combined with Assumption \ref{ass:uissmooth}.2 and \ref{ass:gisregular}.3, we get
\begin{align*}
&1.\ \lrn{A}_2 \leq c,\\
&2.\ \epsilon = \sqrt{\delta} \leq {1}/\lrp{2^4 d L^{1/2}}\leq 1/(2cd).
\end{align*}
Using \eqref{e:re:2},
\begin{align*}
& \det \lrp{I - \sqrt{\delta} \lrp{\sqrt{\delta} \nabla^2 U(x) - \sqrt{2} G_\eta(x)} }^{-1}\\
=:& \det\lrp{I+\epsilon A}^{-1}\\
\leq& 1 - \epsilon \tr(A) + \epsilon^2/2 \lrp{\tr\lrp{A}^2+\tr\lrp{A^2}} + 10 \lrp{\epsilon c d}^3\\
=& 1 + \sqrt{\delta} \tr\lrp{\sqrt{\delta}\nabla^2 U(x) - \sqrt{2} \tr\lrp{G_\eta(x)}}\\
&\quad + \delta/2 \lrp{\tr\lrp{\sqrt{\delta}\nabla^2 U(x) - \sqrt{2} \lrp{G_\eta(x)}}^2 + \tr\lrp{\lrp{\sqrt{\delta}\nabla^2 U(x) - \sqrt{2} G_\eta(x)}^2}}\\
&\quad + 80 \delta^{3/2} d^3 L^{3/2}\\
=& 1 + \delta \tr\lrp{\nabla^2 U(x)} - \sqrt{2\delta} \tr\lrp{G_\eta(x)}\\
&\quad + \delta \tr\lrp{G_\eta(x)}^2 + \delta \tr\lrp{\lrp{G_\eta(x)}^2} \\
&\quad + \delta^2/2 \tr\lrp{\nabla^2 U(x)}^2 + 2\delta^{3/2} \tr\lrp{\nabla^2 U(x)} \tr\lrp{G_\eta(x)}\\
&\quad + \delta^2/2 \tr\lrp{\lrp{\nabla^2 U(x)}^2} + 2\delta^{3/2} \tr\lrp{\nabla^2 U(x)G_\eta(x)}\\
&\quad + 80 \delta^{3/2} d^3 L^{3/2}\\
\leq& 1 + \delta \tr\lrp{\nabla^2 U(x)} - \sqrt{2\delta} \tr\lrp{G_\eta(x)}\\
&\quad + \delta \tr\lrp{G_\eta(x)}^2 + \delta \tr\lrp{\lrp{G_\eta(x)}^2} \\
&\quad + \delta^{3/2} dL^{3/2} + 2\delta^{3/2} d^2 L^{3/2}\\
&\quad + \delta^{3/2} L^{3/2} + 2\delta^{3/2} dL^{3/2}\\
&\quad + 80 \delta^{3/2} d^3 L^{3/2}\\
\leq& 1 + \delta \tr\lrp{\nabla^2 U(x)} - \sqrt{2\delta} \tr\lrp{G_\eta(x)} + \delta \tr\lrp{G_\eta(x)}^2 + \delta \tr\lrp{\lrp{G_\eta(x)}^2} + 90 \delta^{3/2} d^3 L^{3/2},
\end{align*}
where the first inequality is by \eqref{e:re:2}, the second equality is by definition of $A$ and $\epsilon$, the third equality is by moving terms around, the second inequality is by Assumption \ref{ass:uissmooth}.2 and \ref{ass:gisregular}.3, the third inequality is again by moving terms around.

Conversely, one can show that
\begin{align*}
& \det\lrp{I - \sqrt{\delta} \lrp{\sqrt{\delta} \nabla^2 U(x) - \sqrt{2} G_\eta(x)} }^{-1}\\
\geq& 1 + \delta \tr\lrp{\nabla^2 U(x)} - \sqrt{2\delta} \tr\lrp{G_\eta(x)} + \delta \tr\lrp{G_\eta(x)}^2 + \delta \tr\lrp{\lrp{G_\eta(x)}^2} - 90 \delta^{3/2} d^3 L^{3/2}.
\end{align*}
The proof is similar and is omitted.
\end{proof}

\begin{lemma}\label{l:discretizeddeterminantexpansioninverse}
For any $\delta \leq \frac{1}{2^8 d^2 L}$, for any $x$, and for $\eta$ a.s., 
\begin{align*}
&\det\lrp{\nabla F_\eta(F_\eta^{-1} (x))}^{-1}\\
=& 1 - \sqrt{2\delta} \tr\lrp{G_\eta(x)} + 2\delta \tr\lrp{\lint{M_\eta(x),T_\eta(x)}} + \delta \tr\lrp{\nabla^2 U(x)} + \delta \tr\lrp{G_\eta(x)}^2 + \delta
\tr\lrp{\lrp{G_\eta(x)}^2} + \Delta
\end{align*}
for some $\lrabs{\Delta}\leq 128 \delta^{3/2} d^3 L^{3/2}\lrp{\|x\|_2^2 + 1}$.
\end{lemma}

\begin{proof}[Proof of Lemma~\ref{l:discretizeddeterminantexpansioninverse}]
Consider the Jacobian matrix inside the determinant. By definition of $F_\eta$, we know that 
$$\nabla F_\eta\lrp{F_\eta^{-1}(x)} = I - \delta \nabla^2 U\lrp{F_\eta^{-1}(x)} + \sqrt{2\delta} G_\eta\lrp{F_\eta^{-1}(x)}.$$
Thus,
\begin{align*}
&\det\lrp{\nabla F_\eta(F_\eta^{-1}(x)}^{-1}\\
=& \det\lrp{I - \sqrt{\delta} \lrp{\sqrt{\delta} \nabla^2 U(F_\eta^{-1}(x)) - \sqrt{2} G_\eta(F_\eta^{-1}(x))} }^{-1} \\
\leq& 1 - \sqrt{2\delta} \tr\lrp{G_\eta(F_\eta^{-1}(x))}+ \delta \tr\lrp{\nabla^2 U(F_\eta^{-1}(x))}  + \delta \tr\lrp{G_\eta(F_\eta^{-1}(x))}^2 + \delta
\tr\lrp{\lrp{G_\eta(F_\eta^{-1}(x))}^2}\\*
& \quad {} + 80 \delta^{3/2} d^{3} L^{3/2}\\
\leq& 1 - \sqrt{2\delta} \tr\lrp{G_\eta(x)} + 2\delta \tr\lrp{\lint{M_\eta(x),T_\eta(x)}} + \delta \tr\lrp{\nabla^2 U(x)} + \delta \tr\lrp{G_\eta(x)}^2 + \delta \tr\lrp{\lrp{G_\eta(x)}^2}\\
&\quad + \sqrt{2\delta}\lrabs{\tr\lrp{G_\eta(x)} - \tr\lrp{G_\eta(F_\eta^{-1}(x))} - \sqrt{2\delta} \tr\lrp{\lint{M_\eta(x), T_\eta(x)}}}\\
&\quad + \delta \lrabs{\tr\lrp{\nabla^2 U(x)} - \tr\lrp{\nabla^2 U(F_\eta^{-1}(x)}}\\
&\quad + \delta \lrabs{\tr\lrp{G_\eta(x)}^2 - \tr\lrp{G_\eta(F_\eta^{-1}(x))}^2}\\
&\quad + \delta \lrabs{\tr\lrp{G_\eta(x)^2} - \tr\lrp{G_\eta(F_\eta^{-1}(x))^2}}\\
&\quad + 80 \delta^{3/2} d^{3} L^{3/2}\\
\leq& 1 - \sqrt{2\delta} \tr\lrp{G_\eta(x)} + 2\delta \tr\lrp{\lint{M_\eta(x),T_\eta(x)}} + \delta \tr\lrp{\nabla^2 U(x)} + \delta \tr\lrp{G_\eta(x)}^2 + \delta \tr\lrp{\lrp{G_\eta(x)}^2}\\
&\quad + 8 \delta^{3/2} d L^{3/2}\lrp{\|x\|_2^2 + 1}\\
&\quad + 2 \delta^{3/2} d L^{3/2} \lrp{\|x\|_2 + 1}\\
&\quad + 4\delta^{3/2} d^2 L^{3/2}\lrp{\|x\|_2 + 1}\\
&\quad + 4\delta^{3/2} d L^{3/2}\lrp{\|x\|_2 + 1}\\
&\quad + 90 \delta^{3/2} d^{3} L^{3/2}\\
\leq& 1 - \sqrt{2\delta} \tr\lrp{G_\eta(x)} + 2\delta \tr\lrp{\lint{M_\eta(x),T_\eta(x)}} + \delta \tr\lrp{\nabla^2 U(x)} + \delta \tr\lrp{G_\eta(x)}^2 + \delta \tr\lrp{\lrp{G_\eta(x)}^2}\\
&\quad + 128 \delta^{3/2} d^3 L^{3/2}\lrp{\|x\|_2^2 + 1},
\end{align*}
where the first inequality is by Lemma~\ref{l:determinantexpansioninverse}, the second inequality is by triangle inequality, the third inequality is by Lemma~\ref{l:auxfordeterminant}, the fourth inequality is by collecting terms.\\
Conversely, one can show that
\begin{align*}
&\det\lrp{\nabla F_\eta(F_\eta^{-1}(x)}^{-1}\\
\geq& 1 - \sqrt{2\delta} \tr\lrp{G_\eta(x)} + 2\delta \tr\lrp{\lint{M_\eta(x),T_\eta(x)}} + \delta \tr\lrp{\nabla^2 U(x)} + \delta \tr\lrp{G_\eta(x)}^2 + \delta \tr\lrp{\lrp{G_\eta(x)}^2}\\
&\quad - 128 \delta^{3/2} d^3 L^{3/2}\lrp{\|x\|_2^2 + 1}.
\end{align*}
The proof is similar and is omitted.
\end{proof}

\begin{corollary}\label{c:determinantinversenaivebound}
For any $\delta \leq \frac{1}{2^8 d^2 L}$, for any $x$, and for $\eta$ a.s.,
\begin{align*}
&\lrabs{\det\lrp{\nabla F_\eta(F_\eta^{-1} (x))}^{-1} - \lrp{1 - \sqrt{2\delta} \tr\lrp{G_\eta(x)}}}\\
\leq& 8 \delta d^2 L \lrp{\|x\|_2^2 + 1}.
\end{align*}
\end{corollary}
\begin{proof}[Proof of Corollary \ref{c:determinantinversenaivebound}]
$ $\\
Let
\begin{align*}
\Delta :=& \det\lrp{\nabla F_\eta(F_\eta^{-1} (x))}^{-1}\\
&\quad - \lrp{1 - \sqrt{2\delta} \tr\lrp{G_\eta(x)} + 2\delta \tr\lrp{\lint{M_\eta(x),T_\eta(x)}} + \delta \tr\lrp{\nabla^2 U(x)} + \delta \tr\lrp{G_\eta(x)}^2 + \delta
\tr\lrp{\lrp{G_\eta(x)}^2}}.
\end{align*}
By Lemma~\ref{l:discretizeddeterminantexpansioninverse}, 
\begin{align*}
\lrabs{\Delta}\leq 128 \delta^{3/2} d^3 L^{3/2}\lrp{\|x\|_2^2 + 1}.
\end{align*}
Thus 
\begin{align*}
&\lrabs{\det\lrp{\nabla F_\eta(F_\eta^{-1} (x))}^{-1} - \lrp{1 - \sqrt{2\delta} \tr\lrp{G_\eta(x)}}}\\
=& \lrabs{\Delta + 2\delta \tr\lrp{\lint{M_\eta(x),T_\eta(x)}} + \delta \tr\lrp{\nabla^2 U(x)} + \delta \tr\lrp{G_\eta(x)}^2 + \delta \tr\lrp{\lrp{G_\eta(x)}^2}}\\
\leq& 128 \delta^{3/2} d^3 L^{3/2}\lrp{\|x\|_2^2 + 1} + 2\delta d L \lrp{\|x\|_2 + 1} + \delta d L + \delta d^2 L + \delta d\\
\leq& 8 \delta d^2 L \lrp{\|x\|_2^2 + 1},
\end{align*}
where the first line is by our definition of $\Delta$, the second line is by our bound on $\lrabs{\Delta}$ above and by Assumptions \ref{ass:uissmooth} and \ref{ass:gisregular}, the third inequality is by moving terms around.
\end{proof}

\begin{corollary}\label{c:determinantinversenaivenaivebound}
For any $\delta \leq \frac{1}{2^8 d^2 L}$, for any $x$, and for $\eta$ a.s.,
\begin{align*}
&\lrabs{\det\lrp{\nabla F_\eta(F_\eta^{-1} (x))}^{-1} - 1}\\
\leq& 2\delta^{1/2} d L^{1/2} \lrp{\|x\|_2^2 + 1}.
\end{align*}
\end{corollary}
\begin{proof}[Proof of Corollary \ref{c:determinantinversenaivenaivebound}]
From Lemma~\ref{l:discretizeddeterminantexpansioninverse}, we get
\begin{align*}
&\lrabs{\det\lrp{\nabla F_\eta(F_\eta^{-1} (x))}^{-1} - 1}\\
\leq& \lrabs{\sqrt{2\delta}\tr\lrp{G_\eta(x)}} + \lrabs{\det\lrp{\nabla F_\eta(F_\eta^{-1} (x))}^{-1} - \lrp{1 - \sqrt{2\delta} \tr\lrp{G_\eta(x)}}}\\
\leq& \sqrt{2\delta} d L^{1/2} + 8 \delta d^2L\lrp{\|x\|_2^2 + 1}\\
\leq& 2\delta^{1/2} d L^{1/2} \lrp{\|x\|_2^2 + 1},
\end{align*}
where the first inequality is by triangle inequality, and the second inequality is by Corollary \ref{c:determinantinversenaivebound}.
\end{proof}

\begin{lemma}
\label{l:logp^*issmooth}
Under Assumption \ref{ass:p^*isregular}, for all $x$, 
\begin{align*}
1.\ &\lrn{\nabla p^*(x)}_2\leq p^*(x) \theta \lrp{\|x\|_2^2 + 1}\\
2.\ &\lrn{\nabla^2 p^*(x)}_2 \leq p^*(x) \lrp{\theta^2 +\theta}\lrp{\|x\|_2^4 +1}\\
3.\ &\lrn{\nabla^3 p^*(x)}_2\leq 2 p^*(x) \lrp{\theta^3+\theta^2+\theta}\lrp{\|x\|_2^6+1}.
\end{align*}
\end{lemma}
\begin{proof}[Proof of Lemma~\ref{l:logp^*issmooth}]
To prove the first claim:
\begin{align*}
\lrn{\nabla p^*(x)}_2
=& \lrn{p^*(x) \nabla \log p^*(x)}_2\\
\leq& p^*(x) \theta \lrp{\|x\|_2^2 + 1}.
\end{align*}
To prove the second claim:
\begin{align*}
\lrn{\nabla^2 p^*(x)}_2
=& p^*(x)\lrn{\nabla^2 \log p^*(x) + \nabla \log p^*(x) \nabla \log p^*(x)^T}_2\\
\leq& p^*(x)\lrp{\lrn{\nabla^2 \log p^*(x)}_2 + \lrn{\nabla \log p^*(x)}_2^2}\\
\leq& p^*(x) \lrp{\theta^2 +\theta}\lrp{\|x\|_2^4 +1},
\end{align*}
where the second and third inequalities are by Assumption \ref{ass:p^*isregular}.

To prove the third claim:
\begin{align*}
&\lrn{\nabla^3 p^*(x)}_2 \\
=& p^*(x) \left\|\nabla^3 \log p^*(x)  + \nabla^2 \log p^*(x) \otimes \nabla \log p^*(x) + \nabla \log p^*(x) \otimes \nabla^2 \log p^*(x)\right.\\
&\quad \left. + \nabla \log p^*(x) \otimes \nabla \log p^*(x) \otimes \nabla \log p^*(x)\right\|_2\\
\leq& p^*(x) \lrp{\lrn{\nabla^3 \log p^*(x)}_2 + 2\lrn{\nabla^2 \log p^*(x)}_2\lrn{\nabla \log p^*(x)}_2 + \lrn{\nabla \log p^*(x)}_2^3}\\
\leq& 2 p^*(x) \lrp{\theta^3+\theta^2+\theta}\lrp{\|x\|_2^6+1},
\end{align*}
where $\otimes$ denotes the tensor outer product.
\end{proof}

\begin{lemma}\label{l:p^*growsslowly}
Under Assumption \ref{ass:p^*isregular}, for all $x,y$, 
\begin{align*}
p^*(y) \leq p^*(x)\cdot \exp\lrp{\theta \lrp{\|x\|_2^2 + \|y\|_2^2}\lrn{y-x}_2}.
\end{align*}
\end{lemma}
\begin{proof}
Under Assumption \ref{ass:p^*isregular},
\begin{align*}
& \lrabs{\log p^*(y) - \log p^*(x)}\\
=& \lrabs{\int_0^1 \lin{\nabla \log p^*((1-t)x + ty), y-x} dt}\\
\leq& \int_0^1 \lrn{\nabla \log p^*((1-t)x + ty)}_2 \cdot \lrn{y-x}_2 dt \\
\leq& \int_0^1 \theta \lrp{\lrn{(1-t)x+ty}_2^2+1} \lrn{y-x}_2 dt\\
\leq& 2\theta \lrp{\|x\|_2^2 + \|y\|_2^2}\lrn{y-x}_2.
\end{align*}
Therefore,
\begin{align*}
\exp\lrp{-\theta \lrp{\|x\|_2^2 + \|y\|_2^2}\lrn{y-x}_2}\leq \frac{p^*(y)}{p^*(x)} \leq \exp\lrp{\theta \lrp{\|x\|_2^2 + \|y\|_2^2}\lrn{y-x}_2}.
\end{align*}
\end{proof}

\begin{lemma}
\label{l:fokkerplanckcharacterizationofstationarydistribution}
The stationary distribution $p^*$ of \eqref{e:exactsde} satisfies the equality (for all $x$)
\begin{align}
0 =&  p^*(x)\lrp{\sum_{i=1}^d \sum_{j=1}^d \frac{\del^2}{\del x_i \del x_j} \lrb{\sigma_{x}\sigma_{x}^T}_{i,j} + \tr\lrp{\nabla^2 U(x)}}\\
&\quad +  \sum_{i=1}^d \dd{x_i} p^*(x) \cdot \dd{x_i} U(x) \\
&\quad + 2 {\sum_{i=1}^{d}\lrp{\dd{x_i} p^*(x)  \lrp{\sum_{j=1}^d\dd{x_j} \lrb{\sigma_{x}\sigma_{x}^T}_{i,j}}}}\\
&\quad +   \lin{\nabla^2p^*(x), \sigma_{x}\sigma_{x}^T}.
\end{align}

\end{lemma}
\begin{proof}[Proof of Lemma~\ref{l:fokkerplanckcharacterizationofstationarydistribution}]
For a distribution $p_t$, the Fokker Planck equation under \eqref{e:exactsde} is
\begin{align*}
\ddt p_t(x) =& \sum_{i=1}^d \lrp{\dd{x_i} \lrp{\lrb{\nabla U(x)}_i\cdot p_t(x)}} + \sum_{i=1}^d \sum_{j=1}^d \frac{\del^2}{\del x_i \del x_j} \lrp{\lrb{\sigma_x \sigma_x^T}_{i,j}\cdot p_t(x)}\\
=& \sum_{i=1}^d \lrp{\dd{x_i} \lrp{\lrb{\nabla U(x)}_i\cdot p_t(x)}}\\
&\quad + \lrp{\sum_{i=1}^d \sum_{j=1}^d \frac{\del^2}{\del x_i \del x_j} \lrb{\sigma_x \sigma_x^T}_{i,j}}\cdot p_t(x)\\
&\quad + 2\sum_{i=1}^d \sum_{j=1}^d \lrp{\lrp{\dd{x_j} \lrb{\sigma_x \sigma_x^T}_{i,j}}\cdot \lrp{\dd{x_i} p_t(x)}}\\
&\quad + \sum_{i=1}^d \sum_{j=1}^d \lrp{ \lrb{\sigma_x \sigma_x^T}_{i,j}\cdot \lrp{\frac{\del^2}{\del x_i \del x_j} p_t(x)}}\\
=& p_t(x)\lrp{\sum_{i=1}^d \sum_{j=1}^d \frac{\del^2}{\del x_i \del x_j} \lrb{\sigma_{x}\sigma_{x}^T}_{i,j} + \tr\lrp{\nabla^2 U(x)}}\\
&\quad +  \sum_{i=1}^d \dd{x_i} p_t(x) \cdot \dd{x_i} U(x) \\
&\quad + 2 {\sum_{i=1}^{d}\lrp{\dd{x_i} p_t(x)  \lrp{\sum_{j=1}^d\dd{x_j} \lrb{\sigma_{x}\sigma_{x}^T}_{i,j}}}}\\
&\quad +   \lin{\nabla^2p_t(x), \sigma_{x}\sigma_{x}^T}.
\end{align*}
Observe that by definition of $p^*$ being the stationary distribution of \eqref{e:exactsde}, $\at{\ddt p_t(x)}{p_t=p^*}=0$. Thus, we have
\begin{align*}
& p^*(x)\lrp{\sum_{i=1}^d \sum_{j=1}^d \frac{\del^2}{\del x_i \del x_j} \lrb{\sigma_{x}\sigma_{x}^T}_{i,j} + \tr\lrp{\nabla^2 U(x)}}\\
&\quad +  \sum_{i=1}^d \dd{x_i} p^*(x) \cdot \dd{x_i} U(x) \\
&\quad + 2 {\sum_{i=1}^{d}\lrp{\dd{x_i} p^*(x)  \lrp{\sum_{j=1}^d\dd{x_j} \lrb{\sigma_{x}\sigma_{x}^T}_{i,j}}}}\\
&\quad +   \lin{\nabla^2p^*(x), \sigma_{x}\sigma_{x}^T}\\
=0.
\end{align*}
\end{proof}

\begin{lemma}\label{l:reallyannoyinglemma}

For any $\delta$ satisfying
\begin{align*}
\frac{1}{\delta} \geq \max
\begin{cases}
2^8 d^2 L\\
2^{37} L\theta^2\\
2^{37}L\theta^2 \lrp{\frac{c_{\sigma}^2}{m}\log \frac{c_{\sigma}^2}{m}}^3\\
2^{72} L \theta^{2} \frac{c_{\sigma}^2}{m}\log \lrp{\frac{2^{62} L c_{\sigma}^2}{m}}
\end{cases}
\numberthis \label{e:ye:3}
\end{align*}
and for
\begin{align*}
R:=& 2^7\sqrt{\max\lrbb{\frac{c_{\sigma}^2}{m} \log\frac{c_{\sigma}^2}{m}, {\frac{c_{\sigma}^2}{m}\log \lrp{\frac{1}{2^{124}d^6 L^2 \lrp{\theta^3 + \theta^2 + \theta}^2 \delta^3 }}}, 1}},
\end{align*}
$\delta$ and $R$ satisfy
\begin{align*}
\delta \leq \min\lrbb{\frac{1}{2^8 d^2 L}, \frac{1}{2^{15}L \theta^2 \lrp{R^6 + 1}}}.
\end{align*}
\end{lemma}

\begin{proof}[Proof of Lemma~\ref{l:reallyannoyinglemma}]
Our first assumption in \eqref{e:ye:3} immediately implies that $\delta \leq \lrp{2^8 d^2 L}^{-1}$, so we only need to verify that
\begin{align*}
\delta 
\leq& \frac{1}{2^{15}L \theta^2 \lrp{R^6 + 1}}.
\numberthis \label{e:ye:1}
\end{align*}

Since $R$ is a max of three terms, we will consider 2 cases:\\
\textbf{Case 1: $R= 2^7\sqrt{\max\lrbb{\frac{c_{\sigma}^2}{m} \log\frac{c_{\sigma}^2}{m}, 1}}$}

In this case, \eqref{e:ye:1} follows immediately from our second and third assumption in \eqref{e:ye:3}.

\textbf{Case 2: $R = 2^7 \sqrt{\frac{c_{\sigma}^2}{m}\log \lrp{\frac{1}{2^{124}d^6 L^2 \lrp{\theta^3 + \theta^2 + \theta}^2 \delta^3 }}}$}\\

Recall that we would like to prove that
\begin{align*}
\delta \leq \lrp{2^{15} L \theta^2 \lrp{R^6 + 1}}^{-1}
\end{align*}
Since $R^6 + 1\leq \max\lrbb{2R^6, 2}$, it suffices to prove that
\begin{align*}
\frac{1}{\delta} \geq& 2^{16} L \theta^2\\
& \quad \text{and}\\
\frac{1}{\delta} \geq& 2^{16} L \theta^2 R^6.
\end{align*}

The first inequality follows immediately from our second assumption in \eqref{e:ye:3}.

The second inequality expands to be
\begin{align*}
\frac{1}{\delta} \geq 2^{58} L \theta^2 \frac{c_{\sigma}^6}{m^3} \lrp{\log \lrp{\frac{1}{2^{124}d^6 L^2 \lrp{\theta^3 + \theta^2 + \theta}^2 \delta^3}}}^3.
\end{align*}
Moving terms around, we see that it is sufficient to prove
\begin{align*}
\delta^{-1/3} \geq \lrp{2^{20} L^{1/3} \theta^{2/3}  \frac{c_{\sigma}^2}{m} } \log \lrp{\delta^{-1/3} \lrp{2^{124}d^6 L^2 \lrp{\theta^3 + \theta^2 + \theta}^2}^{-1/9}}.
\numberthis \label{e:ye:2}
\end{align*}

We define $a:= \lrp{2^{124}d^6 L^2 \lrp{\theta^3 + \theta^2 + \theta}^2}^{-1/9}$, $c:= \lrp{2^{20} L^{1/3} \theta^{2/3}  \frac{c_{\sigma}^2}{m} }^{-1}$ and $x:=\delta^{-1/3}$. We verify that $a$ and $c$ are both strictly positive quantities. 
By the fourth assumption in \eqref{e:ye:3},
\begin{align*}
& \delta^{-1/3}\\
\geq& 2^{24} L^{1/3} \theta^{2/3} \frac{c_{\sigma}^2}{m}\log \lrp{2^{62} L \frac{c_{\sigma}^2}{m}}\\
\geq& 3 \cdot 2^{20} L^{1/3} \theta^{2/3} \frac{c_{\sigma}^2}{m} \cdot \log \lrp{\frac{2^{20}L^{1/3} \theta^{2/3} \frac{c_{\sigma}^2}{m}}{\lrp{2^{124}d^6 L^2 \lrp{\theta^3 + \theta^2 + \theta}^2}^{1/9}}}\\
=:& 3\cdot \frac{1}{c}\log \frac{a}{c}.
\end{align*}
We can thus apply Corollary \ref{c:xlogxbound}, with the given $a,c,x$, to prove \eqref{e:ye:2} ($\delta^{1/3} > 0$ is guaranteed by the first assumption of \ref{e:ye:3}.

We have concluded the proof of Case 2, and hence \eqref{e:ye:1}.

\end{proof}

\begin{lemma}\label{l:reallyannoyinglemma2}
For any $\epsilon > 0$, and for any stepsize $\delta$  satisfying
\begin{align*}
\frac{1}{\delta} \geq
\frac{d^7}{\epsilon^2} \cdot \frac{2^{142} L^2 \lrp{\theta^3 + \theta^2 + \theta}^2}{\lambda^2} \cdot \max
\begin{cases}
\lrp{\frac{c_{\sigma}^2}{m} \log \frac{c_{\sigma}^2}{m}}^{12}\\
1\\
\lrp{\frac{c_{\sigma}^2}{m} \log \lrp{2^{324} d^5 L \lrp{\theta^3 + \theta^2 + \theta}\lambda^{-6}\epsilon^{-6}}}^{12}
\end{cases}
\numberthis \label{e:iw:1}
\end{align*}
then
\begin{align*}
2^{70}\delta^{1/2} d^{7/2} L \lrp{\theta^3 + \theta^2 + \theta} \max\lrbb{\frac{c_{\sigma}^2}{m} \log \frac{c_{\sigma}^2}{m}, \frac{c_{\sigma}^2}{m}\log
\lrp{\frac{1}{2^{124}d^6 L^2 \lrp{\theta^3 + \theta^2 + \theta} \delta^3 }}, 1}^{6} \lambda^{-1} \leq \frac{\epsilon}{2}.
\end{align*}
\end{lemma}
\begin{proof}[Proof of Lemma~\ref{l:reallyannoyinglemma2}]
By the first two cases in the max in \eqref{e:iw:1} and moving terms around, one can immediately verify that
\begin{align*}
2^{70}\delta^{1/2} d^{7/2} L \lrp{\theta^3 + \theta^2 + \theta} \max\lrbb{\frac{c_{\sigma}^2}{m} \log \frac{c_{\sigma}^2}{m}, 1}^{6} \lambda^{-1} \leq \frac{\epsilon}{2}.
\end{align*}
Thus we only need to prove that
\begin{align*}
2^{70}\delta^{1/2} d^{7/2} L \lrp{\theta^3 + \theta^2 + \theta} \lrp{\frac{c_{\sigma}^2}{m}\log \lrp{\frac{1}{2^{124}d^6 L^2 \lrp{\theta^3 + \theta^2 + \theta}^2 \delta^3
}}}^{6} \lambda^{-1} \leq \frac{\epsilon}{2}.
\end{align*}
The above is equivalent to 
\begin{align*}
\delta^{-1/12} \geq \log \lrp{\lrp{2^{124}d^6 L^2 \lrp{\theta^3 + \theta^2 + \theta}^2}^{-1/36} \delta^{-1/12}} \lrp{2^{71} d^{7/2} L \lrp{\theta^3 + \theta^2 + \theta}
\frac{c_{\sigma}^{12}}{m^6}\lambda^{-1}\epsilon^{-1}}^{1/6}. \numberthis \label{e:iw:2}
\end{align*}
Let us define
\begin{align*}
a:=& \lrp{2^{124}d^6 L^2 \lrp{\theta^3 + \theta^2 + \theta}^2}^{-1/36},\\
c:=&  \lrp{2^{71} d^{7/2} L \lrp{\theta^3 + \theta^2 + \theta} \frac{c_{\sigma}^{12}}{m^6}\lambda^{-1}\epsilon^{-1}}^{-1/6},\\
x:=& \delta^{-1/12}.
\end{align*}
Then by the third case in our max in \eqref{e:iw:1},
\begin{align*}
\delta^{-1/12}
\geq& \lrp{2^{71} d^{7/2} L \lrp{\theta^3 + \theta^2 + \theta} \frac{c_{\sigma}^{12}}{m^6}\lambda^{-1}\epsilon^{-1}}^{-1/6} \cdot \log \lrp{2^{324} d^5 L \lrp{\theta^3 + \theta^2 + \theta}\lambda^{-6}\epsilon^{-6}}\\
\geq& 3 \cdot \lrp{2^{71} d^{7/2} L \lrp{\theta^3 + \theta^2 + \theta} \frac{c_{\sigma}^{12}}{m^6}\lambda^{-1}\epsilon^{-1}}^{-1/6} \cdot \log \lrp{\frac{\lrp{2^{71} d^{7/2} L \lrp{\theta^3 + \theta^2 + \theta} \frac{c_{\sigma}^{12}}{m^6}\lambda^{-1}\epsilon^{-1}}^{1/6}}{\lrp{2^{124}d^6 L^2 \lrp{\theta^3 + \theta^2 + \theta}^2}^{1/36}}}\\
\geq& 3 \frac{1}{c} \log \frac{a}{c}.
\end{align*}
Thus \eqref{e:iw:2} follows immediately from Corollary \ref{c:xlogxbound} with the $a$, $c$, $x$ as defined above.
\end{proof}

\begin{lemma}\label{l:discrete_contraction}
For any $\delta \leq \frac{1}{2L}$ and for $x_k$ with dynamics defined in \eqref{e:sg:1}. If $U_i(x)$ is $m$ strongly convex and has $L$ lipschitz gradients for all $i\in \lrbb{1...S}$, then Assumption \ref{ass:discreteprocesscontracts} holds with $\lambda = m$, i.e. for any two distributions $p$ and $q$, 
\begin{align*}
W_2(\Phi_\delta(p),\Phi_\delta(q)) \leq e^{-m \delta} W_2(p,q)
\end{align*}
\end{lemma}
\begin{proof}[Proof of Lemma \ref{l:discrete_contraction}]
Let $\gamma^*$ be an optimal coupling between $p$ and $q$, i.e.
\begin{align*}
W_2^2(p,q) = \Ep{\gamma^*(x,y)}{\|x-y\|_2^2}
\end{align*}
We define a coupling $\gamma'$ as follows:
\begin{align*}
\gamma'(x,y):= \lrp{F_\eta, F_\eta}_{\#} \gamma^*
\end{align*}
Where $\#$ denotes the push-forward operator. (See \eqref{d:Feta} for the definition of $F_\eta$.)
It is thus true by definition that $\gamma'$ is a valid coupling between $\Phi_\delta(p)$ and $\Phi_\delta(q)$. 

Thus
\begin{align*}
W_2(\Phi_\delta(p),\Phi_\delta(q))
\leq& \Ep{\gamma'(x,y)}{\|x-y\|_2^2}\\
:=& \Ep{\gamma^*(x,y)}{\|F_\eta(x)-F_\eta(y)\|_2^2}\\
=& \Ep{\gamma^*(x,y)}{\|x-\delta \nabla U(x) + \sqrt{2\delta} T_\eta(x) - \lrp{y-\delta \nabla U(y) + \sqrt{2\delta} T_\eta(y)}\|_2^2}\\
=:& \Ep{\gamma^*(x,y)}{\|x- \delta \nabla U_\eta(x) - \lrp{y-\nabla U_\eta(y)}\|_2^2}\\
=& \Ep{\gamma^*(x,y)}{\|x-y - \delta \lrp{\nabla U_\eta(x) - \nabla U_\eta(y)}\|_2^2}\\
\leq& \Ep{\gamma^*(x,y)}{\lrp{1-m\delta/2}\lrn{x-y}_2^2}\\
\leq& e^{-m\delta/4}\Ep{\gamma^*(x,y)}{\lrn{x-y}_2^2}\\
=& e^{-m\delta/4} W_2^2(p,q)
\end{align*}
Where the second inequality follows from our assumption that $U_i(x)$ is m strongly convex and has $L$ lipschitz gradients, and our assumption that $\delta\leq \frac{1}{2L}$, and the third inequality is by the fact that $m\delta/2\leq m/(2L) \leq 1/2$.
\end{proof}
\end{section}
\newpage
\begin{section}{Subgaussian Bounds}

\begin{lemma}
\label{l:p^*hasboundedexponent}
Let $p^*$ be the invariant distribution to \eqref{e:exactsde}. Under the assumptions of Section \ref{ss:mainresult_inhomogeneousnoise}, $p^*$ satisfies
\begin{align*}
\Ep{p^*(x)}{\exp\lrp{\frac{m}{8c_{\sigma}^2}\|x\|_2^2}}\leq 8d
\end{align*}
\end{lemma}
\begin{proof}
Let $p_0$ be an initial distribution for which the above expectation is finite. Let $x_t$ be as defined in \eqref{e:exactsde} (we use $x_t$ instead of $x(t)$ to reduce clutter). For convenience of notation, let $s:=\frac{m}{8c_{\sigma}^2}$.
\begin{align*}
& \ddt \E{\exp\lrp{s\|x_t\|_2^2}}\\
=& \E{\exp\lrp{s\|x_t\|_2^2}\cdot \lrp{\lin{-\nabla U(x_t), 2sx_t} + \lin{2sI + 4s^2 x_t x_t^T, 2\sigma_{x_t} \sigma_{x_t}^T}}}\\
\leq& \E{\exp\lrp{s\|x_t\|_2^2}\cdot \lrp{-2ms\|x_t\|_2^2 + 4ds c_{\sigma}^2 + 8s^2 c_{\sigma}^2 \|x_t\|_2^2}}\\
\leq& \E{\exp\lrp{s\|x_t\|_2^2}\cdot \lrp{-ms\|x_t\|_2^2 + 4ds c_{\sigma}^2}}\\
=& \E{\exp\lrp{s\|x_t\|_2^2}\cdot \lrp{-ms\|x_t\|_2^2 + 4ds c_{\sigma}^2}\cdot \ind{\|x_t\|_2^2 \geq \frac{8c_{\sigma}^2}{m}}} \\
&\quad + \E{\exp\lrp{s\|x_t\|_2^2}\cdot \lrp{-ms\|x_t\|_2^2 + 4ds c_{\sigma}^2}\cdot \ind{\|x_t\|_2^2 < \frac{8c_{\sigma}^2}{m}}}\\
\leq& -4sc_{\sigma}^2 \E{\exp\lrp{s\|x_t\|_2^2}\cdot \ind{\|x_t\|_2^2 \geq \frac{8c_{\sigma}^2}{m}}} + 4dsc_{\sigma}^2 e\\
\leq& -4sc_{\sigma}^2 \E{\exp\lrp{s\|x_t\|_2^2}} + 8dsc_{\sigma}^2 e,
\end{align*}
where the first line is by Ito's lemma, the second line is by Assumption \ref{ass:uissmooth} and Assumption \ref{ass:ximeanandvariance}.2 , the third line is by definition of $s$, the fifth line is again by definition of $s$.

Since $p_t \to p^*$, the above implies that
$$\Ep{p^*}{\exp\lrp{s\|x_t\|_2^2}} < \infty$$

Furthermore, by invariance of $p^*$ under \eqref{e:exactsde}, we have that if $p_0=p^*$ then $\ddt \E{\exp\lrp{s\|x_t\|_2^2}}=0$, so
\begin{align*}
& 0 = \ddt \E{\exp\lrp{s\|x_t\|_2^2}} \leq -4sc_{\sigma}^2 \E{\exp\lrp{s\|x_t\|_2^2}} + 8dsc_{\sigma}^2 e \\
\Rightarrow \quad & 4sc_{\sigma}^2 \E{\exp\lrp{s\|x_t\|_2^2}} \leq  8dsc_{\sigma}^2 e \\
\Rightarrow \quad & \Ep{p^*(x)}{\exp\lrp{\frac{m}{8c_{\sigma}^2}\|x\|_2^2}} \leq  8d
\numberthis \label{e:expexpecttionbound}
\end{align*}
\end{proof}

\begin{lemma}
\label{l:p^*issubgaussian}
Let $p^*$ be the invariant distribution to \eqref{e:exactsde}. Under the assumptions of Section \ref{ss:mainresult_inhomogeneousnoise}, $p^*$ satisfies
\begin{align*}
p^*\lrp{\|x\|_2^2 \geq t} \leq 8 \exp\lrp{-\frac{mt}{8c_{\sigma}^2}},
\end{align*}
where $m$ and $c_{\sigma}$ are as defined in Section \ref{s:definitionsandassumptions}.
\end{lemma}
\begin{proof}[Proof of Lemma~\ref{l:p^*issubgaussian}]
From Lemma~\ref{l:p^*hasboundedexponent},
$$\E{\exp\lrp{\frac{m}{8c_{\sigma}^2}\|x\|_2^2}}\leq 8d$$
By Markov's inequality:
\begin{align*}
\P\lrp{\|x\|_2^2 \geq t}
=& \P\lrp{\exp\lrp{\frac{m}{8c_{\sigma}^2}\|x\|_2^2} \geq \exp\lrp{\frac{m}{8c_{\sigma}^2}t}}\\
\leq& \frac{\E{\exp\lrp{\frac{m}{8c_{\sigma}^2}\|x\|_2^2}}}{\exp\lrp{\frac{m}{8c_{\sigma}^2}t}}\\
\leq& 8d\exp\lrp{-\frac{mt}{8c_{\sigma}^2}}
\end{align*}
\end{proof}

As a Corollary to Lemma~\ref{l:p^*issubgaussian}, we can bound $\E{\|x\|_2^2 \ind{\|x\|_2^2 \geq t}}$ for all $t$:
\begin{corollary}
\label{c:boundedvarianceoutsideradius}
Let $p^*$ be the invariant distribution to \eqref{e:exactsde}. Under the assumptions of Section \ref{ss:mainresult_inhomogeneousnoise}, for any $S\geq \frac{48c_{\sigma}^2}{m} \max\lrbb{\log \lrp{\frac{16c_{\sigma}^2}{m}},1}$, 
$$\Ep{p^*}{\|x\|_2^2 \ind{\|x\|_2^2 \geq S}} \leq 12d\exp\lrp{-\frac{mS}{16c_{\sigma}^2}}$$

\end{corollary}
\begin{proof}[Proof of Corollary \ref{c:boundedvarianceoutsideradius}]
Let $y$ be a real valued random variable that is always positive. We use the equality
\begin{align*}
\E{y} = \int_0^\infty \P(y\geq s) ds
\end{align*}
Let $y := \|x\|_2^2 \cdot \ind {\|x\|_2^2 \geq t}$.
Then
\begin{align*}
\P(y\geq s) = 
\begin{cases}
1 & \text{if}\ s=0 \\
\P(\|x\|_2^2 \geq t) & \text{if}\ s\in (0,t)\\
\P(\|x\|_2^2 \geq s) & \text{if}\ s\geq t
\end{cases}
\end{align*}

Therefore,
\begin{align*}
&\Ep{p^*}{\|x\|_2^2 \cdot \ind {\|x\|_2^2 \geq S}}\\
=& \E{y}\\
=& \int_0^\infty \P(y\geq s) ds\\
=& \int_0^t \P(\|x\|_2^2 \geq s) ds + \int_S^\infty \P(\|x\|_2^2 \geq s) ds\\
\leq& 8dS\exp\lrp{-\frac{mS}{8c_{\sigma}^2}} + \int_S^\infty 8d\exp\lrp{-\frac{ms}{8c_{\sigma}^2}} ds\\
=& 8dS\exp\lrp{-\frac{mS}{8c_{\sigma}^2}} +  \frac{64dc_{\sigma}^2}{m}\exp\lrp{-\frac{mS}{8c_{\sigma}^2}}\\
=& \lrp{8dS+\frac{64dc_{\sigma}^2}{m}}\exp\lrp{-\frac{mS}{8c_{\sigma}^2}}
\numberthis \label{e:l:boundedvarianceoutsideradius:1}\\
\leq& 12dS\exp\lrp{-\frac{mS}{8c_{\sigma}^2}}\\
\leq& 12d\exp\lrp{-\frac{mS}{16c_{\sigma}^2}},
\end{align*}
where the first inequality above uses Lemma~\ref{l:p^*issubgaussian}, the second inequality uses our assumption on $S$, and the third inequality is by our assumption on $S$
combined with Lemma~\ref{l:kthmomentbound}.
\end{proof}
\begin{corollary}\label{c:pdeltaissubgaussian}
Let $p_\delta:=\Phi_\delta(p^*)$, then for all $t\geq 1$ and $\delta \leq \frac{1}{16L}$
\begin{enumerate}
\item $p_\delta\lrp{\|x\|_2^2 \geq t} \leq 8d \exp\lrp{-\frac{mt}{32c_{\sigma}^2}}$
\item $\Ep{p_\delta}{\|x\|_2^2 \ind{\|x\|_2^2 \geq t}} \leq 12d\exp\lrp{-\frac{mt}{64c_{\sigma}^2}}$
\end{enumerate}
\end{corollary}

\begin{proof}[Proof of Corollary \ref{c:pdeltaissubgaussian}]
By Lemma~\ref{l:onestepdiscretizationbounds} and our assumption that $\delta \leq 1/\lrp{16 L}$ and Triangle inequality, we get
\begin{align*}
\lrn{F_\eta^{-1}(x)}_2 
\geq& \|x\|_2 - 2 \delta^{1/2} L^{1/2}\lrp{\|x\|_2 + 1}\\
\geq& 1/2 \|x\|_2 - 1/8
\end{align*}
Thus for $t\geq 1$ and $\delta \leq \frac{1}{4L}$
\begin{align*}
\|x|_2 \geq \sqrt{t} \quad \Rightarrow \quad  \|F_\eta^{-1}(x)|_2 
\geq& 1/2 \|x\|_2 - 1/8\\
\geq& 1/4 \|x\|_2\\
\geq& \sqrt{t}/2
\end{align*}
Thus
\begin{align*}
&p_\delta (\|x\|_2^2 \geq t)\\
\leq& p^*\lrp{\lrn{x}_2^2 \geq t/4}\\
\leq& 8d \exp\lrp{-\frac{mt}{32c_{\sigma}^2}}
\end{align*}
This proves the first claim.

Using the first claim, and an identical proof as Corollary \ref{c:boundedvarianceoutsideradius}, we can prove the second claim.
\end{proof}

\begin{lemma}\label{l:kthmomentbound}
For any $k$, we have the bound
\begin{align*}
\Ep{p^*}{\|x\|_2^{2k}}\leq \max\lrbb{\lrp{2^6 (k-1) \frac{c_{\sigma}^2}{m}\log \lrp{\frac{16(k-1) c_{\sigma}^2}{m}}}^{k-1},128 k d \frac{c_{\sigma}^2}{m}}
\end{align*}
\end{lemma}
\begin{proof}[Proof of Lemma~\ref{l:kthmomentbound}]
Let us define the fixed radius $S:= \max\lrbb{\frac{48 (k-1) c_{\sigma}^2}{m}\log \lrp{\frac{16(k-1) c_{\sigma}^2}{m}},0}$
\begin{align*}
\Ep{p^*}{\|x\|_2^{2k}}
=& \int_0^\infty p^*(\|x\|_2^{2k} \geq t) dt\\
=& k \int_0^\infty p^*(\|x\|_2^{2k} \geq s^{k}) s^{k-1} ds\\
=& k \int_0^\infty p^*(\|x\|_2^2 \geq s) s^{k-1} ds\\
=& k \int_0^S p^*(\|x\|_2^2 \geq s) s^{k-1} ds\\
&\quad + k \int_S^\infty p^*(\|x\|_2^2 \geq s) s^{k-1} ds\\
\leq&  S^k + k \int_S^\infty 8d \exp\lrp{-\frac{ms}{8c_{\sigma}^2}}s^{k-1} ds\\
\leq&  S^k + k \int_S^\infty 8d \exp\lrp{-\frac{ms}{16c_{\sigma}^2}} ds\\
\leq& \max\lrbb{\lrp{2^6 (k-1) \frac{c_{\sigma}^2}{m}\log \lrp{\frac{16(k-1) c_{\sigma}^2}{m}}}^{k-1},128 kd \frac{c_{\sigma}^2}{m}},
\end{align*}
where the first inequality is by Lemma~\ref{l:p^*issubgaussian} and the second inequality is by Lemma~\ref{l:xlogxbound} and our choice of $S$, the third inequality is by some algebra and our choice of $S$.

\end{proof}

\begin{lemma}\label{l:upperboundw2bychisquared}
For any two densities $p$, $q$ over $\Re^d$, and for any radius $R\in \Re^+$, let $c = \max \lrbb{p(\|x\|_2 > R), q(\|x\|_2 > R)}$, then
\begin{align*}
W_2^2(p,q) 
\leq& 4R^2 \int_{B_{{R}}} \lrp{\frac{p(x)}{q(x)}-1}^2 dx  + 32 c^2R^2 + 2cR + 2 \Ep{p}{ \|x\|_2^2 \ind{\|x\|_2 > R}} + 2\Ep{q}{ \|x\|_2^2 \ind{\|x\|_2 > R}}
\end{align*}
\end{lemma}
\begin{proof}[Proof of Lemm \ref{l:upperboundw2bychisquared}]

Let $p$ and $q$ be two distributions. 

Let $a := p(\|x\|_2 > R)$ and $b:= q(\|x\|_2 > R)$, let $c=\max\lrbb{a,b}$. To simplify the proof, assume that $a \leq b$. The proof for the case $b\leq a$ is almost identical and omitted.

For a radius $R$, let 
\begin{align*}
p_R(x) :=& \frac{1}{1-a} \cdot \ind{\|x\|_2\leq R} \cdot p(x)\\
q_R(x) :=& \frac{1}{1-b} \cdot \ind{\|x\|_2\leq R} \cdot q(x)
\end{align*}

I.e. $p$ and $q$ conditioned on $\|x\|_2\leq R$. 

(The proof for when $b\leq a$ is almost identical and is omitted)

We will also define
\begin{align*}
p_R^c(x) 
:=& \frac{1}{b}\lrp{\frac{b-a}{1-a}\cdot \ind{\|x\|_2\leq R} \cdot p(x)} + \frac{1}{b}\lrp{\ind{\|x\|_2> R} \cdot p(x)}\\
q_R^c(x) 
:=& \frac{1}{b} \ind{\|x\|_2^2 > R} \cdot q(x)
\end{align*} 

One can verify that 
\begin{align*}
p(x) =& (1-b)\cdot p_R(x) + b\cdot p_R^c(x)\\
q(x) =& (1-b)\cdot q_R(x) + b\cdot q_R^c(x)
\end{align*}

Suppose that we have a coupling $\gamma_R$ between $p_R$ and $q_R$ (i.e. $\gamma_R$ is a density over $\Re^{2d}$). Then one can verify that $(1-b) \gamma_R+ b \gamma_R^c$ is a valid coupling for $p$ and $q$. Thus
\begin{align*}
W_2^2(p,q) 
\leq& \Ep{(x,y) \sim (1-b) \gamma_R+ b \gamma_R^c}{\|x-y\|_2^2}\\
=& (1-b) \cdot \Ep{(x,y)\sim \gamma_R}{\|x-y\|_2^2} + b \cdot \Ep{(x,y)\sim \gamma_R^c}{\|x-y\|_2^2}\\
\leq& (1-b) \cdot \Ep{(x,y)\sim \gamma_R}{\|x-y\|_2^2} + b \cdot \lrp{2\Ep{p_R^c}{\|x\|_2^2 + 2\Ep{q_R^c}{\|y\|_2^2}}}
\end{align*}

Since the above holds for all valid $\gamma_R$, it holds for the optimal $\gamma_R^*$, thus
\begin{align*}
W_2^2(p,q) \leq (1-b) \cdot W_2^2(p_R,q_R) + 2b \cdot \lrp{\Ep{p_R^c}{\|x\|_2^2 + \Ep{q_R^c}{\|y\|_2^2}}}
\end{align*}

Since $p_R$ and $q_R$ are constrained to the ball of radius $R$, we can upper bound $W_2$ by $TV$:
\begin{align*}
W_2^2(p_R,q_R)
\leq& TV(p_R,q_R)^2 R^2\\
\leq& \KL(p_R,q_R) R^2\\
\leq& \chi^2(p_R,q_R) R^2
\end{align*}

We can upper bound $\chi^2(p_R,q_R)$ as 
\begin{align*}
\chi^2(p_R,q_R)
:=& \int q_R(x) \lrp{\frac{p_R(x)}{q_R(x)}-1}^2 dx\\
=& \int_{B_R} \frac{1}{1-b} q(x) \lrp{\frac{(1-b)}{(1-a)}\frac{p(x)}{q(x)}-1}^2 dx\\
\leq& (1+2c) \int_{B_{\sqrt{R}}} q(x) \lrp{\lrp{1+4c}\frac{p(x)}{q(x)}-1}^2 dx\\
=& (1+2c) \int_{B_R} q(x) \lrp{\lrp{1+4c}\frac{p(x)}{q(x)}-(1+4c) + 4c}^2 dx\\
\leq& 2(1+64c) \int_{B_R} q(x) \lrp{\frac{p(x)}{q(x)}-1}^2 dx + 64 c^2,
\end{align*}
where in the above, $B_R$ is defined as the ball of radius $R$ centered at 0. The two inequalities use Taylor expansion and our assumption that $c\leq \frac{1}{64}$. We also use Young's inequality for the second inequality.

Thus, we get
\begin{align*}
W_2^2(p,q) \leq 4R^2  \int_{B_R} q(x) \lrp{\frac{p(x)}{q(x)}-1}^2 dx  + 64 c^2 R^2 + 2b \lrp{\Ep{p_R^c}{\|x\|_2^2 + \Ep{q_R^c}{\|y\|_2^2}}}
\end{align*}

Using the definition of $p_R^c$:
\begin{align*}
& b \cdot \Ep{p_R^c}{\|x\|_2^2}\\
=& \int \|x\|_2^2 \lrp{\frac{b-a}{1-a}\cdot \ind{\|x\|_2\leq R} \cdot p(x) + \ind{\|x\|_2> R} \cdot p(x)} dx\\
\leq& b \cdot \int \|x\|_2^2 \ind{\|x\|_2 \leq R} \cdot p(x) dx  + \int \|x\|_2^2 \ind{\|x\|_2> R} \cdot p(x) dx\\
\leq& b R^2 + \Ep{p}{ \|x\|_2^2 \ind{\|x\|_2 > R}}
\end{align*}

Using the defintion of $q_R^c$:
\begin{align*}
& b\cdot \Ep{q_R^c}{\|x\|_2^2}\\
=& \Ep{q}{ \|x\|_2^2 \ind{\|x\|_2 > R}}
\end{align*}

Thus we get
\begin{align*}
W_2^2(p,q) 
\leq& 4R^2 \int_{B_R} \lrp{\frac{p(x)}{q(x)}-1}^2 p^*(x) dx  + 32 c^2R^2 + 2bR^2 + 2 \Ep{p}{ \|x\|_2^2 \ind{\|x\|_2 > R}} + 2\Ep{q}{ \|x\|_2^2 \ind{\|x\|_2 > R}}\\
\leq& 4R^2 \int_{B_R} \lrp{\frac{p(x)}{q(x)}-1}^2 p^*(x) dx  + 32 c^2R^2 + 2cR^2 + 2 \Ep{p}{ \|x\|_2^2 \ind{\|x\|_2 > R}} + 2\Ep{q}{ \|x\|_2^2 \ind{\|x\|_2 > R}}
\end{align*}
\end{proof}

\begin{corollary}\label{c:upperboundw2bychisquared}
For any $\epsilon\in[0,1]$, and for $R\geq \sqrt{\max\lrbb{2^{13}\frac{c_{\sigma}^2}{m} \lrp{\log \lrp{\frac{2^{11}c_{\sigma}^2}{m}}},1}}$, and for $\delta \leq \frac{1}{16 L}$
\begin{align*}
W_2^2(p^*,p_\delta) \leq 4R^2 \int_{B_R} \lrp{\frac{p_\delta(x)}{p^*(x)}-1}^2 p^*(x) dx + 84 d \exp\lrp{-\frac{mR^2}{64c_{\sigma}^2}},
\end{align*}
where $p_\delta := \Phi_\delta(p^*)$.
\end{corollary}
\begin{proof}[Proof of Corollary \ref{c:upperboundw2bychisquared}]
By Lemma~\ref{l:p^*issubgaussian}, Corollary \ref{c:boundedvarianceoutsideradius}, and Corollary \ref{c:pdeltaissubgaussian}, and by our assumption that $R^2 \geq 1$ and $\delta \leq \frac{1}{16L}$, we know show that
\begin{enumerate}
\item $p^*\lrp{\|x\|_2 \geq t} \leq 8d \exp\lrp{-\frac{mt^2}{8c_{\sigma}^2}}$
\item $p_\delta\lrp{\|x\|_2 \geq t} \leq 8d \exp\lrp{-\frac{mt^2}{32c_{\sigma}^2}}$
\item $\Ep{p^*}{\|x\|_2^2 \ind{\|x\|_2 \geq R}} \leq 12d\exp\lrp{-\frac{mR^2}{16c_{\sigma}^2}}$
\item $\Ep{p_\delta}{\|x\|_2^2 \ind{\|x\|_2 \geq R}} \leq 12d\exp\lrp{-\frac{mR^2}{64c_{\sigma}^2}}$
\end{enumerate}

Let $p:= p^*$ and $q:=p_\delta$, by the above results, we have
\begin{align*}
\max\lrbb{p(\|x\|_2 \geq R), q(\|x\|_2 \geq R)}\leq 8d \exp\lrp{-\frac{mR^2}{32c_{\sigma}^2}}
\end{align*}
(note that $c_{\sigma}$ is defined in Assumption \ref{ass:ximeanandvariance}.2 and is unrelated to the $c$ we defined in this proof).

Therefore, we apply Lemma~\ref{l:upperboundw2bychisquared} to get
\begin{align*}
W_2^2(p_\delta,p^*) 
\leq& 4R^2 \int_{B_R} \lrp{\frac{p_\delta(x)}{p^*(x)}-1}^2 p^*(x) dx  + 32 R^2\exp\lrp{-\frac{mR^2}{32c_{\sigma}^2}} + 4R^2\exp\lrp{-\frac{mR^2}{64c_{\sigma}^2}}\\
&\quad + 2 \Ep{p_\delta}{ \|x\|_2^2 \ind{\|x\|_2 > R}} + 2\Ep{p^*}{ \|x\|_2^2 \ind{\|x\|_2 > R}}\\
\leq&  4R^2 \int_{B_R} \lrp{\frac{p_\delta(x)}{p^*(x)}-1}^2 p^*(x) dx  + 32 R^2\exp\lrp{-\frac{mR^2}{32c_{\sigma}^2}} + 4R^2\exp\lrp{-\frac{mR^2}{64c_{\sigma}^2}} + 48 d \exp\lrp{-\frac{mR^2}{64c_{\sigma}^2}}\\
\leq& 4R^2 \int_{B_R} \lrp{\frac{p_\delta(x)}{p^*(x)}-1}^2 p^*(x) dx + 36 \exp\lrp{-\frac{mR^2}{128c_{\sigma}^2}} + 48 d \exp\lrp{-\frac{mR^2}{64c_{\sigma}^2}}\\
\leq& 4R^2 \int_{B_R} \lrp{\frac{p_\delta(x)}{p^*(x)}-1}^2 p^*(x) dx + 84 d \exp\lrp{-\frac{mR^2}{64c_{\sigma}^2}},
\end{align*}
where the third inequality is by Lemma~\ref{l:xlogxbound} and our assumption that 
\begin{align*}
R^2 
\geq& \max\lrbb{2^{13}\frac{c_{\sigma}^2}{m} \lrp{\log \lrp{\frac{2^{11}c_{\sigma}^2}{m}}},0}
\end{align*}
\end{proof}
\end{section}
\newpage
\begin{section}{Miscellaneous Lemmas}
\begin{lemma}\label{l:traceupperbound}
For any matrix $A\in \Re^{2d}$, 
\begin{align*}
\tr{A} \leq d\lrn{A}_2
\end{align*}
\end{lemma}
\begin{proof}[Proof of Lemma~\ref{l:traceupperbound}]
For any matrices $A\in \Re^{2d}$ and $B\in\Re^{2d}$, we use the fact that 
\begin{align*}
\lin{A,B}_F:=\tr\lrp{A B^T}
\end{align*}
is an inner product.

Let $A = UDV$ where $U$ and $V$ are two orthonormal matrices and $D$ is a diagonal of positive singular values. Let $\lambda := \max_i D_{i,i}$. It is known that $\lambda = \lrn{A}_2$.

Then
\begin{align*}
\tr\lrp{A}
=& \tr\lrp{U D V}\\
=& \tr\lrp{D V U}\\
=& \lin{D, \lrp{VU}^T}_F\\
\leq& \sqrt{\lin{D,D}_F}\sqrt{\lin{\lrp{VU}^T,\lrp{VU}^T}_F}\\
:=& \sqrt{\tr\lrp{D^2}}\sqrt{\tr\lrp{U^T V^T V U}}\\
\leq& \sqrt{d\lambda^2}\sqrt{d}\\
=& d\lambda \\
=& d\lrn{A}_2,
\end{align*}
where the first inequality is by Cauchy Schawrz, and the second inequality uses the fact that $U^TV^TVU=I$
\end{proof}

\begin{lemma}\label{l:determinanttaylor}
Let $A \in \Re^d \to \Re^d$ be a symmetric matrix such that $\|A\|_2 \leq c$. Let $\epsilon \leq \frac{1}{2cd}$ then 
\begin{align*}
\lrabs{\det \lrp{I + \epsilon A} - \lrp{1 + \epsilon \tr\lrp{A} + \frac{\epsilon^2}{2} \lrp{\tr\lrp{A}^2-\tr\lrp{A^2}} }}\leq \epsilon^3c^3d^3
\end{align*}
\end{lemma}
\begin{proof}[Proof of Lemma~\ref{l:determinanttaylor}]
Let the eigenvalue decomposition of $A$ be $A= U D U^T$, where $U$ is orthogonal, and $D$ is the diagonal matrix of $A$'s eigenvalues. Let $\lambda_i := D_{i,i}$, and let $D$ be chosen such that 
\begin{align*}
\lrabs{\lambda_1} \geq \lrabs{\lambda_2} ... \geq \lrabs{\lambda_d}
\end{align*}
It is known that $\lrabs{\lambda_1} = \lrn{A}_2 \leq c$.

The matrix $I+\epsilon A$ can be written as 
\begin{align*}
U \lrp{I + \epsilon D} U^T
\end{align*}

Since the determinant of products is the product of determinants,
\begin{align*}
\det\lrp{I+ \epsilon A} 
=& \det\lrp{I + \epsilon D} \cdot \lrp{\det\lrp{U} \det\lrp{U}}\\
=& \det\lrp{I + \epsilon D} \cdot det\lrp{UU^T}\\
=& \det\lrp{I + \epsilon D}\\
=& \prod_{i=1}^d \lrp{1 + \epsilon \lambda_i}\\
=& 1 + \epsilon \sum_{i=1}^d \lambda_i + \frac{\epsilon^2}{2} \sum_{i=1}^d\sum_{j\neq i} \lambda_i \lambda_j + ...
\end{align*}
Thus
\begin{align*}
\numberthis \label{e:l:determinantexpansion:1}
& \lrabs{\det\lrp{I + \epsilon A} - \lrp{1 + \epsilon \sum_{i=1}^d \lambda_i + \frac{\epsilon^2}{2} \sum_{i=1}^d\sum_{j\neq i} \lambda_i \lambda_j}}\\
\leq& \sum_{k=3}^d \epsilon^k c^k {d \choose k}\\
\leq& \sum_{k=3}^d \epsilon^k c^k d^k\\
\leq& \epsilon^3 c^3 d^3 ,
\end{align*}
where the last inequality is by the assumption that $\epsilon \leq \frac{1}{2cd}$

It can be verified that 
\begin{enumerate}
\item $\tr\lrp{A} = \sum_{i=1}^d \lambda_i$
\item $\tr\lrp{A^2} = \sum_{i=1}^d \lambda_i^2$
\item $\tr\lrp{A}^2 = \lrp{\sum_{i=1}^d \lambda_i}^2 = \sum_{i=1}^d \lambda_i^2 + \sum_{i=1}^d \sum_{j\neq i} \lambda_i \lambda_j$
\end{enumerate}

Thus, we can rewrite \eqref{e:l:determinantexpansion:1} as 
\begin{align*}
\lrabs{\det \lrp{I + \epsilon A} - \lrp{1 + \epsilon \tr\lrp{A} + \frac{\epsilon^2}{2} \lrp{\tr\lrp{A}^2-\tr\lrp{A^2}}}}\leq  \epsilon^3c^3d^3
\end{align*}
\end{proof}

\begin{lemma}\label{l:xlogxbound}
For any $c> 0$, $x> 3 \max\lrbb{\frac{1}{c} \log \frac{1}{c},0}$, the inequality
\begin{align*}
\frac{1}{c}\log(x) \leq x
\end{align*}
holds.
\end{lemma}
\begin{proof}
We will consider two cases:

\textbf{Case 1}: If $c\geq \frac{1}{e}$, then the inequality 
$$\log(x) \leq c x$$ is true for all $x$.

\textbf{Case 2}: $c \leq \frac{1}{e}$.

In this case, we consider the Lambert W function, defined as the inverse of $f(x) = x e^x$. We will particularly pay attention to $W_{-1}$ which is the lower branch of $W$. (See Wikipedia for a description of $W$ and $W_{-1}$).

We can lower bound $W_{-1}(-c)$ using Theorem 1 from 
\cite{chatzigeorgiou2013bounds}:
\begin{align*}
& \forall u>0,\quad W_{-1} (-e^{-u-1}) > -u-\sqrt{2u} -1\\
\text{equivalently}\quad &\forall c\in (0,1/e),\quad -W_{-1} (-c) < \log\lrp{\frac{1}{c}} + 1 + \sqrt{2\lrp{\log\lrp{\frac{1}{c}}-1}}-1 \\
&\qquad \qquad\qquad\qquad\qquad\ \ = \log\lrp{\frac{1}{c}} + \sqrt{2\lrp{\log\lrp{\frac{1}{c}}-1}}\\
&\qquad \qquad\qquad\qquad\qquad\ \ \leq 3\log \frac{1}{c}
\end{align*}

Thus by our assumption,
\begin{align*}
& x\geq 3\cdot \frac{1}{c}\log\lrp{\frac{1}{c}}\\
\Rightarrow & x\geq \frac{1}{c}\lrp{-W_{-1}(-c)}
\end{align*}

then $W_{-1}(-c)$ is defined, so
\begin{align*}
&x \geq \frac{1}{c}\max\lrbb{-W_{-1} (-c), 1}\\
\Rightarrow & (-cx) e^{-cx} \geq -c\\
\Rightarrow & x e^{-cx} \leq 1\\
\Rightarrow & \log(x) \leq cx
\end{align*}

The first implication is justified as follows:
$W_{-1}^{-1}: [-\frac{1}{\epsilon}, \infty) \to (-\infty, -1)$ is monotonically decreasing. Thus its inverse $W_{-1}^{-1}(y) = ye^y$, defined over the domain $(-\infty, -1)$ is also monotonically decreasing. By our assumption, $-cx \leq -3 \log \frac{1}{c} \leq -3$, thus $-cx \in (-\infty, -1]$, thus applying $W_{-1}^{-1}$ to both sides gives us the first implication.
\end{proof}

\begin{corollary}\label{c:xlogxbound}
For any $a>0$, and for any $c> 0$, $x> 3 \max\lrp{\frac{1}{c}\log \frac{a}{c},0}$, the inequality
\begin{align*}
\frac{1}{c}\log(a\cdot x) \leq x
\end{align*}
holds.
\end{corollary}
\begin{proof}[Proof of Corollary \ref{c:xlogxbound}]
Let $c' := \frac{c}{a}$. Then for any $x' > 3\max\lrbb{\frac{1}{c'} \log \frac{1}{c'},0}$, Lemma~\ref{l:xlogxbound} gives
\begin{align*}
\log \lrp{x'} \leq c'x' = \frac{c}{a} x'
\end{align*}
Thus with a change of variables $x'=ax$, we get that for any $x > \frac{3}{a} \max\lrbb{\frac{1}{c'} \log \frac{1}{c'},0} = 3 \max\lrp{\frac{1}{c}\log \frac{a}{c},0}$,
\begin{align*}
\log (ax) \leq c x
\end{align*}
\end{proof}
\begin{lemma}\label{l:sigmaderivative}
\begin{align*}
\sum_{j=1}^d\dd{x_j} \lrb{\sigma_{x}\sigma_{x}^T}_{i,j}
= \lrb{\Ep{q(\eta)}{G_\eta(x)T_\eta(x) + \tr\lrp{G_\eta(x)}T_\eta(x)}}_i
\end{align*}
\end{lemma}
\begin{proof}[Proof of Lemma~\ref{l:sigmaderivative}]
\begin{align*}
& \sum_{j=1}^d\dd{x_j} \lrb{\sigma_{x}\sigma_{x}^T}_{i,j}\\
=& \sum_{j=1}^d\dd{x_j} \Ep{q(\eta)}{\lrb{T_\eta(x)T_\eta(x)^T}_{i,j}}\\
=& \sum_{j=1}^d\dd{x_j} \Ep{q(\eta)}{\lrb{T_\eta(x)}_{i}\lrb{T_\eta(x)}_j}\\
=& \sum_{j=1}^d \Ep{q(\eta)}{\lrb{G_\eta(x)}_{i,j}\lrb{T_\eta(x)}_j + \lrb{T_\eta(x)}_{i}\lrb{G_\eta(x)}_{j,j}}\\
=& \lrb{\Ep{q(\eta)}{G_\eta(x)T_\eta(x) + \tr\lrp{G_\eta(x)}T_\eta(x)}}_i
\end{align*}
\end{proof}

\begin{lemma}\label{l:fisinvertible}
Let $\delta \leq \frac{1}{8L}$, then the function $F_\eta(y)$ as defined in \eqref{d:Feta} is invertible for all $y$ and for $\eta$ a.s. 
\end{lemma}
\begin{proof}[Proof of Lemma~\ref{l:fisinvertible}]
To prove the invertibility of $F_\eta(x)$, we only need to show that the Jacobian of $F_\eta(x)$ is invertible. The Jacobian of $F_\eta(x)$ is
$$I - \delta \nabla^2 U(x) + \sqrt{2\delta} G_\eta(x) \succ (1-\delta L - \sqrt{2\delta L})I\succ \frac{1}{2} I$$
Where we used Assumption\ref{ass:uissmooth} and Assumption \ref{ass:gisregular}. The existence of $F_{\eta}^{-1}$ thus follows immediately from Inverse Function Theorem.
\end{proof}

\end{section}

\newpage
\begin{section}{Relation to Classical CLT}\label{ss:clt}

\begin{lemma}
\label{l:clt}
Let $\eta_1...\eta_k$ be iid random variables such that $\E{\eta_i}=0$, $\E{\eta_i \eta_i^T} = I$, and $\|\eta_i\|_2$ is a.s. bounded by some constant. Let $\delta_k := \frac{\sqrt{k+1} - \sqrt{k}}{\sqrt{k+1}} \approx \frac{1}{2 (k+1)}$ be a sequence of stepsizes. 
Let $x_{k+1} = x_k - \delta_k x_k + \sqrt{2\delta_k} \eta_k$, and let $p_k$ be the distribution of $x_k$. Let $p^* = N(0,I)$ , then 
\begin{align*}
W_2\lrp{p_k, p^*} = O\lrp{\frac{d^{3/2}}{\sqrt{k}}}
\end{align*}
\end{lemma}

\begin{proof}[Proof of Lemma~\ref{l:clt}]
First, we establish some properties of $\delta_k$.

By performing Taylor expansion of $\sqrt{x+1}$, we get that for $k\geq 2$, 
\begin{align*}
\lrabs{\delta_k - \frac{1}{2(k+1)}} \leq \frac{1}{k^2}
\numberthis \label{e:cl:1}
\end{align*}
We also show that for integers $a \leq b$, 
\begin{align*}
\sum_{i=a}^b \delta_i 
\leq& \sum_{i=1}^{k} \frac{1}{2(i+1)} + \sum_{i=1}^k \frac{1}{k^2}\\
\leq& \frac{1}{2} \log \frac{b}{a} + 1
\end{align*}
A similar argument proves a lower bound, so we have
\begin{align*}
\lrabs{\sum_{i=a}^b \delta_i  - \frac{1}{2}\log \frac{b}{a}} \leq 2
\numberthis \label{e:cl:2}
\end{align*}

Let $K$ be a sufficiently large integer such that 
\begin{align*}
\delta_K = \frac{1}{2(K+1)}\leq \frac{\min\lrbb{m^2,1}}{2^{18} d^2 \lrp{L+1}^3}
\end{align*}

For any $k\geq K$, we can show that
\begin{align*}
&W_2\lrp{p_{k},p^*}\\
\leq& W_2(\Phi_{\delta_{k}}(p_{k-1}), \Phi_{\delta_{k}}(p^*)) + W_2(\Phi_{\delta_{k}}(p^*), p^*)\\
\numberthis \label{e:cl:3}
\leq& e^{-\delta_{k}}W_2(p_{k-1}, p^*) + C \cdot d^{3/2} \cdot k^{-3/2}\\
\leq& ...\\
\leq& \exp\lrp{-\sum_{i=K}^{k} \delta_i}  W_2(p_K, p^*) + \sum_{i=K}^k \lrp{ \exp\lrp{-\sum_{j=i}^k \delta_k}\cdot C\cdot d^{3/2}\cdot i^{-3/2}}\\
\leq& 8 \exp\lrp{\frac{1}{2}\log \frac{k}{K}} W_2(p_K, p^*) + 8 C \cdot d^{3/2}\cdot \sum_{i=1}^k \exp\lrp{\frac{1}{2}\log \frac{k}{i}}\cdot i^{-3/2}\\
\leq& 8\sqrt{\frac{K}{k}} + 8C \cdot d^{3/2} \cdot \sum_{i=1}^k \sqrt{\frac{i}{k}} \cdot i^{-3/2}\\
\leq& 8\sqrt{\frac{K}{k}} + 8C \cdot d^{3/2} \sum_{i=1}^k \frac{1}{\sqrt{k}} \cdot \frac{1}{i}\\
\leq& 8\sqrt{\frac{K}{k}} + 8C \cdot d^{3/2} \frac{1}{\sqrt{k}} \log k\\\leq& C' \cdot d^{3/2} \frac{\log k}{\sqrt{k}},
\end{align*}
where the first inequality is by triangle inequality, the second inequality is by Theorem \ref{t:s_main} (with $k=1$), and our assumption on $\delta_K$ and the fact that $\delta_k \leq \delta_K$, the third and fourth inequalities are by some algebra, the fifth inequality is by \eqref{e:cl:2}, the second last inequality is by harmonic sum.

In applying Theorem \ref{t:s_main} in \eqref{e:cl:3}, we crucially used the fact that $p^*:= N(0,I)$ is the invariant distribution to the SDE
\begin{align*}
d x(t) = - \nabla U(x(t)) dt + \sqrt{2} dB_t
\end{align*}
for $U(x) = \frac{1}{2} \|x\|_2^2$, and the fact that
\begin{align*}
x_{k+1} = x_k - \delta_k \nabla U(x_k) + \sqrt{2\delta_{k+1}} \eta_k
\end{align*}

Note also that the contraction term in \eqref{e:cl:3}, $e^{-\delta_k}$ is tighter than is proven in Theorem \ref{t:s_main}, but this tighter contraction can easily be verified using synchronous coupling as follows: for any two random variables $x_k$ and $y_k$, 
\begin{align*}
\lrn{x_k - \delta_k x_k - \lrp{y_k - \delta_k y_k}}_2^2 \leq \lrp{1 - \delta_k}^2 \lrn{x_k - y_k}_2^2 \leq e^{-2\delta_k} \lrn{x_k - y_k}_2^2
\end{align*}

\end{proof}

\begin{corollary}\label{c:clt}
Let $S_k : = \frac{\sum_{i=1}^k\eta_k}{\sqrt{k}}$. Let $q_k$ be the distribution of $S_k$ and let $p^* = N(0,I)$. Then $W_2\lrp{q_k, p^*} = \tilde{O}\lrp{ \frac{d^{3/2}}{\sqrt{k}}}$
\end{corollary}
\begin{proof}[Proof of Corollary \ref{c:clt}]
Let $\delta_k$, $x_k$ be as defined in Lemma~\ref{l:clt}, with initial $x_0=0$.
It can be verified that
\begin{align*}
S_{k+1} = S_k - \delta_k S_k + \frac{1}{2\lrp{\sqrt{k+1}}} \eta_k
\end{align*}
Thus
\begin{align*}
& \E{\lrn{x_{k+1} - S_{k+1}}_2^2}\\
=& \E{\lrn{ \lrp{1-\delta_k} \lrp{x_k - S_k} + \lrp{\delta_k - \frac{1}{2(k+1)}} \eta_{k+1} }_2^2}\\
=& \E{\lrn{ \lrp{1-\delta_k} \lrp{x_k - S_k}}_2^2} + \E{\lrn{\lrp{\delta_k - \frac{1}{2(k+1)}} \eta_{k+1} }_2^2}\\
\leq& \exp\lrp{-2\delta_k} \E{\lrn{x_k - S_k}_2^2} + \frac{1}{k^2} d,
\end{align*}
where the second last inequality is by the independence of $\eta_k$ and $\E{\eta_k}=0$, and the last ienquality is by \eqref{e:cl:2} and the fact that $\E{\eta \eta^T} = I$.

Applying the above inequality recursively, we get
\begin{align*}
\E{\lrn{x_k - S_k}_2^2} 
\leq& \sum_{i=1}^k \exp\lrp{-\sum_{j=i}^k 2\delta_j}\cdot \frac{d}{i^2} + \exp\lrp{- \sum_{i=1}^k \delta_k} \E{\lrn{x_0 - S_0}_2^2} \\
\leq& 4 \sum_{i=1}^k \exp\lrp{-\log \frac{k}{i}}\cdot \frac{d}{i^2}\\
\leq& 8 \sum_{i=1}^k \frac{d}{k\cdot i}\\
\leq& \frac{16d}{k}\log k,
\end{align*}
where the second inequality is by \eqref{e:cl:2}, and the fact that $x_0=S_0=0$.

Thus
\begin{align*}
W_2 \lrp{x_k, S_k} = \tilde{O} \lrp{\frac{d^{3/2}}{\sqrt{k}}}
\end{align*}

Together with the result from Lemma \ref{l:clt}, we conclude our proof.
\end{proof}

\end{section}

\end{document}